\documentclass[11pt, reqno]{article}
\usepackage{amsmath,amssymb,amsfonts,amsthm}
\usepackage{color}

\usepackage[colorlinks=true,linkcolor=blue,citecolor=blue,urlcolor=blue,pdfborder={0 0 0}]{hyperref}
\usepackage{cleveref}

\usepackage{caption}
\usepackage{thmtools}

\usepackage{mathrsfs}

\crefname{theorem}{Theorem}{Theorems}
\crefname{thm}{Theorem}{Theorems}
\crefname{lemma}{Lemma}{Lemmas}
\crefname{lem}{Lemma}{Lemmas}
\crefname{remark}{Remark}{Remarks}
\crefname{prop}{Proposition}{Propositions}
\crefname{defn}{Definition}{Definitions}
\crefname{corollary}{Corollary}{Corollaries}
\crefname{conjecture}{Conjecture}{Conjectures}
\crefname{question}{Question}{Questions}
\crefname{chapter}{Chapter}{Chapters}
\crefname{section}{Section}{Sections}
\crefname{figure}{Figure}{Figures}
\crefname{example}{Example}{Examples}

\theoremstyle{plain}
\newtheorem{thm}{Theorem}[section]

\newtheorem{lemma}[thm]{Lemma}
\newtheorem{theorem}[thm]{Theorem}

\newtheorem{corollary}[thm]{Corollary}

\newtheorem{prop}[thm]{Proposition}
\newtheorem{conjecture}[thm]{Conjecture}

\newtheorem{question}[thm]{Question}
\theoremstyle{definition}

\newtheorem{example}[thm]{Example}
\newtheorem{problem}[thm]{Problem}

\theoremstyle{remark}
\newtheorem{remark}[thm]{Remark}

\numberwithin{equation}{section}

\renewcommand{\P}{\mathbb P}
\newcommand{\E}{\mathbb E}

\newcommand{\R}{\mathbb R}
\newcommand{\Z}{\mathbb Z}
\newcommand{\N}{\mathbb N}

\newcommand{\cB}{\mathcal B}

\newcommand{\cF}{\mathcal F}
\newcommand{\cG}{\mathcal G}

\newcommand{\cM}{\mathcal M}

\newcommand{\cP}{\mathcal P}

\newcommand{\sA}{\mathscr A}
\newcommand{\sB}{\mathscr B}
\newcommand{\sC}{\mathscr C}
\newcommand{\sD}{\mathscr D}

\newcommand{\sF}{\mathscr F}
\newcommand{\sG}{\mathscr G}

\newcommand{\sO}{\mathscr O}

\newcommand{\sR}{\mathscr R}
\newcommand{\sT}{\mathscr T}

\newcommand{\sW}{\mathscr W}

\newcommand{\bbX}{\mathbb X}
\newcommand{\bbY}{\mathbb Y}

\usepackage[margin=0.95in]{geometry}
\usepackage{tikz}
\usepackage{extarrows}
\usepackage{titling}
\usepackage{commath}
\usepackage{wasysym}
\usepackage{mathtools}
\usepackage{titlesec}
\newcommand{\eps}{\varepsilon}
\usepackage{bbm}
\usepackage{setspace}
\usepackage{enumitem}
\usepackage[hyperpageref]{backref}

\usepackage[textsize=tiny]{todonotes}

\usepackage{cite}

\newcommand{\Aut}{\operatorname{Aut}}

\newcommand{\bP}{\mathbf P}
\newcommand{\bE}{\mathbf E}

\newcommand{\stab}{\operatorname{Stab}}

\newcommand{\opleq}{\preccurlyeq}

\def\P{\mathbb{P}}

\newcommand{\Cov}{{\mathrm{Cov}}}
\newcommand{\Var}{{\mathrm{Var}}}

\newcommand{\sS}{\mathscr{S}}

\newcommand{\wlim}{\mathop{\operatorname{w-lim}}}

\DeclareMathSymbol{\leqslant}{\mathalpha}{AMSa}{"36} 
\DeclareMathSymbol{\geqslant}{\mathalpha}{AMSa}{"3E} 
\DeclareMathSymbol{\eset}{\mathalpha}{AMSb}{"3F}     

\renewcommand{\epsilon}{\varepsilon}

\newcommand{\bn}{\mathbf{n}}
\newcommand{\bm}{\mathbf{m}}
\newcommand{\bQ}{\mathbf{Q}}
\newcommand{\bC}{\mathbf{C}}
\newcommand{\bI}{\mathbf{I}}
\newcommand{\bG}{\mathbf{G}}
\newcommand{\bL}{\mathbf{L}}
\newcommand{\bp}{\mathbf{p}}

\makeatletter
\pgfdeclareshape{slash underlined}
{
  \inheritsavedanchors[from=rectangle] 
  \inheritanchorborder[from=rectangle]
  \inheritanchor[from=rectangle]{north}
  \inheritanchor[from=rectangle]{north west}
  \inheritanchor[from=rectangle]{north east}
  \inheritanchor[from=rectangle]{center}
  \inheritanchor[from=rectangle]{west}
  \inheritanchor[from=rectangle]{east}
  \inheritanchor[from=rectangle]{mid}
  \inheritanchor[from=rectangle]{mid west}
  \inheritanchor[from=rectangle]{mid east}
  \inheritanchor[from=rectangle]{base}
  \inheritanchor[from=rectangle]{base west}
  \inheritanchor[from=rectangle]{base east}
  \inheritanchor[from=rectangle]{south}
  \inheritanchor[from=rectangle]{south west}
  \inheritanchor[from=rectangle]{south east}
  \inheritanchorborder[from=rectangle]
  \foregroundpath{
    \southwest \pgf@xa=\pgf@x \pgf@ya=\pgf@y
    \northeast \pgf@xb=\pgf@x \pgf@yb=\pgf@y
    \pgf@xc=\pgf@xa
    \advance\pgf@xc by .5\pgf@xb
    \pgf@yc=\pgf@ya
    \advance\pgf@xc by -1.3pt
    \advance\pgf@yc by -1.8pt
    \pgfpathmoveto{\pgfqpoint{\pgf@xc}{\pgf@yc}}
    \advance\pgf@xc by  2.6pt
    \advance\pgf@yc by  3.6pt
    \pgfpathlineto{\pgfqpoint{\pgf@xc}{\pgf@yc}}
    \pgfpathmoveto{\pgfqpoint{\pgf@xa}{\pgf@ya}}
    \pgfpathlineto{\pgfqpoint{\pgf@xb}{\pgf@ya}}
 }
}
\tikzset{nomorepostaction/.code=\let\tikz@postactions\pgfutil@empty}

\newcommand\nxleftrightarrow[2][]{%
  \mathrel{\tikz[baseline=-.7ex] \path node[slash underlined,draw,<->,anchor=south] {\(\scriptstyle #2\)} node[anchor=north] {\(\scriptstyle #1\)};}}
\makeatother

\pretitle{\begin{flushleft}\Large}
\posttitle{\par\end{flushleft}\vskip 0.5em
}
\preauthor{\begin{flushleft}}
\postauthor{
\\
\vspace{0.5em}
\footnotesize{Statslab, DPMMS, University of Cambridge.\\ Email: 
\href{mailto:t.hutchcroft@maths.cam.ac.uk}{t.hutchcroft@maths.cam.ac.uk}}
\end{flushleft}}
\predate{\begin{flushleft}}
\postdate{\par\end{flushleft}}
\title{\bf Continuity of the Ising phase transition on nonamenable groups}

\renewenvironment{abstract}
 {\par\noindent\textbf{\abstractname.}\ \ignorespaces}
 {\par\medskip}

\titleformat{\section}
  {\normalfont\large \bf}{\thesection}{0.4em}{}

\author{{\bf Tom Hutchcroft}}

\begin{document}

\date{\small{\today}}

\maketitle

\setstretch{1.15}

\begin{abstract}
We prove rigorously that the ferromagnetic Ising model on any nonamenable Cayley graph undergoes a continuous (second-order) phase transition in the sense that 
there is a unique Gibbs measure at the critical temperature. 
The proof of this theorem is quantitative and also yields power-law bounds on the magnetization at and near criticality. Indeed, we prove more generally that the magnetization $\langle \sigma_o \rangle_{\beta,h}^+$ is a locally H\"older-continuous function of the inverse temperature $\beta$ and external field $h$ throughout the non-negative quadrant $(\beta,h)\in [0,\infty)^2$.
As a second application of the methods we develop, we also prove that the free energy of Bernoulli percolation is twice differentiable at $p_c$ on any transitive nonamenable graph.
\end{abstract}

\tableofcontents






\section{Introduction}
\label{sec:intro}

It has been known since the 19th century that 
the magnetic properties of certain metals such as iron, cobalt, and nickel undergo a qualitative change as they pass through a certain critical temperature, now known as the \emph{Curie temperate}\footnote{The Curie temperature is named after Pierre Curie, who carried out a detailed study of this phase transition in his 1895 doctoral thesis. The fact that such a transition occurs was, however, known well before the work of P.\ Curie, with credit due most appropriately to Pouillet and Faraday; see \cite{jossang1997monsieur} for details. We thank Geoffrey Grimmett for making us aware of this.} of the metal: Below the critical temperature the metal is \emph{ferromagnetic}, meaning that it will remain permanently magnetized after temporary exposure to an external magnetic field, while above the critical temperature the metal is \emph{paramagnetic}, meaning that it will become magnetized in the presence of an external magnetic field but will revert back to being unmagnetized when the external field is removed. The magnetization that remains when the external field is removed is referred to as the \emph{spontaneous magnetization}: it is positive in the ferromagnetic regime and zero in the paramagnetic regime. 

The \emph{Ising model} is a mathematical model that attempts to describe this phase transition. It was introduced in 1920 by Wilhelm Lenz, who suggested the model to his student Ernst Ising as a thesis subject \cite{ising1925beitrag}. 
The model attracted widespread attention following the 1936 work of Peierls \cite{peierls1936ising}, who argued that the model does indeed undergo a phase transition on Euclidean lattices of dimension at least two. 
The Ising model remains today arguably the most famous and intensively studied model in statistical mechanics, with a vast literature devoted to it, and is now used to model many other `cooperative' phenomena in statistical mechanics beyond magnetism. See e.g.\ \cite{MR3752129,1707.00520} for introductions to the Ising model for mathematicians, \cite{MR1446000} for a more physical introduction, and \cite{brush1967history} for a history.


Although the Ising model has traditionally been studied primarily in the setting of Euclidean lattices, there has more recently been substantial interest among both mathematicians and physicists in determining the model's behaviour in other geometric settings, such as hyperbolic spaces. A natural level of generality at which to study the model is that of \emph{(vertex-)transitive graphs}, that is, graphs for which any vertex can be mapped to any other vertex by a symmetry of the graph. The resulting literature is now rather extensive, and includes e.g.\
numerical and non-rigorous studies of critical behaviour \cite{breuckmann2020critical,serina2016free,benedetti2015critical,gendiar2014mean,iharagi2010phase},  rigorous analysis of critical behaviour for some examples \cite{MR1413244,MR1798548,MR1833805,1712.04911,1606.03763}, and analysis of the set of Gibbs measures at low temperature \cite{gandolfo2015manifold,MR1768240,series1990ising,MR1684757}.  Moreover, it is now known that the Ising model has a non-trivial phase transition on any infinite transitive graph that has superlinear volume growth (i.e., is not one-dimensional) \cite{1806.07733}. 

Once non-triviality of the phase transition has been established, it becomes of great interest to understand the model \emph{at the critical temperature}, where it is expected to display various interesting behaviours. Perhaps the most basic  question one can ask about the critical model is whether it belongs to the ferromagnetic or paramagnetic regime. Mathematically, this amounts to asking if the spontaneous magnetization of the model vanishes at the critical temperature, in which case we say that the Ising model undergoes a \emph{continuous phase transition}. 
It is widely believed that the Ising phase transition should be continuous in most cases that it is non-trivial, although this is known to be \emph{false} for certain long-range models in one dimension \cite{MR939480}.

The primary goal of this paper is to prove that the Ising model undergoes a continuous phase transition on any \emph{nonamenable, unimodular transitive graph}. Here, we recall that a graph $G=(V,E)$ is said to be \textbf{nonamenable} if 
$\inf\left\{ |\partial_E W|/\sum_{v\in W} \deg(v) : W \subseteq V \text{ finite} \right\}>0$, where $\partial_E W$ is the set of edges with one endpoint in $W$ and the other not in $W$. \emph{Unimodularity} is a technical condition that holds in most natural examples, including in every Cayley graph of a finitely generated group and every transitive amenable graph \cite{MR1082868}; see \cref{subsec:unimodularity_background} for background. The theorem applies in particular to the Ising model on tessellations of $d$-dimensional hyperbolic space $\mathbb{H}^d$ with $d\geq 2$, for which the result was only previously known under perturbative hypotheses \cite{MR1798548,MR1413244,MR1833805}.

\begin{thm}
\label{thm:main_simple}
Let $G$ be a connected, locally finite, transitive, unimodular, nonamenable graph. Then the phase transition of the Ising model on $G$ is continuous: at the critical temperature the spontaneous magnetization is zero and there is a unique Gibbs measure.
\end{thm}

We will in fact prove more general and quantitative versions of this theorem, \cref{thm:main,thm:main_continuity}, which establish continuity of the model at \emph{all} temperatures, as well as power-law bounds on the magnetization at and near the critical temperature under the same hypotheses.

Let us now briefly outline how our results relate to previous work.
For the hypercubic lattice $\Z^d$, continuity of the phase transition is well understood: The case $d=2$ was settled by Yang in 1952 \cite{yang1952spontaneous}, who built upon the works of Onsager \cite{MR0010315} and Kaufman \cite{kaufman1949crystal},  the case $d\geq 4$ was settled by Aizenman and Fernandez in 1986 \cite{MR857063},  while the case $d=3$ was settled relatively recently by Aizenman, Duminil-Copin, and Sidoravicius in 2015 \cite{MR3306602}. Some aspects of each of these proofs are rather specific to the hypercubic case 
 and do not generalize to other Euclidean lattices, let alone arbitrary transitive graphs. 
While various subsequent works have extended these results to several other Euclidean models \cite{sakai2007lace,MR3898174,MR1896880,MR4026609}, the rigorous understanding of the critical Ising model beyond the Euclidean setting has remained somewhat limited. In our context, the most significant progress was due to Schonmann \cite[Theorem 1.9]{MR1833805} who proved (among many other things) that the Ising model undergoes a continuous phase transition with mean-field critical exponents on certain `highly nonamenable' Cayley graphs. Similar results in the more specific setting of hyperbolic lattices have been obtained by Wu \cite{MR1798548,MR1413244}. The arguments of Schonmann and Wu are of a perturbative nature (that is, they require some parameter associated to the graph to be small), and cannot be used to treat arbitrary nonamenable transitive Cayley graphs. Aside from the classical case of trees, we are only aware of two previous works establishing non-perturbative results in the non-Euclidean context: Our earlier paper \cite{1712.04911}, in which we established continuity of the phase transition for products of regular trees of degree at least three, and the work of Raoufi \cite{1606.03763}, who combined the methods of \cite{Hutchcroft2016944} and \cite{MR3306602} to prove that the Ising model undergoes a continuous phase transition on any \emph{amenable} transitive graph of exponential volume growth. Raoufi's argument relies on amenability in a crucial way and cannot be used to analyze nonamenable examples.

Our techniques draw heavily on the machinery that has been developed to understand \emph{Bernoulli percolation} in the same context \cite{bperc96,1808.08940,BLPS99b}. Indeed, the central technical contribution of our paper is a new method, based on the spectral theory of automorphism-invariant processes, that allows the machinery of \cite{1808.08940} to be applied to certain models that are not positively associated. This new method can be applied  to prove that the \emph{double random current model} does not have any infinite clusters at criticality, from which \cref{thm:main_simple} can be deduced by the methods of Aizenman, Duminil-Copin, and Sidoravicius \cite{MR3306602}. A detailed overview of this new method and how it compares to existing techniques is given in \cref{subsec:intro_overview}.

We hope that this paper will be of value and interest  both to experts on percolation and the Ising model who know relatively little group theory and to experts on group theory who know relatively little about the Ising model; we have included a detailed discussion of background material with the aim of making the paper accessible to both communities.







\subsection{Definitions and statement of results}
\label{subsec:intro_definitions}

Let us now define the Ising model formally. Further background on the Ising model may be found in e.g.\ \cite{MR3752129,1707.00520}; see also \cite[Section 13.1]{Pete} and \cite{MR1757952} for background on aspects specific to the nonamenable case. We will take the approach of \cite{1901.10363}, which allows for a unified treatment of short- and long-range models.
We define a \textbf{weighted graph} $G=(V,E,J)$ to be a countable graph $(V,E)$ together with an assignment of positive \textbf{coupling constants} $\{J_e : e \in E\}$ such that for each vertex $v$ of $G$, the sum of the coupling constants $J_e$ over all edges $e$ adjacent to $v$ is finite. Locally finite graphs can be considered as weighted graphs by setting $J_e \equiv 1$. A graph automorphism of $(V,E)$ is a weighted graph automorphism of $(V,E,J)$ if it preserves the coupling constants, and a weighted graph $G$ is said to be \textbf{transitive} if for every two vertices $x$ and $y$ in $G$ there exists an automorphism of $G$ sending $x$ to $y$. A weighted graph $G=(V,E,J)$ is said to be \textbf{nonamenable} if $\inf\{\sum_{e\in \partial_E K} J_e /  \sum_{e \in E(K)} J_e : K \subseteq V$ finite$\}>0$, where $E(K)$ denotes the set of edges with at least one endpoint in $K$.

Let $G=(V,E,J)$ be a weighted graph with $V$ finite, so that $\sum_{e\in E} J_e <\infty$. For each $\beta\geq 0$ and $h\in \R$ we define the \textbf{Ising measure} $\bI_{\beta,h}=\bI_{G,\beta,h}$ to be the probability measure on $\{-1,1\}^V$ given by
\[
\bI_{G,\beta,h}(\{\sigma\}) \propto \exp\left[ \beta \sum_{e\in E} J_e \sigma_e + \beta \sum_{v\in V} h \sigma_v \right] \qquad \text{ for each $\sigma \in \{-1,1\}^V$}
\]
where for each edge $e\in E$ with endpoints $x$ and $y$ we define $\sigma_e=\sigma_x\sigma_y \in \{-1,1\}$. The parameters $\beta$ and $h$ are known as the \textbf{inverse temperature} and \textbf{external field} respectively. The quantity $\sigma_v\in \{-1,1\}$ is known as the \textbf{spin} at $v$. Thus, the measure favours configurations in which the spins of adjacent vertices are aligned with each other and with the external field.

Now suppose that $G=(V,E,J)$ is an \emph{infinite} weighted graph. 
For each $\beta\geq 0$ and $h\in \R$ we define $\cG_{\beta,h}$ to be the set of \textbf{Gibbs measures} for the Ising model on $G$, that is, the set of probability measures $\mu$ on $\{-1,1\}^V$ satisfying the \textbf{Dobrushin, Lanford, and Ruelle (DLR) equations}
\[
\mu\Bigl(\sigma|_A = \xi|_A \Bigm \vert \sigma|_{V\setminus A} = \xi|_{V\setminus A}\Bigr) = \frac{1}{Z(\xi|_{V\setminus A})} \exp\left[ \beta \sum_{e\in E(A)} J_e \xi_e + \beta  \sum_{v\in A}  h\xi_v \right]
\]
for every $A\subseteq V$ finite and $\xi \in \{-1,1\}^V$, where $E(A)$ denotes the set of edges that have at least one endpoint in $A$ and $Z(\xi|_{V\setminus A})$ is a normalizing constant. Note that Gibbs measures need not in general be invariant under the automorphisms of $G$. A central problem in the study of the Ising model is is to understand the structure of the set of Gibbs measures $\cG_{\beta,h}$, and in particular how this structure depends on $\beta$ and $h$. 
The \textbf{critical inverse temperature} $\beta_c$ is defined by
$\beta_c=\inf\bigl\{\beta\geq 0 : |\cG_{\beta,0}| >1\bigr\}$.

We now introduce three particularly important Gibbs measures for the Ising model: the free, plus, and minus measures.
  Let $G=(V,E,J)$ be an infinite, connected, weighted graph, and let $(V_n)_{n\geq 1}$ be an \textbf{exhaustion} of $V$, that is, an increasing sequence of finite subsets of $V$ with $\bigcup_{n \geq 1} V_n = V$. For each $n\geq 1$, let $G_n$ be the weighted subgraph of $G$ induced by $V_n$. (That is, $G_n$ has vertex set $V_n$, edge set equal to the set of all edges of $G$ with both endpoints in $V_n$, and edge weights inherited from $G$.) For each $\beta \geq 0$ and $h\in \R$, we define the \textbf{free Ising measure} $\bI_{\beta,h}^f=\bI_{G,\beta,h}:=\wlim_{n\to\infty}\bI_{G_n,\beta,h}$ to be the weak limit of the measures $\bI_{G_n,\beta,h}$, so that
\[
\bI_{G,\beta,h}^f (\sigma_a =\kappa_a \text{ for every $a\in A$}) = \lim_{n\to \infty} \bI_{G_n,\beta,h}(\sigma_a =\kappa_a \text{ for every $a\in A$})
\]
for every finite set $A \subseteq V$ and $\kappa\in \{-1,1\}^A$. See e.g.\ \cite[Exercise 3.16]{MR3752129} for a proof that this limit exists, belongs to $\cG_{\beta,h}$, and does not depend on the choice of exhaustion.
%
%
 For each $n\geq 1$ we also define $G_n^*$ to be the finite weighted graph obtained from $G$ by contracting every vertex in $V\setminus V_n$ into a single vertex $\partial_n$ and deleting all self-loops from $\partial_n$ to itself.
  For each $\beta \geq 0$ and $h\in \R$, the \textbf{plus} and \textbf{minus Ising measures} $\bI_{\beta,h}^+=\bI_{G,\beta,h}^+$ 
  and $\bI_{G,\beta,h}^-=\bI_{\beta,h}^-$ on $G$ are
    defined to be the weak limits of the \emph{conditional} measures of the Ising model on $G_n^*$ given that the boundary spin is $+1$ or $-1$ as appropriate.
   In particular, 
\begin{align}
\bI_{G,\beta,h}^+ (\sigma_a =\kappa_a \text{ for every $a\in A$}) &= \lim_{n\to \infty} \bI_{G_n^*,\beta,h}(\sigma_a =\kappa_a \text{ for every $a\in A$} \mid \sigma_{\partial_n}=1)
\end{align}
for every finite set $A \subseteq V$ and $\kappa\in \{-1,1\}^A$. 
The  fact that these weak limits exist and do not depend on the choice of exhaustion is a consequence of the Holley inequality \cite[Theorem 3.17]{MR3752129}.  

The measures $\bI_{\beta,h}^+$ and $\bI_{\beta,h}^-$ are maximal and minimal elements of $\cG_{\beta,h}$ with respect to the partial ordering of stochastic domination: If $\mu$ is any element of $\cG_{\beta,h}$ then $\mu$ stochastically dominates $\bI_{\beta,h}^-$ and is stochastically dominated by $\bI_{\beta,h}^+$ \cite[Lemma 3.23]{MR3752129}. It follows in particular that $|\cG_{\beta,h}|=1$ if and only if $\bI_{\beta,h}^+=\bI_{\beta,h}^-$ if and only if $\bI_{\beta,h}^+=\bI_{\beta,h}^f$. 
Note also that the measure $\bI_{\beta,h}^\#$ is invariant under all automorphisms of $G$ for every $\beta\geq 0$, $h\in \R$, and $\#\in\{f,+,-\}$; this follows from the fact that the limits defining these measures do not depend on the choice of exhaustion \cite[Theorem 3.17 and Exercise 3.16]{MR3752129}.

\begin{remark}
If $G$ is transitive and amenable then $\bI_{\beta,h}^-=\bI_{\beta,h}^f=\bI_{\beta,h}^+$ for every $\beta >0$ and $h \neq 0$ \cite[Section 3.2]{MR1684757}, so that the question of uniqueness of Gibbs measures is only interesting in the case $h=0$. See also \cite[Section 3.7.4]{MR3752129}. This is no longer  true when $G$ is nonamenable.  Indeed, it is a theorem of Jonasson and Steif \cite{MR1684757} that if $G$ is a nonamenable, bounded degree graph then there exists $h_0>0$ such that if $|h| \leq h_0$ then there exists $\beta_c(h)<\infty$ such that $|\cG_{\beta,h}|>1$ for all $\beta > \beta_c(h)$. (The statement they give is different since their definition of the Ising model with external field follows different conventions to  ours.) Intuitively, the difference between these two theorems stems from the fact that boundary effects are always negligible compared with bulk effects in the amenable setting, while the two effects can be of the same order in the nonamenable setting.
\end{remark}

It is traditional to denote expectations taken with respect to the measures $\bI_{\beta,h}^f$,  $\bI_{\beta,h}^+$ and $\bI_{\beta,h}^-$ using the notation $\langle \cdots \rangle_{\beta,h}^f$, $\langle \cdots \rangle_{\beta,h}^+$, and $\langle \cdots \rangle_{\beta,h}^-$ respectively, so that, for example,
\[
\langle \sigma_x \sigma_y \rangle_{\beta,h}^f = \bI_{\beta,h}^f[\sigma_x\sigma_y]
\]  
denotes the expectation of the product of the spins $\sigma_x$ and $\sigma_y$ under the measure $\bI_{\beta,h}^f$ for each $x,y\in V$. We will use both notations throughout the paper as is convenient.

\medskip

Now suppose that $G=(V,E,J)$ is a transitive weighted graph and let $o$ be a fixed root vertex of $G$. For each $\beta\geq 0$, $h\in \R$, and $\#\in\{f,+,-\}$ we define the \textbf{magnetization}
\[
m^\#(\beta,h)=m_G^\#(\beta,h) = \langle \sigma_o \rangle_{\beta,h}^\#.
\]
Note that $m^f(\beta,0)=0$ for every $\beta \geq 0$ by symmetry. For each $\beta \geq 0$, the \textbf{spontaneous magnetization} is defined by
 $m^*(\beta):= m^+(\beta,0)$
 The spontaneous magnetization is a quantitative measure of how much the measures $\bI_{\beta,0}^+$ and $\bI_{\beta,0}^f$ differ, and we have in particular that
\[
m^*(\beta)=0 \iff \bI^f_{\beta,0} = \bI^+_{\beta,0} \iff \bI^+_{\beta,0} = \bI^-_{\beta,0} \iff |\cG_{\beta,0}|=1
\]
for every $\beta \geq 0$.
Thus, we can express the critical inverse temperature $\beta_c$ equivalently as
$\beta_c=\inf\bigl\{\beta \geq 0 : |\cG_{\beta,0}| >1\bigr\}=\inf\bigl\{\beta \geq 0 : m^*(\beta)>0 \bigr\} =\inf\bigl\{\beta \geq 0 : \bI^f_{\beta,0} \neq \bI^+_{\beta,0}\bigr\}$.

\medskip

The following theorem strengthens and generalizes \cref{thm:main_simple}.

\begin{thm}
\label{thm:main}
Let $G=(V,E,J)$ be a connected, nonamenable, transitive, unimodular weighted graph.  Then there exist positive constants $C$ and $\delta$ such that
\[
|m^\#(\beta,h)| \leq C \left( |h| + \max\{ \beta-\beta_c,0\} \right)^\delta
\]
for every $\beta \geq 0$, $h\in \R$, and $\# \in \{f,+,-\}$. In particular, $m^*(\beta_c)=0$ and $|\cG_{\beta_c,0}|=1$.
\end{thm}

In fact, our proof establishes more generally that the spontaneous magnetization is continuous not just at $\beta_c$, but for \emph{all} non-negative $\beta$. The following theorem provides a strong quanitative statement to this effect which implies \cref{thm:main}. Recall that if $\alpha>0$ and $X$ is a locally compact metric space then a function $f:X\to \R$ is said to be \textbf{locally $\alpha$-H\"older continuous} if for every compact set $K \subseteq X$ there exists  $C<\infty$ such that $|f(x)-f(y)|\leq C d(x,y)^\alpha$ for every $x,y\in K$.

\begin{thm}
\label{thm:main_continuity}
Let $G=(V,E,J)$ be a connected, nonamenable, transitive, unimodular weighted graph. Then there exists $\delta>0$ such that  if $F:\{-1,1\}^V \to \R$ is any function depending on at most finitely many vertices then $\langle F(\sigma) \rangle^+_{\beta,h}$ is a 
 locally $\delta$-H\"older continuous function of $(\beta,h)\in [0,\infty)^2$. In particular, the plus Ising measure $\bI_{\beta,h}^+$ is a weakly-continuous function of $(\beta,h)\in [0,\infty)^2$. 
\end{thm}



It is a theorem of Raoufi \cite[Theorem 1 and Corollary 1]{MR4089494} that if $G$ is an \emph{amenable} transitive weighted graph then the plus and free Ising measures $\bI_{\beta,0}^+$ and $\bI_{\beta,0}^f$ are equal and depend continuously on $\beta$ throughout $[0,\beta_c)\cup(\beta_c,\infty)$. In fact, \cite[Theorem 1]{MR4089494} together with the uniqueness of the Gibbs measure in non-zero external field \cite[Section 3.2]{MR1684757} imply more generally that the plus Ising measure $\bI_{\beta,h}^+$ depends continuously on $(\beta,h)$ throughout $[0,\infty)^2 \setminus \{(\beta_c,0)\}$ for every amenable transitive weighted graph. Combining this result with \cref{thm:main_continuity}, we deduce that this conclusion holds for \emph{all} unimodular transitive weighted graphs, and in particular for all Cayley graphs. 

\begin{corollary}
\label{cor:general_continuity}
Let $G=(V,E,J)$ be an infinite, connected, transitive, unimodular weighted graph. Then the plus Ising measure $\bI_{\beta,h}^+$ is a weakly-continuous function of $(\beta,h)$ on $[0,\infty)^2 \setminus \{(\beta_c,0)\}$. 
\end{corollary}

\begin{remark}
We show in \cref{subsec:discontinuity} that there exist nonamenable Cayley graphs for which the \emph{free} Ising measure $\bI^f_{\beta,0}$ is weakly \emph{discontinuous} at some $\beta>\beta_c$. Thus, \cref{thm:main_continuity,cor:general_continuity} cannot be extended to the free Ising measure in general.
\end{remark}

\begin{remark}
\cref{thm:main,thm:main_continuity} have various consequences for the Ising model on transitive nonamenable \emph{planar} graphs with transitive dual, which are discussed in \cref{subsec:planar}. In particular, applying the results of \cite{MR1894115}, we obtain that for any such graph there is a non-trivial interval of $\beta$ for which the free and plus Ising measures are distinct.
\end{remark}





\begin{remark}
The proofs of \cref{thm:main,thm:main_continuity} are effective, and can be used to give explicit estimates on the constants $C$ and $\delta$ depending only on a few important parameters associated to the graph, such as the spectral radius and the value of $\beta_c$.
 It is strongly believed that the Ising model on any transitive nonamenable graph should be governed by the mean-field critical exponents
\[
\langle \sigma_o \rangle_{\beta_c,h}^+ \asymp h^{1/3} \qquad \text{ and } \qquad \langle \sigma_o \rangle_{\beta_c+\eps,0}^+ \asymp \eps^{1/2}.
\]
  See \cite{MR1833805,MR857063} for further discussion. It seems unlikely that our methods can be used to establish this conjecture, and the exponent $\delta$ that we obtain will be very small in general.
 See \cite{1808.08940} for a detailed discussion of related issues in the context of Bernoulli  percolation. 
\end{remark}

\begin{remark}
All the results of this paper should generalize unproblematically to \emph{quasi-transitive} weighted graphs. We restrict attention to the transitive  case to clarify the exposition.
\end{remark}

\subsection{Overview of previous work}
\label{subsec:intro_overview}

In this section we outline previous work on critical statistical mechanics models beyond $\Z^d$, describing in particular the strengths and limitations of existing methods in the context of the Ising model.
  We also take the opportunity to define the random cluster model and briefly explain its connection to the Ising model via the Edwards--Sokal coupling \cite{edwards1988generalization}.

\emph{Bernoulli bond percolation} is by far the most-studied statistical mechanics model outside of the Euclidean context, with an extensive literature stemming from the seminal 1996 work of Benjamini and Schramm \cite{bperc96}; see \cite[Chapters 7 and 8]{LP:book} and references therein for background.
The study of the Ising model and of Bernoulli percolation are closely analogous, and techniques developed to study one model can often (but not always) be applied to study the other. The analogue of \cref{thm:main_simple} for Bernoulli percolation was established in the milestone work of Benjamini, Lyons, Peres, and Schramm \cite{BLPS99b}, who proved that
%
%
critical percolation on any unimodular transitive graph has no infinite clusters almost surely.
 This result was extended to transitive graphs of \emph{exponential growth} 
  by the author \cite{Hutchcroft2016944}.
  More recently, a new  and more quantitative method of proof was developed in \cite{1808.08940}, which allowed us to prove in particular that the the tail of the volume of the cluster of the origin in critical percolation satisfies power-law upper bounds on any unimodular transitive graph of exponential growth. These methods were pushed further to handle certain graphs of \emph{subexponential} volume growth in joint work with Hermon \cite{HermonHutchcroftIntermediate}.

The methods of both \cite{BLPS99b} and \cite{1808.08940} are not particularly specific to Bernoulli percolation and can both  be modified to establish various more general results. The proof of \cite{BLPS99b},  which relies only on soft properties of percolation, can be generalized to show in particular that if $G=(V,E)$ is a unimodular nonamenable transitive graph and  $(\omega_p)_{p\in [0,1]}$ is a family of random subsets of $E$ such that
\begin{enumerate}
  [leftmargin=1.25cm]
    \item[I.i.] The law of 
$(\omega_p)_{p\in [0,1]}$ is invariant and ergodic under the automorphisms of $G$,
  \item[I.ii.] $\omega_p$ is contained in $\omega_{p'}$ for every $p' \geq p$ almost surely,
\item[I.iii.]
 $\omega_p$ is insertion-tolerant for each $p>0$, and
 \item[I.iv.] $\omega_p = \lim_{\eps \downarrow 0} \omega_{p-\eps}$ almost surely for each $p>0$
\end{enumerate}
then the set $\{p\in [0,1] : \omega_p$ has no infinite clusters almost surely$\}$ is a closed interval \cite[Theorem 8.23]{LP:book}. See also \cite{AL07} for extensions of the results of \cite{BLPS99b} to the setting of \emph{unimodular random rooted graphs}. The minimum hypotheses needed to apply the methods of \cite{1808.08940} are a little less clear. One rather general general statement that these methods can be used to prove is as follows: 
Let $G=(V,E)$ be a unimodular transitive graph of exponential growth and suppose that $(\mu_n)_{n \geq 1}$ is a sequence of automorphism-invariant probability measures on $\{0,1\}^E$ converging weakly to some probability measure $\mu$. Suppose further that the following hold:
\begin{enumerate}
  [leftmargin=1.25cm]
\item[II.i.] each of the measures $\mu_n$ is positively associated,
\item[II.ii.] the expected size of the cluster of the origin in $\mu_n$ is finite for each $n\geq 1$, and
\item[II.iii.] each of the measures $\mu_n$ may be written as `percolation in random environment', where the conditional probability of an edge being open given the environment is bounded away from zero by some positive constant that does not depend on $n$.
\end{enumerate}
Then $\mu$ is supported on configurations in which there are no infinite clusters, and the tail of the volume of the cluster of the origin in $\mu$ satisfies a power-law upper bound.  If $G$ is taken to be nonamenable, the hypothesis II.ii above may be replaced with the weaker assumption that each of the measures $\mu_n$ is supported on configurations with no infinite clusters. See \cref{sec:free_energy} for various precise statements. While it may seem that these conditions are much more restrictive than the conditions I.i--iv required to implement the proof of \cite{BLPS99b}, we note that, crucially, we do \emph{not} require a monotone coupling of the measures $(\mu_n)_{n\geq 1}$. 

Both methods can, with work, be applied to the \emph{random cluster model} (a.k.a.\ FK-percolation) with $q\geq 1$ and \emph{free boundary conditions}; This was done for the method of \cite{BLPS99b} by H\"aggstr\"om, Jonasson, and Lyons \cite{MR1913108,MR1894115}.  Let us now quickly recall the definition of this model and its relation to the Ising model, referring the reader to e.g.\ \cite{1707.00520,GrimFKbook} for further background.
Let $G=(V,E,J)$ be a weighted graph with $V$ finite so that $\sum_{e\in E} J_e <\infty$. ($E$ may be finite or infinite.) For each $q>0$ and $\beta,h\geq 0$, we define the \textbf{random cluster measure} $\phi_{q,\beta,h}=\phi_{G,q,\beta,h}$ to be the purely atomic probability measure on $\{0,1\}^E \times \{0,1\}^V$ given by\footnote{Using $2\beta$ instead of $\beta$ in the definition of $\phi_{q,\beta,h}$ is not standard, but makes the relationship between the Ising model and random cluster models simpler to state.}
\[
\phi_{q,\beta,h}(\{\omega\}) \propto q^{k(\omega)} \prod_{e\in E} (e^{2\beta J_e}-1)^{\omega(e)}\prod_{v\in V} (e^{2\beta h}-1)^{\omega(v)}
\]
for each $\omega \in \{0,1\}^E\times\{0,1\}^V$, where $k(w)$ is the number of clusters (i.e., connected components) of the subgraph of $G$ spanned by $\{e:\omega(e)=1\}$ that do not contain a vertex $v$ with $\omega(v)=1$. 
 Note that if $q=1$ and $J_e \equiv 1$ then the measure $\phi_{q,\beta,0}$ is simply the law of Bernoulli bond percolation on $G$ with retention probability $p=(e^{2\beta}-1)/e^{2\beta}=1-e^{-2\beta}$. 

Now suppose that $G=(V,E,J)$ is an infinite connected weighted graph, let $(V_n)_{n\geq 1}$ be an exhaustion of $G$, and let $(G_n)_{n\geq 1}$ and $(G_n^*)_{n\geq 1}$ be defined as in \cref{subsec:intro_definitions}. For each $q\geq 1$ and $\beta,h\geq 0$ we define the \textbf{free} and \textbf{wired} random cluster measures $\phi^f_{q,\beta,h}$ and $\phi^w_{q,\beta,h}$ to be
\begin{align*}
\phi^f_{q,\beta,h}=\phi^f_{G,q,\beta,h}:= \wlim_{n\to\infty} \phi_{G_n,q,\beta,h} \qquad \text{ and } \qquad \phi^w_{q,\beta,h} = \phi^w_{G,q,\beta,h}:= \wlim_{n\to\infty}\phi_{G^*_n,q,\beta,h}.
\end{align*}
Both of these weak limits exist and do not depend on the choice of exhaustion.
  Indeed, it is a consequence of the Holley inequality \cite[Theorem 1.6 and Proposition 1.8]{1707.00520} that if $A\subseteq E \cup V$ is finite and $n_0$ is such that every vertex in $A$ and every edge touching $A$ belongs to $V_{n_0}$ then
\begin{align}
\phi^f_{q,\beta,h}(\omega(x)=1 \text{ for every $x\in A$}) &= \sup_{n\geq n_0} \phi_{G_n,q,\beta,h}(\omega(x)=1 \text{ for every $x\in A$}) \qquad \text{and}\\
\phi^w_{q,\beta,h}(\omega(x)=1 \text{ for every $x\in A$}) &= \inf_{n\geq n_0} \phi_{G_n^*,q,\beta,h}(\omega(x)=1 \text{ for every $x\in A$})
\end{align}
for every $\beta,h \geq 0$, so that $\phi^\#_{q,\beta,h}(\omega(x)=1 \text{ for every $x\in A$})$ depends on $(\beta,h)$ lower semicontinuously  when $\#=f$ and upper semicontinuously when $\# = w$. A further  consequence of the Holley inequality is that $\phi_{q,\beta,h}^w$ stochastically dominates $\phi_{q,\beta,h}^f$ for each fixed $\beta,h\geq 0$ and $q\geq 1$ and that $\phi_{q_1,\beta_1,h_1}^\#$ stochastically dominates $\phi_{q_2,\beta_2,h_2}^\#$ for each $\# \in \{f,w\}$, $q_2 \geq q_1 \geq 1$, $0 \leq \beta_2 \leq \beta_1$, and $0 \leq h_2 \leq h_1$ \cite[Theorem 3.21]{GrimFKbook}. Putting these two facts together, it follows  that $\phi^f_{q,\beta,h}$ is weakly left-continuous in $\beta$ 
and that $\phi^w_{q,\beta,h}$ is weakly right-continuous in $\beta$ 
\cite[Proposition 4.28]{GrimFKbook}.


For each $q \geq 1$ and $\# \in \{f,w\}$ we define the critical inverse temperature
$\beta_c^\#(q)= \sup\{ \beta\geq 0: \phi^\#_{q,\beta,0}$ is supported on configurations with no infinite clusters$\}$.
%
In \cref{sec:free_energy} we extend the analysis of \cite{1808.08940} to the random cluster model, proving the following.

\begin{thm}
\label{thm:free_random_cluster}
Let $G=(V,E)$ be a locally finite, transitive, unimodular graph of exponential growth, and let $q \geq 1$. Then there exist positive constants $\delta$ and $C$ such that
\[
\phi_{q,\beta_c^f(q),0}^f(|K_o| \geq n) \leq C n^{-\delta} \qquad \text{for every $n\geq 1$.}
\]
\end{thm}


The Ising model and the $q=2$ random cluster model, a.k.a.\ the \textbf{FK-Ising model}, are related by the \textbf{Edwards--Sokal coupling} \cite[Section 1.4]{GrimFKbook}. We describe this coupling in the wired/plus case, which is the only case we will use; a similar coupling holds in the free case. Let $G=(V,E,J)$ be an infinite connected weighted graph with $V$ finite, let $\beta,h\geq 0$, and let $\omega$ a random variable with law $\phi_{2,\beta,h}^w$. 
Given $\omega$, we assign a value of $+1$ or $-1$ to each \emph{cluster} of $\omega$ as follows:
\begin{enumerate}
  \item If the cluster is infinite or intersects the set $\{v:\omega(v)=1\}$, we assign it the value $+1$.
\item Otherwise, the cluster is finite and does not intersect the set $\{v:\omega(v)=1\}$, in which case we assign it a value from $\{-1,+1\}$ uniformly at random, where the choices of signs for different clusters are made independently given $\omega$.
\end{enumerate}
Finally, let $\sigma_v$ be equal to the value assigned to the cluster of $v$ for each $v\in V$. Then the resulting random variable $\sigma=(\sigma_v)_{v\in V}$ has law $\bI^+_{\beta,h}$.
 It follows in particular that $m^*(\beta_c)=0$ if and only if the \emph{wired} FK-Ising model has no infinite clusters at $\beta_c$ \cite[Exercise 3.77]{MR3752129}. (In \cref{prop:betafbetaw} we show that $\beta_c=\beta_c^f(2)=\beta_c^w(2)$ on any transitive weighted graph.)


Unfortunately, the methods of \cite{BLPS99b} and \cite{1808.08940} cannot be used to say anything about the \emph{wired} random cluster model at criticality since the measure $\phi_{q,\beta}^w$ need not be weakly left-continuous in $\beta$. 
 This is not merely a technical obstacle, as it is expected that the random cluster model undergoes a \emph{discontinuous} phase transition on nonamenable transitive graphs when $q>2$. 
See \cite{bollobas1996random,laanait1991interfaces,duminil2016discontinuity,ray2019short,MR1894115,MR1833805} and references therein for related results.
 As such, any proof of continuity of the phase transition for the Ising model or FK-Ising model must use some property that distinguishes between the cases $q=2$ and $q>2$, and it is unclear how this could be done within the frameworks of \cite{BLPS99b} or \cite{1808.08940}.

 A similar obstacle was overcome in the \emph{amenable} setting by Aizenman, Duminil-Copin, and Sidoravicius \cite{MR3306602}, who proved in particular that if $G$ is an amenable transitive graph such that the \emph{free} FK-Ising model on $G$ has no infinite clusters at criticality, then the free and wired FK-Ising models coincide at criticality and the spontaneous magnetization of the Ising model vanishes at criticality. 
It will be informative for later developments for us to briefly outline their argument, which was based on the analysis of the \emph{random current\footnote{In this context, a \emph{current} on a graph is an $\N$-valued function on the edge set. Any measure on currents defines a measure on subgraphs by taking an edge $e$ to be open if and only if the current takes a positive value on $e$.} model}. This is an alternative graphical representation of the Ising model that was introduced by Griffiths, Hurst, and Sherman \cite{griffiths1970concavity} and developed extensively by Aizenman \cite{MR678000}; see \cref{subsec:random_currents} for further background and definitions. 
Although the random current model is in many ways a much less well-behaved object than FK-percolation (it is not positively associated or deletion-tolerant, but is insertion tolerant), it has many very interesting features which, roughly speaking, allow it to communicate information between Ising models with different parameters and boundary conditions. 
In particular, it is established in \cite{MR3306602} that the Ising model on a transitive graph undergoes a continuous phase transition if and only if a certain system of two independent random currents on the graph has no infinite clusters at the critical temperature, where one random current is taken with free boundary conditions and the other with wired. On the other hand, the probability that any two vertices are connected in this duplicated system of random currents is bounded by the probability that they are connected in the \emph{free} FK-Ising model. 

To conclude, the authors of \cite{MR3306602} applied the classical theorem of Burton and Keane \cite{burton1989density} (as generalized by Gandolfi, Keane, and Newman \cite{gandolfi1992uniqueness}) to deduce that, \emph{in the amenable case}, the duplicated system of random currents has at most one infinite cluster. Thus, the existence of an infinite cluster is incompatible with connection probabilities between the origin and a distant vertex tending to zero, and the proof of their theorem may easily be concluded. Note that the last part of this argument is very specific to the amenable setting and cannot be used in the nonamenable case where the Burton--Keane theorem does not hold. Note also that this proof is not quantitative, and does not lead to any explicit control of the magnetization near $\beta_c$.
%



\subsection{Overview of the proof and applications to the random cluster model}
 In order to prove \cref{thm:main_simple,thm:main,thm:main_continuity}, we develop a new variation on the methods of \cite{1808.08940} that can be applied to certain models that are \emph{not} positively associated. More specifically, we argue that this hypothesis may be replaced by the assumption that the measures in question have a \emph{spectral gap}; see \cref{subsec:spectral_background} for definitions. The fact that the random cluster measure on a nonamenable transitive graph has such a spectral gap follows from the results of \cite{MR1913108}. This new method also allows us to study the \emph{finite} clusters in supercritical models, leading to the following theorem which is new even in the case of Bernoulli percolation. (Note that the collection of finite clusters in the random cluster model is itself an automorphism-invariant percolation model, but is not positively associated in the supercritical regime.) 

\begin{thm}
\label{thm:finite_clusters}
Let $G=(V,E,J)$ be an infinite transitive nonamenable unimodular weighted graph and let $q \geq 1$. Then there exist positive constants $C$ and $\delta$ such that
\[
\phi_{q,\beta,h}^\#(n \leq |K_o| < \infty) \leq C n^{-\delta}
\]
for every $n\geq 1$, $\beta,h \geq 0$, and $\#\in \{f,w\}$.
\end{thm}


In order to prove our main theorems, we argue that this new method can also be applied to obtain uniform polynomial tail bounds on the finite clusters in a certain variation on the double random current model in which the two currents can have different values of $\beta$ and $h$. See \cref{prop:finite_clusters_mismatched} for a precise statement. To do this we must first bound the spectral radius of the random current model, which we do in \cref{thm:spectralradius}. In \cref{subsec:mismatched,subsec:mainproof} we use our new construction of double random currents with mismatched temperatures to develop a quantitative version of the arguments of \cite{MR3306602}. This lets us deduce the H\"older continuity claimed in \cref{thm:main_continuity}  from the uniform control of finite clusters for the double random current and FK-Ising models provided by \cref{prop:finite_clusters_mismatched} and \cref{thm:finite_clusters}. Once this is done, \cref{thm:main_simple,thm:main} are easily deduced from \cref{thm:main_continuity}. 

Our proof also yields the following analogue of \cref{thm:main_continuity} for the FK-Ising model. 

\begin{thm}
\label{thm:main_continuity_FK}
Let $G=(V,E,J)$ be a connected, nonamenable, transitive, unimodular weighted graph. Then the wired FK-Ising measure $\phi_{2,\beta,h}^w$ is a weakly-continuous function of $(\beta,h)\in [0,\infty)^2$. Moreover, there exists $\delta>0$ such that if $F:\{0,1\}^{E \cup V} \to \R$ is any function depending on at most finitely many edges and vertices of $G$ then $\phi_{2,\beta,h}^w[F(\omega)]$ is a locally $\delta$-H\"older continuous function of $(\beta,h)\in [0,\infty)^2$. 
\end{thm}

We show in \cref{subsec:discontinuity} that the \emph{free} FK-Ising measure $\phi_{2,\beta,0}^f$ can be weakly \emph{discontinuous} in $\beta$ under the same hypotheses.



\subsection{Corollaries for the percolation free energy} We now briefly discuss an interesting application of \cref{thm:finite_clusters} to Bernoulli percolation.
Let $G$ be a connected, locally finite, transitive graph, and write $\bE_p$ for expectations taken with respect to Bernoulli-$p$ bond percolation on $G$. For each $p\in [0,1]$, the \textbf{free energy} (a.k.a.\ open-clusters-per-vertex) $\kappa(p)$ of Bernoulli-$p$ percolation  is defined to be
\[
\kappa(p) := \bE_p \frac{1}{|K_o|}.
\]
  It has historically been a problem of great interest, motivated in part by the non-rigorous work of Sykes and Essam \cite{sykes1964exact}, to determine the location and nature of the singularities of this function. See \cite[Chapter 4]{grimmett2010percolation} for further background. In the nonamenable context, it follows from the results of \cite{HermonHutchcroftSupercritical} that $\kappa(p)$ is an analytic function of $p$ on $[0,p_c)\cup(p_c,1]$. See \cite{georgakopoulos2018analyticity,georgakopoulos2020analyticity} for analogous results in the Euclidean context. On the other hand, the nature of the singularity at $p_c$ (and indeed the question of whether or not there is such a singularity) remains open, even in the nonamenable context:

\begin{conjecture}
\label{conj:free_energy}
Let $G$ be an infinite, connected, locally finite, transitive graph with $p_c<1$. Then the percolation free energy $\kappa(p)$ is twice differentiable but \emph{not} thrice differentiable at $p_c$. 
\end{conjecture}

See \cite[Chapter 4 and Proposition 10.20]{grimmett2010percolation} for an overview of progress on this conjecture. Our results lead to the following partial progress on this conjecture in the nonamenable setting.

\begin{corollary}
Let $G$ be a connected, locally finite, transitive nonamenable graph. Then the percolation free energy $\kappa(p)$ is twice continuously differentiable at $p_c$.
\end{corollary}

\begin{proof}
Aizenman, Kesten, and Newman \cite[Proposition 3.3]{MR901151} proved that if there exists $\eps>0$ such that the truncated $\log^{1+\eps}$-moment $\bE_p \left[\mathbbm{1}(|K_o| <\infty)\log^{1+\eps} |K_o|\right]$ is bounded in a neighbourhood of $p_c$ then the free energy $\kappa(p)$ is twice continuously differentiable in a neighbourhood of $p_c$. The $q=1$ case of \cref{thm:finite_clusters} (see also \cref{thm:finite_percolation})  is easily seen to imply that this criterion holds when $G$ is unimodular. On the other hand,  it follows from \cite[Theorem 1.1]{hutchcroft2020slightly} that this criterion holds whenever critical percolation on $G$ satisfies the \emph{$L^2$ boundedness condition}, which is always the case when $G$ is nonunimodular by the results of \cite{Hutchcroftnonunimodularperc}.
\end{proof}




\section{Background}

\subsection{Unimodularity and the mass-transport principle}
\label{subsec:unimodularity_background}

We now briefly review the notions of unimodularity and the mass-transport principle. See e.g.\ \cite[Chapter 8]{LP:book} and \cite{MR1082868} for further background.

Let $\Gamma$ be a locally compact Hausdorff topological group. Recall that a Radon measure $\nu$ on $\Gamma$ is said to be a \textbf{left Haar measure} if it is non-zero, locally finite, and left-invariant in the sense that $\nu(\gamma A)=\nu(A)$ for every Borel set $A \subseteq \Gamma$ and $\gamma \in \Gamma$. Similarly, $\nu$ is said to be a \textbf{right Haar measure} if it is locally finite and right-invariant in the sense that $\nu(A \gamma) = \nu(A)$ for every Borel set $A \subseteq \Gamma$ and $\gamma \in \Gamma$. Haar's Theorem states that every locally compact Hausdorff topological group has a left Haar measure that is unique up to multiplication by a positive scalar. 
(Similar statements hold for right Haar measures by symmetry.)
 The group $\Gamma$ is said to be \textbf{unimodular} if its left Haar measures are also right Haar measures. Note that every countable discrete group is unimodular since the counting measure is both left- and right-invariant.

Let $G=(V,E,J)$ be a transitive weighted graph, and let $\Aut(G)$ be the group of automorphisms of $G$, which is a locally compact Hausdorff topological group when equipped with the product topology (i.e., the topology of pointwise convergence). The weighted graph $G$ is said to be \textbf{unimodular} if $\Aut(G)$ is unimodular. It follows from \cite[Proposition 8.12]{LP:book} that if $\Gamma$ is a closed, transitive, unimodular subgroup of $\Aut(G)$ then every intermediate closed subgroup $\Gamma \subseteq \Gamma' \subseteq \Aut(G)$ is unimodular also. In particular, if $G$ is a Cayley graph of a finitely generated group $\Gamma$ then $\Gamma$ can also be thought of as a discrete unimodular transitive subgroup of $\Aut(G)$, so that $G$ is unimodular \cite{MR1082868}. 

Note that if $\Gamma \subseteq \Aut(G)$ is a closed subgroup of $\Aut(G)$ then the stabilizer $\stab(v)=\{\gamma \in \Gamma : \gamma v = v\}$ of each vertex $v$ of $G$ is a compact subgroup of $\Gamma$, and in particular has finite Haar measure. If $\Gamma$ is transitive then the subgroups $\{\stab(v) :v\in V\}$ are all conjugate to each other, so that if $\Gamma$ is unimodular and $\nu$ is a Haar measure on $\Gamma$ then $\nu(\stab(u))=\nu(\stab(v))$ for every $u,v\in V$. It follows that if $\Gamma$ is transitive and unimodular then there exists a unique Haar measure $\nu$ such that $\nu(\stab(v))=1$ for every $v\in V$, which we call the \textbf{unit Haar measure}.

Let $G=(V,E,J)$ be a connected transitive weighted graph, let $\Gamma$ be a closed transitive unimodular subgroup of $\Aut(G)$, and let $o$ be an arbitrary root vertex of $G$. The \textbf{mass-transport principle} \cite[Eq.\ 8.4]{LP:book} states that for every function $F:V^2\to [0,\infty]$ that is diagonally-invariant in the sense that $F(\gamma u, \gamma v)=F(u,v)$ for every $u,v \in V$ and $\gamma \in \Gamma$, we have that
\begin{equation}
\label{eq:MTP}
\sum_{v\in V} F(o,v) = \sum_{v\in V} F(v,o).
\end{equation}
As in \cite{1808.08940}, we will also use a version of the mass-transport principle indexed by oriented edges rather than vertices. 
 Write $E^\rightarrow$ for the set of oriented edges of $G$, where an oriented edge $e$ is oriented from its tail $e^-$ to its head $e^+$ and has reversal $e^\leftarrow$.
Let $\eta$ be chosen at random from the set of oriented edges of $G$ emanating from $o$ with probability proportional to $J_e$, so that $\eta$ has the law of the first edge crossed by a random walk started at $o$. Then for every function $F:E^\rightarrow \times E^\rightarrow \to [0,\infty]$ that is diagonally-invariant in the sense that $F(\gamma e_1, \gamma e_2)=F(e_1,e_2)$ for every $e_1,e_2\in E^\rightarrow$ and $\gamma \in \Gamma$, we have that
\begin{equation}
\E\sum_{e\in E^\rightarrow} J_e F(\eta,e) = \E\sum_{e\in E^\rightarrow} J_e F(e,\eta),
\label{eq:edgeMTP}
\end{equation}
where the expectation is taken over the random oriented edge $\eta$. 
This
 equality follows by applying \eqref{eq:MTP} 
 to the function 
  $\tilde F(u,v) = \sum_{e_1^-=u} \sum_{e_2^-=v} J_{e_1} J_{e_2} F(e_1,e_2)$. Moreover, the equality \eqref{eq:edgeMTP} also holds for \emph{signed} diagonally-invariant functions $F:E^\rightarrow\times E^\rightarrow \to\R$ satisfying the absolute integrability condition
\begin{equation}
\label{eq:integrability}
\E\sum_{e\in E^\rightarrow}J_e |F(\eta,e)|<\infty.
\end{equation}
This follows by applying \eqref{eq:edgeMTP} separately to the positive and negative parts of $F$, which are defined by
 $F^+(e_1,e_2)=0\vee F(e_1,e_2)$ and $F^-(e_1,e_2)=  0\vee (-F(e_1,e_2))$.

\subsection{The spectral theory of automorphism-invariant processes}
\label{subsec:spectral_background}

In this section we review some notions from the spectral theory of group-invariant processes that will be used in the proofs of our main theorems. 
Everything we discuss in this section is likely to be known to some experts, but we have given a fairly detailed and self-contained account since we expect it to be unfamiliar to many of our readers and the required material is spread over several papers and not always written in the form that we wish to apply it. Good resources for further background on this material from a probabilistic perspective include \cite{MR3729654,MR1082868,MR2825538}.

\medskip


 We begin by quickly recalling the definition of the spectral radius of a weighted graph. Let $G=(V,E,J)$ be a connected, transitive weighted graph.  The random walk on $G$ is defined to be the reversible Markov chain on $V$ which, at each time step, chooses a random oriented edge $e$ emanating from its current location with probability proportional to $J_e$ and then crosses this edge, independently of everything it has done previously. 
We write $\bE_v$ for the law of the random walk $X=(X_n)_{n\geq 0}$ on $G$ started at $v$. The \textbf{Markov operator} $P:L^2(V)\to L^2(V)$ is defined by
\[
Pf(v)=\bE_v f(X_1)
\]
for each $v\in V$. The operator $P$ is clearly self-adjoint, while Jensen's inequality implies that $P$ is bounded with operator norm $\|P\|\leq 1$. It is a well-known theorem of Kesten \cite{Kesten1959b} (see also \cite[Chapter 6.2]{LP:book}) that the strict inequality $\|P\|<1$ holds if and only if $G$ is nonamenable, and moreover that
\[
\|P\| = \lim_{n\to\infty} p_{2n}(v,v)^{1/2n} =\limsup_{n\to\infty} p_n(u,v)^{1/n}
\]
for every $u,v\in V$. The norm $\|P\|$ is known as the \textbf{spectral radius} of $G$ and is also denoted $\rho(G)$.

\medskip

\textbf{Random walks on graphs vs.\ random walk on groups.} When $G$ is a Cayley graph of a finitely generated group $\Gamma$, we can always think of the random walk on $G$ as a random walk on the group and hence as a random walk on (a subgroup of) $\Aut(G)$. 
This duality between random walks on graphs and on their automorphism groups comes with some subtleties, however.
 Let $(Z_n)_{n\geq 0}$ be a sequence of i.i.d.\ $\Gamma$-valued random variables such that $Z_n$ is distributed as the first step of a random walk on $G$ started from the identity and define $X=(X_n)_{n\geq0}$ by $X_0=\mathrm{id}$ and $X_n=X_{n-1}Z_n$ for every $n\geq 1$ then $X$ is distributed as a random walk on $G$ started at the identity. 
Moreover, if we let $\Gamma$ act on $L^2(V) \cong L^2(\Gamma)$ by $\gamma f(v) = f(\gamma^{-1} v)$ then we can define a bounded self-adjoint operator $\hat P$ on $L^2(V)$ by
\[
\hat P f(v) = \E \left[X_1 f(v)\right] = \E\left[ f(X_1^{-1} v)\right].
\]
It follows by induction on $n \geq 1$ that
\[
\hat P^n f(v) = \E\left[ Z_1 \hat P^{n-1} f(v) \right] = \E\left[  Z_1 \E \left[ Z_2 \cdots Z_n f(v)\right]\right] = \E\left[ X_n f(v)\right]
\]
for every $v\in V$, $n\geq 1$, and $f\in L^2(V)$, where we used that $X_{n-1}=Z_1\cdots Z_{n-1}$ and $Z_2 \cdots Z_n$ have the same distribution in the central equality.
Note however that $X_n^{-1}v$ does \emph{not} in general have the same distribution as the $n$th step of a random walk on $G$ started at $v$, so that the operators $P$ and $\hat P$ are not generally the same. Nevertheless, the \emph{spectral radii} $\|P\|$ and $\|\hat P\|$ are always the same, for the simple reason that $\hat P$ is the Markov operator for the random walk on the \emph{left} Cayley graph of $\Gamma$ with respect to the same generating set as $G$, which is isomorphic to $G$.
More concretely, if we consider the isometric involution $\operatorname{Inv}$ of $L^2(V) \cong L^2(\Gamma)$ defined by
\begin{align*}
\operatorname{Inv} f(\gamma) = f(\gamma ^{-1}) \qquad \text{ for all $f\in L^2(\Gamma)$ and $\gamma \in \Gamma$}
\end{align*}
then $\hat P = \operatorname{Inv} P \operatorname{Inv}$ and $P =  \operatorname{Inv} \hat P  \operatorname{Inv}$, so that all the spectral properties of the two operators $P$ and $\hat P$ are the same.


\medskip

It will be useful to have a similar duality in place for general transitive graphs. Much of this duality was developed by Soardi and Woess \cite{MR1082868}, although we will follow some slightly different conventions.
Let $G=(V,E,J)$ be a connected transitive weighted graph, let $o$ be a fixed root vertex of $G$, and let $\Gamma \subseteq \Aut(G)$ be a unimodular closed transitive group of automorphisms.
As above, $\Gamma$ acts on $L^2(V)$ by
\[
\gamma f(v) = f(\gamma^{-1}v) \text{ for every $\gamma\in \Gamma$, $f\in L^2(V)$, and $v\in V$.}
\]
 Let $X=(X_n)_{n\geq 0}$ be a random walk on $G$ started at $o$. Conditional on $X$, let $\hat X=(\hat X_{n})_{n\geq 0}$ be drawn independently at random from the normalized  Haar measures on the compact sets of automorphisms $\{\gamma \in \Gamma : \gamma o = X_n\}$. Meanwhile, let $\nu$ be the law of $\hat X_1$, let $(Z_n)_{n\geq 1}$ be i.i.d.\ random variables each with law $\nu$, and let $\hat Y = (\hat Y_n)_{n\geq 0}$ be the random walk on $\Gamma$ defined by $\hat Y_0=\mathrm{id}$ and $\hat Y_n = \hat Y_{n-1} Z_n = Z_1 \cdots Z_n$ for every $n\geq 1$. The following lemma is an easy consequence of the fact that the Haar measure on the unimodular group $\Gamma$ is both left- and right-invariant.
\begin{lemma}
\label{lem:groupsvsgraphs}
Let $G=(V,E,J)$ be a connected transitive weighted graph, let $o$ be a fixed root vertex of $G$, and let $\Gamma \subseteq \Aut(G)$ be a unimodular closed transitive group of automorphisms. Then the two processes $\hat X$ and $\hat Y$ we have just defined have the same distribution.
\end{lemma}



%
%
%
%
%
Thus, we may think of $\hat X$ as a random walk on $\Gamma \subseteq \Aut(G)$. 
We define the associated Markov operator $\hat P : L^2(\Gamma) \to L^2(\Gamma)$ by
\[
\hat P f(\gamma) = \E \left[ Z_1 f(\gamma)\right] = \E\left[f( Z_1^{-1}\gamma)\right]
\]
for every $f\in L^2(\Gamma)$ and a.e.\ $\gamma\in \Gamma$. Similarly to above, $\hat P$ is bounded, self-adjoint, and satisfies
\[
\hat P^n f(\gamma) = \E \left[Z_1 \hat P_{n-1} f(\gamma)\right] = \E\left[  Z_1 \E \left[ Z_2 \cdots Z_n f(\gamma)\right]\right] = \E\left[ \hat X_n f(\gamma)\right]
\]
for every $n\geq 0$, $f\in L^2(V)$, and a.e.\ $\gamma \in \Gamma$. The Markov operator $\hat P$ can be related to the usual Markov operator $P$ as follows: Let $\lambda$ be the unit Haar measure on $\Gamma$ and consider the three operators
\begin{align*}
\operatorname{Proj} : L^2(\Gamma) &\to L^2(V) \qquad &\operatorname{Proj} f(v) &= \int_\Gamma f (\gamma) \mathbbm{1}(\gamma o = v) \dif \lambda (\gamma),\\
\operatorname{Inj} : L^2(V) &\to L^2(\Gamma) \qquad & \operatorname{Inj} f(\gamma) &= f(\gamma o),
 &\text{ and }\\
\operatorname{Inv} : L^2(\Gamma) &\to L^2(\Gamma) \qquad & \operatorname{Inv} f(\gamma) &= f(\gamma^{-1}),
\end{align*}
each of which is easily seen to be bounded with norm $1$. The self-adjoint operator $\operatorname{Inv}$ is an isometric involution of $L^2(\Gamma)$, while $\operatorname{Proj}$ and $\operatorname{Inj}$ are adjoints of each other. The two Markov operators $P$ and $\hat P$ satisfy the congruence-type relation
\begin{equation}
\label{eq:congruence}
P = \operatorname{Proj} \operatorname{Inv} \hat P \operatorname{Inv} \operatorname{Inj} \qquad \text{ and } \qquad \hat P = \operatorname{Inv} \operatorname{Inj}  P  \operatorname{Proj} \operatorname{Inv},
\end{equation}
which implies in particular that $\|P\|=\|\hat P\|=\rho(G)$. The equation \eqref{eq:congruence}, which is easily verified directly, is essentially equivalent to \cite[Proposition 1]{MR1082868}.

\medskip

\textbf{Spectral radii of automorphism-invariant processes.} We now define the spectral radius of an \emph{automorphism-invariant stochastic process} on $G$. Let $G=(V,E,J)$ be a connected transitive weighted graph, let $o$ be a fixed root vertex of $G$, and let $\Gamma \subseteq \Aut(G)$ be a unimodular closed transitive group of automorphisms. Let $X=(X_n)_{n\geq 0}$ be a random walk started at $o$ on $G$ and let $\hat X= (\hat X_n)_{n\geq 0}$ be the associated random walk on $\Gamma$ as above. 
Let $\mathbb{X}_V$ and $\bbX_E$ be Polish spaces, which we will usually take to be either $\{\emptyset\}$, $\{0,1\}$, $[0,1]$, or $\N_0 = \{0,1,2,\ldots\}$, and let $\Omega$ be the product space\footnote{We have chosen to restrict to product spaces of this form to help make the resulting theory more intuitive to probabilists; one could just as well consider arbitrary actions of $\Gamma$ on probability spaces by measure preserving transformations (a.k.a.\ pmp actions of $\Gamma$), as is standard in other parts of the literature.} $\Omega=\bbX_V^V \times \bbX^E_E$. The group $\Gamma$ acts on $\Omega$ by
\[
\gamma \omega(x) = \omega(\gamma^{-1} x) \qquad \text{ for each $\gamma \in \Gamma$, $\omega \in \Omega$, and $x\in V \cup E$,}
\]
and a probability measure $\mu$ on $\Omega$ is said to be \textbf{$\Gamma$-invariant} if $\mu(A)=\mu(\gamma^{-1}A)$ for every Borel set $A \subseteq \Omega$ and $\gamma \in \Gamma$. Given an automorphism-invariant probability measure $\mu$ on $\Omega$, we define the Markov operator $\hat P_\mu$ on $L^2(\Omega,\mu)$ by
\[
\hat P_\mu f(\omega) = \bE \left[\hat X_1 f( \omega) \right] = \bE \left[f(\hat X_1^{-1} \omega) \right] \qquad \text{ for every $f\in L^2(\Omega,\mu)$ and $\omega\in \Omega$,}
\]
which is bounded and self-adjoint with norm $\|\hat P_\mu\| = 1$.
Let $L^2_0(\Omega,\mu)=\{f \in L^2(\Omega,\mu): \mu(f)=0\}$. The Markov operator $\hat P_\mu$ fixes $L_0^2(\Omega,\mu)$, and can therefore also be seen as a bounded self-adjoint operator on $L^2_0(\Omega,\mu)$. We define
\[
\rho(\mu) = \rho(\mu,\Gamma)= \sup\left\{ \frac{\|\hat P_{\mu} f\|_2}{\|f\|_2} : f\in L^2_0(\Omega,\mu) \setminus \{0\} \right\} = 
\sup\left\{ \frac{|(\hat P_{\mu} f,f)|}{|(f,f)|} : f\in L^2_0(\Omega,\mu) \setminus \{0\} \right\} 
\]
to be the spectral radius of the Markov operator $\hat P_{\mu}$ on $L^2_0(\Omega,\mu)$. We say that $\mu$ has a \textbf{spectral gap} if $\rho(\mu)<1$.

\medskip

The spectral radius may also be expressed probabilistically as follows. Let $\bbX_V$ and $\bbX_E$ be Polish spaces and let $\varphi=(\varphi_x)_{x\in V \cup E}$ be a random variable taking values in $\Omega= \bbX_V^V \times \bbX_E^E$ whose law $\mu$ is $\Gamma$-invariant and let the processes $X$ and $\hat X$ and as be defined as  above and independent of $\varphi$. The definition of the spectral radius may be rewritten probabilistically as
\begin{align}
\label{eq:covariance}
\rho(\mu) = \rho(\mu,\Gamma) &= 
\sup\left\{ \frac{\bigl|\Cov\bigl( F(\varphi),F(\hat X_{1}^{-1}\varphi)\bigr)\bigr|}{\Var(F(\varphi))}: F\in \R^\Omega, 0<\Var(F(\varphi)) <\infty
\right\}
\end{align}
where we write $\Var$ and $\Cov$ for variances and covariances taken with respect to the joint law of the random variables $\varphi$ and $\hat X$.

\medskip

 Note that we will typically be interested in random fields that are indexed only by the edge set or by the vertex set. Such fields are easily included within this formalism by setting $\bbX_V=\{\emptyset\}$ or $\bbX_E=\{\emptyset\}$ as appropriate and setting the random field to be constantly equal to $\emptyset$ over the irrelevant indices, and we will apply the results and terminology of this section to such fields without further comment in the remainder of the paper.

\begin{example}
\label{example:FiniteClusters}
Let $G=(V,E,J)$ be a connected, transitive weighted graph, let $\Gamma$ be a closed transitive unimodular subgroup of $\Aut(G)$ and let $o$ be a fixed root vertex of $G$. Let $\mu$ be an automorphism-invariant probability measure on $\{0,1\}^E$ and let $\omega \in \{0,1\}^E$ be a random variable with law $\mu$. Let $X=(X_n)_{n\geq 0}$ be a random walk on $G$ started at $o$ and let $\hat X$ be the associated random walk on $\Gamma$, where we take $X$ and $\hat X$ to be independent of $\omega$. Letting $K_v$ be the cluster of $v$ in $\omega$ for each vertex $v$ of $G$, we have that $\hat X_k\mathbbm{1}(n\leq |K_o| < \infty)= \mathbbm{1}(n\leq |K_{X_k}|<\infty)$ for every $k,n\geq 0$, so that
\begin{multline}
|\P( n \leq |K_{o}|,|K_{X_{k}}| < \infty)-\P( n \leq |K_{o}|< \infty)^2| \\\leq \rho(\mu)^{k} |\P( n \leq |K_{o}| < \infty)-\P( n \leq |K_{o}|< \infty)^2| \leq \rho(\mu)^k
\end{multline}
for every $n,k\geq 0$ by definition of $\rho(\mu)$. This inequality will play a central role in the proofs of our main theorems.
\end{example}

We next discuss some useful properties of the spectral radius that we will use in the proofs of our main theorems.

\medskip

\textbf{The limit formula.} 
We first recall some standard facts about self-adjoint operators on Hilbert spaces that will help us to compute spectral radii in examples. Let $T$ be a bounded self-adjoint operator on a Hilbert space $H$. Cauchy-Schwarz gives that
\[
\| T^{n+1} x\|^4 = (T^nx,T^{n+2}x)^2 \leq \|T^n x\|^2\|T^{n+2}x\|^2
\] 
for every $x \in H$ and $n\geq 0$, which implies that if $x\in H$ is such that $Tx \neq 0$ then $T^n x \neq 0$ for all $n\geq 0$ and that $\|T^{n+1}x\|/\|T^n x\|$ is an increasing function of $n\geq 0$. This is easily seen to imply that $\lim_{k\to\infty} \|T^{k} x\|^{1/k}$ exists for every $x\in H$ and that
\begin{equation}
\label{eq:norm_limit_bound}
\frac{\|T x\|}{\|x\|} \leq 
\lim_{k\to\infty} \|T^{k} x\|^{1/k} \leq \|T\|
\end{equation}
for every $x\in H\setminus \{0\}$. Moreover, we have by the triangle inequality that
\[\lim_{k\to\infty} \Bigl\|T^k\sum_{i=1}^m a_i x_i\Bigr\|^{1/k} 
\leq \lim_{k\to\infty} \left(\sum_{i=1}^m |a_i| \|T^k  x_i\| \right)^{1/k} =
\max \Bigl\{ \lim_{k\to\infty} \|T^kx_i\|^{1/k} : 1\leq i \leq m, a_i \neq 0\Bigr\}\]
for every $x_1,\ldots,x_m \in H$ and $a_1,\ldots a_m \in \R$. It follows that 
 if $A$ is a subset of $H$ with dense linear span $S(A)$ then
\begin{align}
\|T\|
=\sup\Bigl\{ \lim_{k\to\infty} \|T^kx\|^{1/k} : x\in S(A)\Bigr\}
=\sup\Bigl\{ \lim_{k\to\infty} \|T^kx\|^{1/k} : x\in A\Bigr\}.
\label{eq:norm_limit_subset}
\end{align}
Translating this into probabilistic notation, the formula \eqref{eq:norm_limit_subset} yields in the context of \eqref{eq:covariance} that
\begin{align}
\rho(\mu) &= 
 \sup\left\{ \lim_{k\to\infty} \Cov\left( F(\varphi),F(\hat X_{2k}^{-1}\varphi)\right)^{1/2k} : F\in A
\right\}
\label{eq:rho_limit_covariance}
\end{align}
for every set of functions $A \subseteq L^2(\Omega,\mu)$ that has dense linear span in $L^2(\Omega,\mu)$. 

\medskip

\textbf{Spectral radii of i.i.d.\ processes.}
Let $G=(V,E,J)$ be a connected transitive weighted graph and let $\Gamma \subseteq \Aut(G)$ be a closed group of automorphisms. A 
\textbf{Bernoulli process} on $G$ is a family of independent random variables $(\varphi_x)_{x\in E \cup V}$ taking values in a Polish space of the form $\bbX_V^V \times \bbX_E^E$ such that $\varphi_x$ and $\varphi_{\gamma x}$ have the same distribution for every $x\in V \cup E$ and $\gamma \in \Gamma$. The law of a Bernoulli process is called a \textbf{Bernoulli measure}. We say that a Bernoulli measure is \textbf{non-trivial} if it is not concentrated on a single point.
%
 The following theorem is  folklore. 

\begin{thm}
\label{thm:Bernoulli_radius}
Let $G=(V,E,J)$ be an infinite, connected, transitive weighted graph and let $\Gamma \subseteq \Aut(G)$ be a unimodular, closed, transitive group of automorphisms. If $\mu$ is a non-trivial Bernoulli measure on $G$ then $\rho(\mu)=\rho(G)$.
\end{thm}

See \cite[Theorem 2.1 and Corollary 2.2]{MR2825538} for stronger results in the case that $G$ is a Cayley graph.

\begin{proof}[Proof of \cref{thm:Bernoulli_radius}]
Let $\varphi$ be a random variable with law $\mu$, let $X$ be a random walk started from the origin on $G$, and let $\hat X$ be the associated random walk on $\Gamma$, where we take $\varphi$ and $\hat X$ to be independent.
Observe that functions of the form $\mathbbm{1}(\varphi|_A \in \mathscr{A})$ where $A \subseteq V \cup E$ is finite and $\sA \subseteq \bbX^{A \cap V}_V \times \bbX_E^{A \cap E}$ is Borel have dense linear span in $L^2(\Omega,\mu)$. Fix one such pair of sets $A$ and $\sA$ and let $V(A)$ be the set of vertices that either belong to $A$ or are the endpoint of an edge belonging to $A$. We have by independence that
\begin{align*}
\lim_{k\to\infty} \Cov\left(\mathbbm{1}(\varphi|_A \in \mathscr{A}),\hat X_{2k} \mathbbm{1}(\varphi|_A \in \mathscr{A})\right)^{1/2k} &\leq \lim_{k\to\infty}\P(\hat X_{2k}^{-1} A \cap A \neq \emptyset)^{1/2k} \\
&\leq \lim_{k\to\infty} (\hat P^{2k} \mathbbm{1}_{V(A)}, \mathbbm{1}_{V(A)})^{1/2k} \leq \|\hat P\|=\|P\|=\rho(G),
\end{align*}
 and it follows from \eqref{eq:rho_limit_covariance} that $\rho(\mu)\leq \rho(G)$.
 The matching lower bound (which we will not use) follows by similar reasoning, using the assumption that $\mu$ is non-trivial, and is left as an exercise to the reader.
\end{proof}



\textbf{Monotonicity under factors.}
 Let $G=(V,E,J)$ be a transitive connected weighted graph and let $\Gamma$ be a closed unimodular transitive subgroup of $\Aut(G)$. Let $\bbX_V$, $\bbX_E$, $\bbY_V$, and $\bbY_E$ be Polish spaces, and suppose that $\mu$ and $\nu$ are $\Gamma$-invariant probability measures on the product spaces $\Omega_1 = \bbX_V^V \times \bbX_E^E$ and $\Omega_2=\bbY_V^V \times \bbY_E^E$ respectively.
 We say that $\nu$ is a $\Gamma$-\textbf{factor} of $\mu$ if there exists a measurable function $\pi :\Omega_1 \to \Omega_2$ such that $\mu(\pi^{-1}(A))=
\nu(A)$ for every measurable set $A \subseteq \Omega_2$ --- this means that if $\varphi=(\varphi_x)_{x\in V \cup E}$ is a random variable with law $\mu$ then $\pi(\varphi)=(\pi(\varphi)_x)_{x\in V \cup E}$ has law $\nu$ --- and that is $\Gamma$-equivariant in the sense that
\begin{equation}
\label{eq:intertwining}
\gamma \pi \omega_1 = \pi  \gamma \omega_1  \qquad \text{ for $\mu$-a.e.\ $\omega_1 \in \Omega_1$ for each $\gamma \in \Gamma$}.
\end{equation}
In this case we say that  $\nu$ is a $\Gamma$-factor of $\mu$ with \textbf{factor map} $\pi$. 
We say that a probability measure $\mu$ on a product space $\bbX_V^V \times \bbX_E^E$ is a $\Gamma$\textbf{-factor of i.i.d.} if it is a $\Gamma$-factor of a Bernoulli measure.

Observe that if $\pi:\Omega_1\to \Omega_2$ is such a factor map then
$\pi_* L^2(\Omega_2,\nu) := \{f \in L^2(\Omega_1,\mu) : f = g \circ \pi$ for some  $g \in L^2_0(\Omega_2,\nu)\}$
is a closed linear subspace of $L^2(\Omega_1,\mu)$ that is naturally identified with $L^2(\Omega_2,\nu)$ via the linear isometry 
\begin{align}
\label{eq:pi*identification}
\pi_*:L^2(\Omega_2,\nu) \to \pi_* L^2(\Omega_2,\nu) \qquad \qquad g \mapsto g \circ \pi.
\end{align}
Moreover, it follows by $\Gamma$-equivariance \eqref{eq:intertwining} that  the Markov operator $\hat P_{\nu}$ coincides with the restriction of  $\hat P_{\mu}$ to $\pi_* L^2(\Omega_2,\nu)$ under the identification \eqref{eq:pi*identification}. A simple consequence of this is that
\begin{align}
\rho(\nu,\Gamma) &= \sup\left\{ \frac{\|\hat P_{\nu} f\|_2}{\|f\|_2} : f\in L^2_0(\Omega_2,\nu) \setminus \{0\} \right\} 
=
\sup\left\{ \frac{\|\hat P_{\mu} f\|_2}{\|f\|_2} : f\in \pi_* L^2_0(\Omega_2,\nu) \setminus \{0\} \right\} 
\nonumber
\\&\leq 
\sup\left\{ \frac{\|\hat P_{\mu} f\|_2}{\|f\|_2} : f\in L^2_0(\Omega_1,\mu) \setminus \{0\} \right\} = \rho(\mu,\Gamma)
\label{eq:factor_rho}
\end{align}
whenever $\nu$ is a $\Gamma$-factor of $\mu$ with factor map $\pi$: 
 \emph{the spectral radius is decreasing under factors}.



To apply these results in our setting, we will use the fact, originally due to H\"aggstr\"om, Jonasson, and Lyons \cite{MR1913108}, that the Ising model and random cluster models can often be expressed as factors of i.i.d. The strongest and most general versions of these theorems are due to Harel and Spinka \cite{harel2018finitary}, who study the Gibbs measures of a very general class of positively associated models. The following theorem is an immediate consequence of \cite[Theorem 7]{harel2018finitary} together with \eqref{eq:factor_rho}. See also \cite{ray2019finitary,MR3603969} for further related results.

\begin{thm}
\label{thm:Ising_factor}
Let $G=(V,E,J)$ be a connected transitive weighted graph and let $\Gamma$ be a closed, transitive, unimodular subgroup of $\Aut(G)$. Then the following hold:
\begin{enumerate}
  \item The free and wired random cluster measures $\phi^f_{q,\beta,h}$ and $\phi^w_{q,\beta,h}$ on $G$ are $\Gamma$-factors of i.i.d.\ for every $q\geq 1$ and $\beta,h\geq 0$, so that
  \[\rho\bigl(\phi^\#_{q,\beta,h}\bigr) \leq \rho(G)\] for every $q \geq 1$, $\beta,h \geq 0$, and $\#\in\{f,w\}$.
  \item 
The plus Ising measure $\mathbf{I}^+_{\beta,h}$ on $G$ is a $\Gamma$-factor of i.i.d.\ for every $\beta\geq 0$ and $h\geq 0$, and therefore satisfies $\rho(\bI_{\beta,h}^+) \leq \rho(G)$ for every $\beta > 0$ and $h\geq 0$.
\end{enumerate}
\end{thm}

Note that item $1$ of this theorem does not imply that the free Ising measure is a factor of i.i.d.\ when $\beta>\beta_c$, since in this case we do not know that the Edwards--Sokal coupling can be implemented as a factor of the random cluster measure and a Bernoulli measure.
This is related to several very interesting problems regarding the regimes in which the free Ising model is a factor of i.i.d.\ that remain open in the nonamenable case, even when the underlying graph is a regular tree; see \cite{MR3603969} and references therein. In \cref{subsec:free_spectral_radius}, we show that the free \emph{gradient} Ising measure always has spectral radius at most $\rho(G)$.

\section{Bounds on the volume of finite clusters without FKG}
\label{sec:free_energy}

Let $G=(V,E,J)$ be a countable weighted graph. 
Suppose that $\mu$ is a probability measure on $[0,1]^E$, and let $\bp=(\bp_e)_{e\in E}$ be a $[0,1]^E$-valued random variable with law $\mu$. Let $(U_e)_{e\in E}$ be i.i.d.\ Uniform$[0,1]$ random variables independent of $\mathbf{p}$ and let $\omega=\omega(\bp,U)$ be the $\{0,1\}^E$-valued random variable defined by
\[
\omega(e) = \mathbbm{1}(U_e \leq \bp_e) \text{ for each $e\in E$.}
\]
 We write $\bP_\mu$ for the law of the pair of random variable $(\bp,\omega)$ and $\P_\mu$ for the joint law of $(\bp,\omega)$ and an independent random oriented root edge $\eta$ defined as in \cref{subsec:unimodularity_background}. We say that
the random variable
  $\omega$ is distributed as \textbf{percolation in random environment} on $G$ with \textbf{environment distribution} $\mu$. 
 Note that every random variable $\omega$ on $\{0,1\}^E$ can trivially be represented as percolation in random environment by taking the environment $\bp_e=\omega(e)$; we will be interested in less degenerate random environments in which  at least some of the probabilities $\bp_e$ do not belong to $\{0,1\}$. (We shall see that edge probabilities close to zero are far more problematic than edge probabilities close to $1$ as far as our methods are concerned.)

In this section we show how the methods of \cite{1808.08940} can be extended to percolation in random environment models that have a spectral gap but are not necessarily positively associated.

\subsection{The two-ghost inequality}

We begin by proving a generalization of the \emph{two-ghost inequality} of \cite{1808.08940} that applies to (possibly long-range) percolation in random environment models. The proof of this inequality is based ultimately on the methods of Aizenman, Kesten, and Newman \cite{MR901151}, who implicitly proved a related inequality in the course of their proof that Bernoulli percolation on $\Z^d$ has at most one infinite cluster almost surely. See  \cite{MR3395466} and the introduction of \cite{1808.08940} for further discussion of inequalities derived from the Aizenman-Kesten-Newman method and their applications.

Let $G=(V,E,J)$ be a connected, transitive weighted graph and let $\Gamma\subseteq \Aut(G)$ be a closed transitive group of automorphisms. Let $\mu$ be a $\Gamma$-invariant probability measure on $[0,1]^E$, let $\bp$ be a random variable with law $\mu$ and let $\omega$ be the associated percolation in random environment process as above. 
Let $h>0$. Given the environment $\bp$, let $\cG \in \{0,1\}^E$  be a random subset of $E$ where each edge $e$ of $E$ is included in $\cG$ independently at random with probability $1-e^{-hJ_e}$
 of being included, and where we take $\cG$ and $\omega$ to be conditionally independent given $\bp$. 
Following \cite{aizenman1987sharpness}, we call $\cG$ the \textbf{ghost field} and call an edge \textbf{green} if it is included in $\cG$. We write $\bP_{\mu,h}$ and $\bE_{\mu,h}$ for probabilities and expectations taken with respect to the joint law of $\bp$, $\omega$, and $\cG$. Similarly, we write $\P_{\mu,h}$ and $\E_{\mu,h}$ for probabilities and expectations taken with respect to the joint law of $\bp$, $\omega$, $\cG$, and $\eta$, where $\eta$ is the random oriented root edge of $G$ defined as in \cref{subsec:unimodularity_background}, which is taken to be independent of $(\bp,\omega,\cG)$. 
The density of $\cG$ is chosen so that
$\bP_{\mu,h}(A \cap \cG \neq \emptyset \mid \bp) = \exp\left[-h|A|_{J}\right]$ 
for every finite set $A \subseteq E$, where we write $|A|_J=\sum_{e\in A} J_e$.


Define $\sT_e$ to be the event that $e$ is closed in $\omega$ and that the endpoints of $e$ are in distinct clusters of $\omega$, each of which touches some green edge, and at least one of which is finite. The primary purpose of this section is to prove the following inequality.

\begin{thm}[Generalized Two-Ghost Inequality]
\label{thm:two_ghost} Let $G=(V,E,J)$ be a connected transitive weighted graph and let $\Gamma \subseteq \Aut(G)$ be a closed transitive unimodular subgroup of automorphisms. If $\mu$ is a $\Gamma$-invariant probability measure on $[0,1]^E$
then the inequality
\begin{align}
\label{eq:two_ghost}
\E_{\mu,h}\left[\mathbbm{1}(\sT_\eta) \sqrt{\frac{\bp_\eta}{(1-\bp_\eta)J_\eta}}\right] \leq 21 \sqrt{h}
\end{align}
holds for every $h>0$, where we take $\mathbbm{1}(\sT_\eta) \sqrt{\frac{\bp_\eta}{(1-\bp_\eta)J_\eta}}=0$ when $\bp_\eta=1$.
\end{thm}

Note that it is not obvious  \emph{a priori} that that the left hand side of \eqref{eq:two_ghost} is finite.

%
%
%
\cref{thm:two_ghost} has the following corollary which does not refer to the ghost field. For each $e\in E$ and $\lambda>0$, let $\sS_{e,\lambda}$ be the event that $e$ is closed in $\omega$ and that the endpoints of $e$ are in distinct clusters $K_1$ and $K_2$ of $\omega$, each of which has $|E(K_i)|_J \geq \lambda$ and at least one of which is finite. The deduction of \cref{cor:two_ghost_S} from \cref{thm:two_ghost} is similar to the proof of \cite[Corollary 1.7]{1808.08940} and is omitted.

\begin{corollary}
\label{cor:two_ghost_S} Let $G=(V,E,J)$ be a connected transitive weighted graph and let $\Gamma \subseteq \Aut(G)$ be a closed transitive unimodular subgroup of automorphisms. If $\mu$ is a $\Gamma$-invariant probability measure on $[0,1]^E$
then the inequality
\begin{align}
\E_{\mu}\left[\mathbbm{1}(\sS_{\eta,\lambda}) \sqrt{\frac{\bp_\eta}{(1-\bp_\eta)J_\eta}}\right] \leq \frac{42}{\sqrt{\lambda}}
\end{align}
holds for every $\lambda >0$, where we take $\mathbbm{1}(\sS_{\eta,\lambda}) \sqrt{\frac{\bp_\eta}{(1-\bp_\eta)J_\eta}}=0$ when $\bp_\eta=1$.
\end{corollary}

\medskip

We now begin to work towards the proofs of \cref{thm:two_ghost,cor:two_ghost_S}.
We will first prove these results under the additional assumptions that $\bp_e \in (0,1)$ for every $e\in E$ a.s.\ and that
\[
\E_{\mu}\left[\sqrt{\frac{\bp_\eta(1-\bp_\eta)}{J_\eta}} \right]<\infty
\]
and then show that both assumptions can be removed via a limiting argument.


Let $G=(V,E,J)$ be a connected transitive weighted graph and let $\Gamma$ be a closed transitive subgroup of automorphisms of $G$.
 For each environment $\bp\in (0,1)^E$ and subgraph $H$ of $G$, we define the \textbf{fluctuation} of $H$ to be
\begin{multline*}h_{\bp}(H):=  \sum_{e \in E(H)} \sqrt{J_e} \left[\sqrt{\frac{\bp_e}{1-\bp_e}}\mathbbm{1}\left(e\in \partial H\right)-\sqrt{\frac{1-\bp_e}{\bp_e}} \mathbbm{1}\left(e\in E_o(H)\right) \right]\\
=\sum_{e \in E(H)} \sqrt{\frac{J_e \bp_e}{1-\bp_e}} \frac{\bp_e-\mathbbm{1}(e\in E_o(H))}{\bp_e} \end{multline*}
where $E(H)$ denotes the set of (unoriented) edges that \emph{touch} $H$, i.e., have at least one endpoint in the vertex set of $H$, 
 $\partial H$ denotes the set of (unoriented) edges of $G$ that touch the vertex set of $H$ but are not included in $H$, and $E_\circ(H)$ denotes the set of (unoriented) edges of $G$ that are included in $H$, so that $E(H)=\partial H \cup E_o(H)$. This quantity is defined so that $h_{\bp}(K_v)$ and $|E(K_v)|_{J}$ are the final value and total quadratic variation of a certain martingale that arises when exploring the cluster $K_v$ of $v$ in $\omega$ in an edge-by-edge manner.

\begin{lemma}
\label{lem:AKN} Let $G=(V,E,J)$ be a connected transitive weighted graph and let $\Gamma \subseteq \Aut(G)$ be a closed transitive unimodular subgroup of automorphisms. Let $\mu$ be a $\Gamma$-invariant probability measure on $(0,1)^E$.
If 
\[
\E_{\mu}\left[\sqrt{\frac{\bp_\eta(1-\bp_\eta)}{J_\eta}} \right]<\infty
\]
then the inequality
\begin{align}
\E_{\mu,h}\left[\mathbbm{1}(\sT_\eta) \sqrt{\frac{ \bp_\eta}{J_\eta(1-\bp_\eta)}}\right] \leq 2\bE_{\mu,h}\left[\frac{|h_{\bp}(K_{o})|}{|E(K_{o})|_{J}}  \mathbbm{1}\bigl(|K_o| < \infty \text{ and } E(K_{o}) \cap \cG \neq \emptyset \bigr)\right]
\label{eq:AKN_main_estimate}
\end{align}
holds for every $p\in (0,1]$ and $h>0$.
\end{lemma}

\begin{proof}[Proof of \cref{lem:AKN}]
Let $\sF_e$ be the event that every cluster touching $e$ is finite, so that
 $\sT_e \cap \sF_e$ is the event that the endpoints of $e$ are in distinct finite clusters each of which touches $\cG$, and let $\sG_e$ be the event that there exists a finite cluster touching $e$ and $\cG$. For each edge $e$ of $G$ we can verify that
\begin{equation*}
\mathbbm{1}(\sT_e \cap \sF_e) = \mathbbm{1}(\omega(e)=0)\cdot \#\{\text{finite clusters touching $e$ and $\cG$}\}
\\- \mathbbm{1}\bigl(\{\omega(e)=0\}\cap \sG_e), 
\end{equation*}
and hence that
\begin{multline}
\label{eq:unimodghost1}
\bP_{\mu,h}(\sT_e \cap \sF_e \mid \bp\,) = \bE_{\mu,h}\left[\mathbbm{1}(\omega(e)=0)\cdot\#\{\text{finite clusters touching $e$ and $\cG$}\} \mid \bp\,\right]
\\- \bP_{\mu,h}\bigl(\{\omega(e)=0\} \cap \sG_e \mid \bp\,\bigr). 
\end{multline}
 The event $\sF_e \cap \sG_e$ is conditionally independent of the value of $\omega(e)$ given $\bp$, so that
\begin{multline}
\bP_{\mu,h}\bigl(\{\omega(e)=0\} \cap \sF_e \cap \sG_e \mid \bp\, \bigr)
= \frac{1-\bp_e}{\bp_e} \bP_{\mu,h}\bigl(\{\omega(e)=1\} \cap \sF_e \cap \sG_e \mid \bp\, \bigr).
\\
= \frac{1-\bp_e}{\bp_e} \bP_{\mu,h}\bigl(\{\omega(e)=1\}  \cap \sG_e \mid \bp\,\bigr).
\label{eq:unimodghost2}
\end{multline}
Putting together \eqref{eq:unimodghost1} and \eqref{eq:unimodghost2} yields that
\begin{multline}
\label{eq:AKNmainstep}
\bP_{\mu,h}(\sT_e \cap \sF_e \mid \bp\,) = \bE_{\mu,h}\left[\mathbbm{1}(\omega(e)=0)\cdot \#\{\text{finite clusters touching $e$ and  $\cG$}\} \mid \bp\,\right]
\\- \frac{1-\bp_e}{\bp_e}\bP_{\mu,h}(\{\omega(e)=1\} \cap \sG_e \mid \bp\,) - \bP_{\mu,h}\bigl(\{\omega(e)=0\} \cap \sG_e \setminus \sF_e \mid \bp\,\bigr).
\end{multline}
Finally, observe that $\{\omega(e)=0\} \cap \sG_e \setminus \sF_e$ and $\sT_e \cap \sF_e$ are disjoint and that $\sT_e$ coincides with $(\sT_e \cap \sF_e) \cup (\{\omega(e)=0\} \cap \sG_e \setminus \sF_e)$ up to a null set, so that \eqref{eq:AKNmainstep} implies that
\begin{multline*}
\bP_{\mu,h}(\sT_e \mid \bp\,) = \bE_{\mu,h}\left[\mathbbm{1}(\omega(e)=0)\cdot \#\{\text{finite clusters touching $e$ and  $\cG$}\} \mid \bp\right]
\\- \frac{1-\bp_e}{\bp_e}\bP_{\mu,h}(\{\omega(e)=1\} \cap \sG_e \mid \bp\,).
\end{multline*}
This equality can be written more concisely as
\begin{equation}
\label{eq:AKNmainstep2}
\bP_{\mu,h}(\sT_e \mid \bp\,)= 
 \bE_{\mu,h}\left[\frac{\bp_e-\omega(e)}{\bp_e} \cdot \#\{\text{finite clusters touching $e$ and  $\cG$}\} \;\Bigm|\; \bp\;\right].
\end{equation}
Note that we have not yet used any assumptions on the weighted graph $G$ or the group $\Gamma$.

We will now apply the assumption that the group $\Gamma$ is transitive and unimodular. 
Define a mass-transport function $F:E^\rightarrow\times E^\rightarrow \to \R$ by
\begin{equation*}
F(e_1,e_2) =\\ \bE_{\mu,h}\sum\left\{\frac{1}{2|E(K)|_{J}} \left[\frac{\bp_{e_1}-\omega(e_1)}{\bp_{e_1}}\right] \sqrt{\frac{\bp_{e_1}}{(1-\bp_{e_1})J_{e_1}}} : \begin{array}{l}\text{$K$ is a finite cluster}\\ \text{of $\omega$ touching $e_1,e_2$, and $\cG$}\end{array}\right\},
\end{equation*}
where we write $\sum\{x(i) :i\in I\} = \sum_{i\in I} x(i)$ and where we include the factor of $1/2$ to account for the fact that each edge in $E(K)$ can be oriented in two directions.
The multiset of numbers being summed over has cardinality either $0,1,$ or $2$, and we can therefore compute that
\[
 \E\sum_{e\in E^\rightarrow} J_e |F(\eta,e)| \leq 2 \E_{\mu,h}\left[ \frac{|\bp_{\eta}-\omega(\eta)|}{\bp_{\eta}} \sqrt{\frac{\bp_{\eta}}{(1-\bp_{\eta})J_{\eta}}}\right] = 4 \E_{\mu,h} \left[\sqrt{\frac{\bp_\eta(1-\bp_\eta)}{J_\eta}} \right] < \infty,
\]
where the final inequality is by the hypotheses of the lemma. 
 Thus, we may safely apply the mass-transport principle \eqref{eq:edgeMTP} together with \eqref{eq:AKNmainstep2} to deduce that
\begin{align*}
\E_{\mu,h}\left[\mathbbm{1}(\sT_\eta)\sqrt{\frac{\bp_{\eta}}{(1-\bp_{\eta})J_\eta}}\right] &=  \bE_{\mu,h}\left[ \frac{\bp_\eta-\omega(\eta)}{\bp_\eta}  \sqrt{\frac{\bp_{\eta} }{(1-\bp_{\eta})J_{\eta}}}\cdot\#\{\text{finite clusters touching $e$ and  $\cG$}\} \right]\nonumber
\\&=\E \sum_{e\in E^\rightarrow}J_e F(\eta,e) = \E \sum_{e\in E^\rightarrow}J_e F(e,\eta)\nonumber   \\&=
\E_{\mu,h} \sum\left\{\frac{h_{\bp}(K)}{|E(K)|_{J}} : \begin{array}{l}\text{$K$ is a finite cluster}\\ \text{of $\omega$ touching $\eta$ and $\cG$}\end{array}\right\}.
\end{align*}
For each vertex $v$ of $G$, let 
$\sO_v$  be the event that the cluster $K_v$ is finite and touches $\cG$. 
Then we deduce from the above that
\begin{multline*}\E_{\mu,h}\left[\mathbbm{1}(\sT_\eta)\sqrt{\frac{\bp_{\eta}}{(1-\bp_{\eta})J_\eta}}\right] 
\leq \E_{\mu,h} \sum\left\{\frac{|h_{\bp}(K)|}{|E(K)|_{J}} : \begin{array}{l}\text{$K$ is a finite cluster}\\ \text{of $\omega$ touching $\eta$ and $\cG$}\end{array}\right\}\\
\leq\E_{\mu,h}\left[\frac{|h_\bp(K_{\eta^-})|}{|E(K_{\eta^-})|}\mathbbm{1}\bigl(\sO_{\eta^-} \bigr)+ \frac{|h_\bp(K_{\eta^+})|}{|E(K_{\eta^+})|}\mathbbm{1}\bigl(\sO_{\eta^+} \bigr)\right]
=2\bE_{\mu,h}\left[\frac{|h_\bp(K_{o})|}{|E(K_{o})|}\mathbbm{1}\bigl(\sO_{o} \bigr)\right]
\end{multline*}
as claimed, where the final equality follows by transitivity. 
\end{proof}

As in \cite{1808.08940}, we will now bound the right hand side of \eqref{eq:AKN_main_estimate} using maximal inequalities for martingales\footnote{The original paper of Aizenman, Kesten and Newman \cite{MR901151} used large deviations estimates rather than maximal inequalities. The idea of using maximal inequalities instead, which leads to cleaner proofs and sharper inequalities, first arose in discussions with Vincent Tassion in 2018. }. Since the martingale we have here is a little more complicated than that of \cite{1808.08940}, we will need to introduce some more machinery before doing this.
In particular, we will employ the following simple variation on Doob's $L^2$ maximal inequality, which is inspired by Freedman's maximal inequality \cite{MR0380971}. It seems unlikely that this inequality is new, but we are not aware of a reference.
Let $X=(X_n)_{n\geq0}$ be a real-valued martingale with respect to the filtration $\cF=(\cF_n)_{n\geq 0}$, and suppose that $X_0=0$. The \textbf{quadratic variation process} $Q=(Q_n)_{n\geq 0}$ associated to $(X,\cF)$ is defined by $Q_0=0$ and
\[Q_n = \sum_{i=1}^n\E\left[ |X_i - X_{i-1}|^2 \mid \cF_{i-1} \right] \]
for each $n\geq 1$. Note that $Q$ is \emph{predictable}, that is, $Q_n$ is $\cF_{n-1}$-measurable for every $n\geq 1$.
\begin{lemma}
\label{lem:martingale_stuff}
 Let $(X_n)_{n\geq0}$ be a martingale with respect to the filtration $(\cF_n)_{n\geq 0}$ such that $X_0=0$, and let  $(Q_n)_{n\geq 0}$ be the associated quadratic variation process. Then
\[\E\Bigl[ \sup\bigl\{X_n^2 : n\geq 0,\, Q_n \leq \lambda \bigr\} \Bigr] \leq 4\lambda\]
for every $\lambda \geq 0$.
\end{lemma}

(This lemma holds vacuously if the increments of $X$ have infinite conditional variance a.s.)

\begin{proof}
Fix $\lambda \geq 0$ and let $\tau=\sup\{k\geq 0: Q_k \leq \lambda\}=\inf\{k\geq 0 : Q_k > \lambda\}-1$, which may be infinite. Since $Q_n$ is $\cF_{n-1}$-measurable for every $n\geq 0$, $\tau$ is a stopping time and $X_{n\wedge \tau}$ is a martingale. 
Thus, we have by the orthogonality of martingale increments  that
\begin{align*}
\E\left[X^2_{n\wedge \tau}\right] &= \sum_{i=1}^n\E\left[ (X_{i\wedge \tau}-X_{(i-1)\wedge \tau})^2\right]
= \sum_{i=1}^n\E\left[ \E\left[(X_{i\wedge \tau}-X_{(i-1)\wedge \tau})^2\mid \cF_{i-1} \right]\right]\\
&=\sum_{i=1}^n\E\left[ \E\left[(X_{i}-X_{i-1})^2\mid \cF_{i-1} \right] \mathbbm{1}(i \leq \tau)\right] = \E\left[ Q_{n \wedge \tau}\right] \leq \lambda
\end{align*}
for every $n\geq 1$. The claim follows by applying Doob's $L^2$ maximal inequality to
  $(X_{n\wedge \tau})_{n\geq 0}$.
\end{proof}

We next apply this lemma to prove a generalized version of the martingale estimate appearing in the proof of \cite[Theorem 1.6]{1808.08940}. (Note that $Q_n$ is increasing in $n$, so that $Q_\infty$ is well-defined as an element of $[0,\infty]$ and the case $T=\infty$ does not cause us any problems.)

\begin{lemma}
\label{lem:martingale_stuff2}
 Let $(X_n)_{n\geq0}$ be a martingale with respect to the filtration $(\cF_n)_{n\geq 0}$ such that $X_0=0$, and let  $(Q_n)_{n\geq 0}$ be the associated quadratic variation process. Then
\begin{equation}
\label{eq:martingale_stuff2}
\E\left[ \frac{\sup_{0 \leq n \leq T}|X_n|}{Q_T} (1-e^{-h Q_T})\mathbbm{1}(0<Q_T < \infty) \right] \leq 2\sqrt{eh}\sum_{k=-\infty}^\infty \frac{1-e^{-e^{k}}}{e^{k/2}}  \leq \frac{21}{2} \sqrt{h} \end{equation}
for every stopping time $T$ and every $h> 0$.
\end{lemma}

\begin{proof}
Write $M_n=\max_{0\leq m \leq n} |X_n|$ for each $n\geq 0$. Since $(1-e^{-hx})/x$ is a decreasing function of $x>0$, we may write
\begin{equation*}
  \E\left[ \frac{M_T}{Q_T}\bigl(1-e^{-hQ_T}\bigr)\mathbbm{1}(0<Q_T<\infty) \right]
\leq h\sum_{k=-\infty}^\infty \frac{1-e^{-e^{k}}}{e^k} \E\left[ M_T \mathbbm{1}(e^k \leq h Q_T \leq e^{k+1})\right].
\label{eq:scales}
\end{equation*}
\cref{lem:martingale_stuff} and Jensen's inequality let us bound each summand
\begin{align}
\label{eq:martingale_Jensen_not_optimized}
\E\left[ M_T \mathbbm{1}(e^k \leq hQ_T \leq e^{k+1})\right] &\leq \E\Bigl[ \max\bigl\{X_n^2 : n\geq 0,\, hQ_n \leq e^{k+1} \bigr\} \Bigr]^{1/2}\leq \sqrt{\frac{4e^{k+1}}{h}}
\end{align}
for each $k\in \Z$, 
so that
\begin{align*}
\E\left[ \frac{M_T}{Q_T}\bigl(1-e^{-hQ_T}\bigr)\mathbbm{1}(0<Q_T<\infty) \right]
\leq 2\sqrt{e h}\sum_{k=-\infty}^\infty \frac{1-e^{-e^{k}}}{e^{k/2}}
\end{align*}
as claimed.
This series is easily seen to converge. Moreover, the constant appearing here can be evaluated numerically as $2\sqrt{e}\sum_{k=-\infty}^\infty \frac{1-e^{-e^{k}}}{e^{k/2}}=10.47\ldots$, which we bound by $21/2$ for simplicity.
\end{proof}

\begin{remark}
The inequality \eqref{eq:martingale_stuff2} can be improved if one knows something about the tail of $Q_T$ by using Cauchy-Schwarz instead of Jensen in \eqref{eq:martingale_Jensen_not_optimized}. This eventually leads to better bounds on the exponents appearing in \cref{thm:main,thm:free_random_cluster,thm:finite_clusters}. We do not pursue this further here, but similar considerations for Bernoulli percolation are discussed in detail in \cite[Section 6]{1808.08940}.
\end{remark}

\begin{proof}[Proof of \cref{thm:two_ghost}]
As discussed above, we will first prove the theorem under the additional assumption that $\bp\in (0,1)^E$ almost surely and that
\begin{equation}
\label{eq:assumption}
\E_{\mu}\left[\sqrt{\frac{\bp_\eta(1-\bp_\eta)}{J_\eta}} \right]<\infty;
\end{equation}
We will then show that these assumptions can be removed via a limiting argument.

To this end, let $\mu$ be a $\Gamma$-invariant probability measure on $(0,1)^E$ and let $(\bp,\omega)$ be random variables with law $\bP_\mu$.
Write $K=K_o$ for the cluster of $o$ in $\omega$.
Fix an enumeration $E=\{e_1,e_2,\ldots\}$ of the edge set of $G$, and let $\opleq$ be the associated well-ordering of $E$, so that $e_i \opleq e_j$ if and only if $i \leq j$. After conditioning on the environment $\bp$, we will explore  $K$ one edge at a time and define a martingale in terms of this exploration process. 
At each stage of the exploration we will have a set of vertices $U_n$, a set of revealed open edges $O_n$, and a set of revealed closed edges $C_n$. 
We begin by setting $U_0=\{o\}$ and  $C_0=O_0=\emptyset$. Let $n\geq 1$. Given everything that has happened up to and including step $n-1$ of the exploration, we define $(U_n,O_n,C_n)$ as follows: If every edge touching $U_{n-1}$ is included in $O_{n-1}\cup C_{n-1}$, we set $(U_n,O_n,C_n)=(U_{n-1},O_{n-1},C_{n-1})$. Otherwise, we take $E_n$ to be the $\opleq$-minimal element of the set of edges that touch $U_{n-1}$ but are not in $O_{n-1}$ or $C_{n-1}$. If $E_n$ is open in $\omega$, we set $O_n=O_{n-1}\cup\{E_n\}$, $C_n=C_{n-1}$, and set $U_{n}$ to be the union of $U_n$ with the set of endpoints of $E_n$. Otherwise, $E_n$ is closed in $\omega$ and we set $O_n=O_{n-1}$, $C_n =C_{n-1} \cup \{E_n\}$, and $U_n = U_{n-1}$. Let $\cF_0$ be the $\sigma$-algebra generated by the environment $\bp$ and let $(\cF_n)_{n\geq 0}$ be the filtration generated by this exploration process and the environment $\bp$. 

Let $T = \inf\{n \geq 0: E(U_n) \subseteq O_n \cup C_n\}$ be the first time that there are no unexplored edges touching $U_n$,  setting $T=\infty$ if this never occurs, and observe that $(U_T,O_T,C_T,T)$ is equal to $(K,E_\circ(K),\partial K,|E(K)|)$. 
Let the process $(Z_n)_{n\geq 0}$ be defined by $Z_0=0$ and
\[Z_n = \sum_{i=1}^{n\wedge T} 
\sqrt{J_{E_i}} 
\left[\sqrt{\frac{\bp_{E_i}}{1-\bp_{E_i}}} \mathbbm{1}(\omega(E_i)=0) - \sqrt{\frac{1-\bp_{E_i}}{\bp_{E_i}}} \mathbbm{1}(\omega(E_i)=1)\right]
\]
for each $n\geq 1$. 
The process $Z$ is a martingale with respect to the filtration $(\cF_n)_{n\geq 0}$ satisfying $Z_T = h_{\bp}(K)$. Moreover, the quadratic variation process $Q_n=\sum_{i=1}^n \bE_{\mu}[(Z_{i+1}-Z_i)^2 \mid \cF_i]$ satisfies
\[
Q_n = \sum_{i=1}^{n\wedge T} \bE_{\mu}\left[J_{E_i}
\left[\frac{\bp_{E_i}}{1-\bp_{E_i}} \mathbbm{1}(\omega(E_i)=0) + \frac{1-\bp_{E_i}}{\bp_{E_i}} \mathbbm{1}(\omega(E_i)=1)\right] \Biggm| \cF_{n-1}\right] = \sum_{i=1}^{n\wedge T} J_{E_i}
\]
for every $n\geq 0$, so that $Q_T = |E(K)|_{J}$. It follows from \cref{lem:AKN} and \cref{lem:martingale_stuff2} that if \eqref{eq:assumption} holds then
\begin{align}
\label{eq:converttoMTG}
\E_{\mu,h}\left[\mathbbm{1}(\sT_\eta) \sqrt{\frac{\bp_\eta}{(1-\bp_\eta)J_\eta}}\right] \leq 2 \bE_{p}\left[\frac{|h_{\bp}(K)|}{|E(K)|_{J}} (1-e^{-h |E(K)|_{J}}) \mathbbm{1}\bigl(|E(K)|_{J}<\infty\bigr)\right]
\nonumber
\\ = 2 \bE_p\left[ \frac{|Z_T|}{Q_T}\bigl(1-e^{-h Q_T}\bigr)\mathbbm{1}(0<Q_T<\infty) \right] \leq 21 \sqrt{h}.
\end{align}
This establishes the claim in the case that $\bp\in(0,1)^E$ almost surely and \eqref{eq:assumption} holds.

Now suppose that $\bp\in (0,1]^E$ almost surely and that \eqref{eq:assumption} does not necessarily hold. Let $\bp$ be a random environment with law $\mu$ and for each $n\geq 1$ let $\bp^n \in (0,1)^E$ be the environment defined by
\[
\bp^n_e = \min\left\{\bp_e,nJ_e,e^{-1/n}\right\} \qquad \text{ for each $e\in E$.}
\]
We couple percolation in the random environments $(\bp^n)_{n\geq 1}$ and $\bp$ in the standard monotone way by letting $(U_e)_{e\in E}$ be i.i.d.\ Uniform$[0,1]$ random variables independent of $\bp$ and setting
\[
\omega(e)=\mathbbm{1}(U_i \leq \bp_e) \quad \text{ and } \quad \omega^n(e)=\mathbbm{1}(U_i \leq \bp_e^n) \quad \text{ for each $n\geq 1$ and $e\in E$,}
\]
so that $\omega^n$ converges to $\omega$ pointwise from below almost surely.
Write $\E_h$ for expectations taken with respect to the joint law of $\bp$, $\omega$, $(\omega^n)_{n\geq 1}$, the independent ghost field $\cG$, and the independent root edge $\eta$. 
Let $\sT^n_e$ be the event that $e$ is closed in $\omega^n$ and that the endpoints of $e$ are in distinct clusters of $\omega^n$, at least one of which touches some green edge and at least one of which is finite. The law of $\bp^n$ is clearly $\Gamma$-invariant, and since 
\[
\E_{\mu}\left[ \sqrt{\frac{\bp^n_\eta (1-\bp^n_\eta)}{J_\eta}}\right] \leq \E_\mu\left[ \sqrt{n(1-\bp^n_\eta)}\right]< \infty,
\]
we may apply the inequality \eqref{eq:converttoMTG} to deduce that
\begin{align}
\label{eq:AKN_Fatou}
\E_h\left[\mathbbm{1}(\sT_\eta^n) \sqrt{\frac{ \bp_\eta^n}{(1-\bp_\eta^n)J_\eta}}\right] \leq 21 \sqrt{h}
\end{align}
for every $n\geq 1$ and $h>0$. Since $\omega^n$ converges to $\omega$ pointwise from below, if $\sT_\eta$ holds then $\sT^n_\eta$ holds for all $n$ sufficiently large almost surely. It follows that
\begin{equation}
\label{eq:Fatou}
\mathbbm{1}(\sT_\eta) \sqrt{\frac{\bp_\eta}{(1-\bp_\eta)J_\eta}} \leq \liminf_{n\to\infty} \mathbbm{1}(\sT_\eta^n) \sqrt{\frac{ \bp_\eta^n}{(1-\bp_\eta^n)J_\eta}}
\end{equation}
almost surely, and Fatou's lemma implies that
\begin{equation}
\label{eq:ppositive}
\E_{\mu,h}\left[\mathbbm{1}(\sT_\eta) \sqrt{\frac{\bp_\eta}{(1-\bp_\eta)J_\eta}}\right] \leq   21 \sqrt{h}
\end{equation}
for every $h>0$ under the assumption that $\bp \in (0,1]^E$ almost surely.
 (Eq.\ \eqref{eq:Fatou} is an inequality rather than an equality since we might have that $\eta$ is incident to a finite cluster in $\omega^n$ for every $n\geq 1$ without this being true in $\omega$.)

 It remains to consider the case in which edge probabilities may be zero. Let $\bp$ be a random environment with law $\mu$ and for each $n\geq 1$ let $\bp^n \in (0,1]^E$ be the environment defined by
\[
\bp^n_e = \max\left\{\bp_e,\min\left\{1,\frac{J_e}{n}\right\}\right\} \qquad \text{ for each $e\in E$.}
\]
Similarly to before, we can couple the associated percolation processes $\omega$ and $(\omega^n)_{n\geq 1}$ so that $\omega^n$ tends to $\omega$ pointwise from above. Since $\sum_{e\in E^\rightarrow_v} J_e<\infty$ for every $v\in V$, we have for every finite set $A \subseteq V$ there exists an almost surely finite random $N_A$ such that $\{e \in E(A) : \omega^n(e)=1\}=\{e \in E(A) : \omega(e)=1\}$ for every $n\geq N_A$. It follows easily that
\begin{equation}
\mathbbm{1}(\sT_\eta) \sqrt{\frac{\bp_\eta}{(1-\bp_\eta)J_\eta}} = \lim_{n\to\infty} \mathbbm{1}(\sT_\eta^n) \sqrt{\frac{\bp_\eta^n}{(1-\bp_\eta^n)J_\eta}}
\end{equation}
almost surely, and the claim follows from \eqref{eq:ppositive} and Fatou's lemma as before.
\end{proof}

\subsection{Finite clusters in Bernoulli percolation}

We now apply the two-ghost inequality to study finite clusters in percolation in random environment models under a spectral gap condition. 
 We begin with the case of Bernoulli percolation on a locally finite graph so that we can present the basic method in the simplest possible setting.

\begin{thm}
\label{thm:finite_percolation}
Let $G$ be a connected, locally finite, nonamenable, transitive unimodular graph with spectral radius $\rho<1$ and let $o$ be a vertex of $G$. Then there exist positive constants $C=C(\deg(o),\rho)$ and $\delta=\delta (\deg(o),\rho)$ such that
\[
\bP_p(n\leq |K_o| <\infty) \leq C n^{-\delta}
\]
for every $n\geq 1$ and $p\in [0,1]$.
\end{thm}

The proof will apply the following general fact about percolation on nonamenable graphs, which is a version of \emph{Schramm's Lemma}. A similar lemma for Bernoulli percolation (with a very different proof) first arose in unpublished work of Schramm; see \cite{kozma2011percolation} for a detailed discussion and \cite[Section 3]{1712.04911} for further related results.

\begin{prop}
\label{prop:Schramm}
Let $G=(V,E,J)$ be a connected, nonamenable, transitive, weighted graph, let $o$ be a vertex of $G$, and let $\Gamma \subseteq \Aut(G)$ be a closed unimodular transitive subgroup of automorphisms. Suppose that $\omega \in \{0,1\}^E$ is a random variable whose law is invariant under $\Gamma$ and that $(X_n)_{n\geq 0}$ is an independent random walk on $G$ started at $X_0=o$. Then
\[
\P(X_0 \text{ and } X_n \text{ both belong to the same finite cluster of $\omega$}) \leq \rho(G)^n
\]
for every $n\geq 0$.
\end{prop}

\begin{proof}[Proof of \cref{prop:Schramm}]
Let $P:L^2(V)\to L^2(V)$ be the Markov operator on $G$, so that
\[
\P(X_0 \text{ and } X_n \text{ both belong to the same finite cluster of $\omega$})
= \E \left[\langle P^n \mathbbm{1}_{K_o},\mathbbm{1}_{o}\rangle \mathbbm{1}(|K_o|<\infty)\right].
\]
We have by the mass-transport principle that
\[
\E \left[\langle P^n \mathbbm{1}_{K_o},\mathbbm{1}_{o}\rangle \mathbbm{1}(|K_o|<\infty)\right] = \E \left[\frac{\langle P^n \mathbbm{1}_{K_o},\mathbbm{1}_{K_o}\rangle}{|K_o|} \mathbbm{1}(|K_o|<\infty)\right] \leq \|P\|^n \cdot \P(|K_o|<\infty)
\]
which implies the claim.
\end{proof}

\begin{proof}[Proof of \cref{thm:finite_percolation}]
First note that if $p \leq 1/2\deg(o)$ then counting paths gives that $\bE_p |K_o| \leq \sum_{i=0} p^i \deg(o)^i \leq 2$, so that the claim is trivial in this case. We may therefore assume throughout the proof that $p \geq p_0:= 1/2\deg(o)$.

Fix $p \geq p_0$ and let $\omega$ be an instance of Bernoulli-$p$ bond percolation on $G$.
Let $X$ be a random walk on $G$ started at $o$ and independent of the percolation configuration $\omega$, and let $X_{i,i+1}$ be the edge crossed by $X$ between times $i$ and $i+1$ for each $i\geq 0$. For each $i\geq 0$, let $\omega^i$ be obtained from $\omega$ by setting $\omega^0=\omega$ and
\[\omega^i(e) =\begin{cases} 1 & e \in \{X_{j,j+1} : 0 \leq j  \leq i-1\}\\
\omega(e) & e \notin \{X_{j,j+1} : 0 \leq j  \leq i-1\}
\end{cases}
\]
for each $i\geq 1$ and $e\in E$.
For each $n,m \geq 1$ let $\sA_{n,m}$ be the event that the cluster of $X_0=o$ in $\omega$ is finite and that $X_0$ and $X_m$ are in distinct clusters of $\omega$ each of which touches at least $n$ edges. For each $n,m \geq 0$ and $1\leq i \leq m$, let $\sB_{n,m,i}$ be the event that the following hold:
\begin{enumerate}
  \item
 $X_0$ and $X_m$ are in distinct clusters of $\omega^{i-1}$ each of which touches at least $n$ edges,
 \item the cluster of $X_0$ is finite in $\omega^{i-1}$, and
 \item either $X_0$ and $X_m$ are connected in $\omega^i$ or the cluster of $X_0$ is infinite in $\omega^i$.
\end{enumerate}
 On the event $\sA_{n,m}$ the vertices $X_0$ and $X_m$ are connected in $\omega^m$ and not connected in $\omega^0$, and since $\omega^i$ is monotone increasing in $i$ it follows that
\begin{equation}
\label{eq:ABsetinclusion}
\sA_{n,m} \subseteq \bigcup_{i=1}^m \sB_{n,m,i}
\end{equation}
for every $n,m\geq 1$.
Now, for each $n,i\geq 1$ let $\sC_{n,i}$ be the event that the cluster of $X_{i-1}$ in $\omega$ is finite and that $X_{i-1}$ and $X_{i}$ are in distinct clusters of $\omega$ each of which touches at least $n$ edges. Observe that $\sC_{n,i} \supseteq \sB_{n,m,i} \cap \{\omega(X_{j,j+1})=1 \text{ for every $0\leq j \leq i-2$}\}$ for every $n,m\geq 1$ and $1\leq i \leq m$. Moreover, these two events are conditionally independent given the random walk $X$, and we deduce that
\[
\P\bigl(\sC_{n,i}) \geq \E \left[\P(\omega(X_{j,j+1})=1 \text{ for every $0\leq j \leq i-2$}\mid X) \P\bigl( \sB_{n,m,i}\mid X)\right]\geq p^{i-1} \P(\sB_{n,m,i})
\]
for every $n,m\geq 1$ and $1\leq i \leq m$. Applying \eqref{eq:ABsetinclusion} and \cref{cor:two_ghost_S} we deduce that
\begin{equation}
\label{eq:Bernoulli_surgery}
\P(\sA_{n,m}) \leq \sum_{i=1}^m p^{-i+1} \P\bigl(\sC_{n,i})= \P\bigl(\sS_{\eta,n}) \sum_{i=1}^m p^{-i+1} 
\leq \frac{p_0^{-m+1}}{1-p_0} \sqrt{\frac{1-p}{p}} \frac{42}{ \sqrt{n}}
\leq \frac{42}{ p_0^m\sqrt{n}}
\end{equation}
for every $n,m\geq 1$, where we used transitivity in the central equality.
On the other hand, we trivially have that
\begin{equation*}\P(\sA_{n,m}) \geq \P(n \leq |K_{X_0}|,|K_{X_m}| <\infty) \\- \P(\text{$X_0$ and $X_m$ belong to the same finite cluster of $\omega$})
\end{equation*}
for every $n,m\geq 1$.
Write $\rho=\rho(G)$. Using \cref{thm:Bernoulli_radius} as in \cref{example:FiniteClusters} to bound the first term and \cref{prop:Schramm} to bound the second gives that
\begin{align}
\P(\sA_{n,m})&\geq  \bP_p(n \leq |K_{o}| <\infty)^2 - \rho^m \left[ \bP_p(n \leq |K_{o}| <\infty)-\bP_p(n \leq |K_{o}| <\infty)^2\right] - \rho^m
\nonumber
\\
&\geq \bP_p(n \leq |K_{o}| <\infty)^2 - 2 \rho^m,
\label{eq:using_rho}
\end{align}
and hence by \eqref{eq:Bernoulli_surgery} that
\[
\bP_p(n \leq |K_{o}| <\infty)^2 \leq 2\rho^m + \frac{42}{ p_0^m\sqrt{n}}
\]
for every $n,m\geq 1$. The claim follows easily by taking $m=\lceil c \log n \rceil$ for an appropriate choice of constant $c=c(\deg(o),\rho)$; we omit the details.
\end{proof}


\subsection{Finite clusters in the random cluster model}
\label{subsec:finite_FK}

We now prove \cref{thm:finite_clusters}, which concerns finite clusters in the random cluster model on transitive weighted graphs that are not necessarily locally finite.

Let us first discuss how the random cluster model may be represented as a percolation in random environment model via a non-integer version of the Edwards--Sokal coupling. This representation was first used by Bollob\'as, Grimmett, and Janson in the context of the complete graph \cite[Section 3]{bollobas1996random}.
Let $q \geq 1$ and $\beta \geq 0$ and let $\omega$ be a sample of the random cluster measure $\phi_{q,\beta,0}$ on a weighted graph $G=(V,E,J)$ with $V$ finite. 
Given $\omega$, colour each \emph{cluster} of $\omega$ \emph{red} or \emph{white} independently at random with probability $1/q$ to be coloured red, let $R$ be the set of vertices belonging to a red cluster, and let $\omega' \in \{0,1\}^E$ be defined by $\omega'(e)=\omega(e)\mathbbm{1}($both endpoints of $e$ belong to $R)$. Note that when $q \in \{2,3,\ldots\}$, the set $R$ has the same distribution as set of vertices that have some particular colour in the Potts model. For each set $A \subseteq V$, let $E(A)$ be the set of edges touching $A$, let $E_o(A)$ be the set of edges with both endpoints in $A$, and let $\overline{A}$ be the subgraph of $G$ with vertex set $V \setminus A$ and edge set $E \setminus E(A)$, where edges inherit their weights from $G$. 
 It is shown in \cite[Eq.\ (3.76)]{GrimFKbook} that
\begin{align}
\label{eq:red_percolation}
\P(R=A, \omega' = \xi) 
&= \frac{Z_{\overline{A}}(q-1,\beta,0)}{Z_G(q,\beta,0)} \prod_{e\in E_o(A)} (e^{2\beta J_e}-1)^{\xi(e)} 
\end{align}
for every $A\subseteq V$ and $\xi \in \{0,1\}^{E}$ such that $\xi(e)=0$ for every edge $e \notin E_o(A)$, where $Z_G(q,\beta,h)$ is the partition function for the random cluster model on $G$. Since this expression depends on $\xi$ only through the product $\prod_{e\in E_o(A)} (e^{2\beta J_e}-1)^{\xi(e)}$, it follows that the conditional distribution of $\omega'$ given $R$ coincides with that of the Bernoulli bond percolation process on $E_o(R)$ in which each edge of $E_o(R)$ is included independently at random with inclusion probability $(e^{2\beta J_e}-1)/e^{2\beta J_e}=1-e^{-2\beta J_e}$. This allows us to think of the restriction of the random cluster model to the (random) set of red vertices as a percolation in random environment model.

(We note that for the FK-Ising model there is an alternative percolation in random environment representation, due to Lupu and Werner \cite{MR3485382}, in which Bernoulli edges are added to the loop $O(1)$ model. See \cref{subsec:mainproof} for further discussion. This representation could also be used to prove \cref{thm:finite_clusters} in the case $q=2$. In fact, using this representation makes the proof somewhat simpler in this case since the edge-inclusion probabilities $\bp_e \geq \sinh(\beta J_e)/\cosh^2(\beta J_e)$ are bounded away from zero for each $e\in E$ almost surely.) 

Let us now discuss how this representation extends to the infinite volume case and to models with non-zero external field. Let $G=(V,E,J)$ be an infinite, connected, weighted graph. Let $q\geq 1$, $\beta,h\geq 0$, and $\# \in \{f,w\}$, and let $\omega \in \{0,1\}^{E \cup V}$ be a random variable with law $\phi_{q,\beta,h}^\#$.
Given $\omega$, we colour the clusters of $\omega$ red or white as follows:
\begin{enumerate}
  \item
Colour each cluster of $\omega$ intersecting the set $\{v:\omega(v)=1\}$ red.
\item If $\#=w$, colour each infinite cluster of $\omega$ red.
\item Choose to colour each remaining cluster of $\omega$ red or white independently at random, with probability $1/q$ to be coloured red.
\end{enumerate}
Let $R$ be the set of vertices that are coloured red. It follows from \cref{eq:red_percolation} and a straightforward limiting argument that, conditional on $R$, the restriction of $\omega$ to $E_o(R) \cup R$ is a product measure in which 
$\P(\omega(x)=1\mid R)= (e^{2\beta J_x}-1)/e^{2\beta J_x}=1-e^{-2\beta J_x}$ for every $x \in E_o(R) \cup R$, where we write $J_v = h$ for every $v\in R$. 

\begin{proof}[Proof of \cref{thm:finite_clusters}]
By scaling, we may assume without loss of generality that $\sum_{e\in E^\rightarrow_o} J_e=1/2$. 
Since the the restriction of $\phi^\#_{q,\beta,h}$ is stochastically dominated by the product measure $\phi^\#_{1,\beta,h}$ \cite[Theorem 3.21]{GrimFKbook}, a simple counting argument as before yields that $\phi^\#_{q,\beta,h} |K_o| \leq \phi^\#_{1,\beta,h} |K_o| \leq 2$ for every $\beta \leq 1/2$, $h \geq 0$, $q\geq 1$, and $\# \in \{f,w\}$. This concludes the proof in this case, so that it suffices to consider the case $\beta \geq 1/2$.

Fix $\beta \geq 1/2$, $h \geq 0$, $q\geq 1$, and $\# \in \{f,w\}$. 
Let $\omega$ be a random variable with law $\phi^\#_{q,\beta,h}$, let $R$ be the random subset of $V$ defined by colouring the clusters of $\omega$ red or white as above, and let $\omega_R \in \{0,1\}^E$ be defined by $\omega_R(e)=\omega(e) \mathbbm{1}($both endpoints of $e$ belong to $R)$. Thus, as discussed above, $\omega_R$ may be thought of as a percolation in random environment model in which the environment $\bp$ is given by
\[
\bp_e = (1-e^{-2\beta J_e})\mathbbm{1}\left(e \in E_o(R)\right).
\]
For a more general percolation in random environment model, the fact that these probabilities can be zero could be problematic. In our case, however, there is enough independence to pull the proof through with care.

Let $X$ be a random walk on $G$ started at $o$ and independent of $(\omega,R)$,  let $X_{i,i+1}$ be the edge crossed by $X$ between times $i$ and $i+1$ for each $i\geq 0$, and let $J_i$ be the weight of the edge $X_{i,i+1}$ for each $i\geq 0$. 
For each $m\geq1$, let $\sR_m$ be the event that $X_i \in R$ for every $0 \leq i \leq m$. The definitions ensure that 
\begin{equation}
\label{eq:Rprob}
\P(\sR_m \mid \omega, X) \geq q^{-m-1}\end{equation}
for every $m \geq 1$.
For each $i\geq 0$, let $\omega^i_R$ be obtained from $\omega_R$ by setting $\omega^0_R=\omega_R$ and
\[\omega^i_R(e) =\begin{cases} 1 & e \in \{X_{j,j+1} : 0 \leq j  \leq i-1\}\\
\omega_R(e) & e \notin \{X_{j,j+1} : 0 \leq j  \leq i-1\}
\end{cases}
\]
for each $i\geq 1$ and $e\in E$.
For each $m \geq 1$ and $\lambda>0$, let $\sA_{\lambda,m}$ be the event that the cluster of $X_0=o$ in $\omega$ is finite and that $X_0$ and $X_m$ are in distinct clusters of $\omega$ each of which touches a set of edges with total weight at least $\lambda$. For each $m \geq 1$, $\lambda>0$, and $1\leq i \leq m$, let $\sB_{\lambda,m,i}$ be the event that the following hold:
\begin{enumerate}
  \item
 $X_0$ and $X_m$ are in distinct clusters of $\omega^{i-1}_R$ each of which touches a set of edges with total weight at least $\lambda$,
 \item the cluster of $X_0$ is finite in $\omega^{i-1}_R$, and
 \item either $X_0$ and $X_m$ are connected in $\omega^i_R$ or the cluster of $X_0$ is infinite in $\omega^i_R$.
\end{enumerate}
 On the event $\sA_{\lambda,m} \cap \sR_m$ the vertices $X_0$ and $X_m$ are connected in $\omega^m_R$ and not connected in $\omega^0_R$, and since $\omega^i_R$ is monotone increasing in $i$ it follows that
\begin{equation}
\label{eq:ABsetinclusionFK}
\sA_{\lambda,m} \cap \sR_m \subseteq \bigcup_{i=1}^m \sB_{\lambda,m,i} \cap \sR_m
\end{equation}
for every $m\geq 1$ and $\lambda>0$.
Now, for each $i\geq 1$ and $\lambda>0$ let $\sC_{\lambda,i}$ be the event that the cluster of $X_{i-1}$ in $\omega_R$ is finite and that $X_{i-1}$ and $X_{i}$ are in distinct clusters of $\omega_R$ each of which touches a set of edges with total weight at least $\lambda$. Observe that $\sC_{\lambda,i} \cap \sR_m \supseteq \sB_{\lambda,m,i} \cap \{\omega_R(X_{j,j+1})=1 \text{ for every $0\leq j \leq i-2$}\} \cap \sR_m$ for every $n,m\geq 1$ and $1\leq i \leq m$. 
The events $\sB_{\lambda,m,i}$ and $\{\omega_R(X_{j,j+1})=1$ for every $0\leq j \leq i-2\}$ are conditionally independent given the random walk $X$ and the set $R$, and we deduce that
\begin{align}
\P\bigl(\sC_{\lambda,i} \cap \sR_m \mid X,R) &\geq \mathbbm{1}(\sR_m)\P(\omega_R(X_{j,j+1})=1 \text{ for every $0\leq j \leq i-2$}\mid X,R) \P\bigl( \sB_{\lambda,m,i}  \mid X,R) 
\nonumber
\\&\geq \mathbbm{1}(\sR_m) \P(\sB_{\lambda,m,i} \mid X,R) \prod_{j=0}^{m-2} (1-e^{-2\beta J_j}) 
\label{eq:FK_R_Surgery}
\end{align}
for every $m\geq 1$, $\lambda>0$, and $1\leq i \leq m$.

To proceed, we will first complete the proof under the additional assumption that there exists $\alpha<1$ such that $\sum_{e \in E^\rightarrow_o} J_e^{\alpha}<\infty$, which holds trivially in the locally finite case, before explaining how this assumption can be removed. 
Under this assumption we have that $\E J_\eta^{-(1-\alpha)} = \E e^{-(1-\alpha)\log J_\eta} < \infty$ and since $\beta \geq 1/2$ and  $1-e^{-x} \geq x/2$ for every $x \in [0,1]$ we deduce that
\begin{multline*}
\E \exp\left( -(1-\alpha) \log (1-e^{-2\beta J_\eta}) \right) \leq 
\E \exp\left( -(1-\alpha) \log (1-e^{-J_\eta}) \right) \\ \leq 2 \E \exp\left( -(1-\alpha) \log J_\eta \right)<\infty.
\end{multline*}
 Since the random variables $(J_i)_{i\geq 0}$ are i.i.d., we have by a Chernoff bound that there exists a finite constant $C_1$  such that
\begin{align*}
\P\left(\prod_{j=0}^{m-2} (1-e^{-2\beta J_j}) \leq e^{-C_1 (m-1)} \right) 
&= 
\P\left(\sum_{j=0}^{m-2} -\log (1-e^{-2\beta J_j}) \geq C_1 (m-1) \right)
\\&\leq e^{-C_1(1-\alpha)(m-1)}\left[2 \E \exp\left( -(1-\alpha) \log J_\eta \right)\right]^{m-1} \leq q^{-2(m-1)}
\end{align*}
for every $m \geq 1$.
For each $m \geq 1$, let $\sW_{m}$ be the event that $\prod_{j=0}^{m-2} (1-e^{-2\beta J_j}) \geq e^{-C_1 (m-1)}$. It follows from \eqref{eq:FK_R_Surgery} that
\begin{align*}
\P(\sB_{\lambda,m,i} \cap \sR_m \mid X,R) \leq e^{C_1 (m-1)} \P\bigl(\sC_{\lambda,i} \cap \sR_m \mid X,R)  + \mathbbm{1}(\sW_m) 
\end{align*}
for every $m\geq 1$, $\lambda>0$, and $1\leq i \leq m$. Taking expectations and applying \eqref{eq:ABsetinclusionFK} we deduce that
\begin{align}
\label{eq:Bernoulli_surgeryFK}
\P(\sA_{\lambda,m} \cap \sR_m)  &\leq m q^{-2(m-1)}  + e^{C_1 (m-1)} \sum_{i=1}^m \P\bigl(\sC_{\lambda,i} \cap \sR_m).
\end{align}
for every $m\geq 1$ and $\lambda>0$. There is easily seen to exist a positive constant $c_1$ such that $\bp_e / (1-\bp_e) J_e \geq c_1 \mathbbm{1}(e\in E_o(R))$ for every $\beta \geq 1/2$. Thus, \cref{cor:two_ghost_S} yields that there exists a constant $C_2 = 42/c_1$ such that
\[
\P\bigl(\sC_{\lambda ,i} \cap \sR_m ) \leq  \frac{C_2}{\sqrt{\lambda}}
\]
for every $m\geq 1$, $\lambda>0$, and $1 \leq i \leq m$ and hence that
\[
\P(\sA_{\lambda,m} \cap \sR_m) \leq m q^{-2(m-1)} + m e^{C_1 (m-1)} \frac{C_2}{\sqrt{\lambda}}
\]
for every $m\geq 1$ and $\lambda>0$.
 Since $\sA_{\lambda,m}$ is measurable with respect to the $\sigma$-algebra generated by $\omega$ and $X$, it follows from \eqref{eq:Rprob} that there exist finite constants $C_3$ and $C_4$ such that
\begin{multline}
\P(\sA_{\lambda,m})= \frac{\P(\sA_{\lambda,m} \cap \sR_m)}{\P(\sR_m \mid \sA_{\lambda,m})} \leq q^{m+1}\P(\sA_{\lambda,m} \cap \sR_m) \\\leq m q^{-m+3} + m e^{C_1 m} q^{m+1}\frac{C_2}{\sqrt{\lambda}}
\leq m q^{-m+3} + \frac{C_3 e^{C_4 m}}{\sqrt{\lambda}}
\label{eq:FK_surgery}
\end{multline}
for every $m\geq 1$ and $\lambda>0$.

We may now conclude the proof in an essentially identical way to the proof of \cref{thm:finite_percolation}. Indeed, we have trivially have that
\begin{multline*}\P(\sA_{\lambda,m}) \geq \P(\lambda \leq |E(K_{X_0})|_J,|E(K_{X_m})|_J <\infty) \\- \P(\text{$X_0$ and $X_m$ belong to the same finite cluster of $\omega$})
\end{multline*}
for every $m\geq 1$ and $\lambda>0$.
We deduce from \cref{thm:Ising_factor,prop:Schramm} that
\begin{align}
\P(\sA_{\lambda,m})
&\geq \P(\lambda \leq |E(K_{X_0})|_J <\infty)^2 - 2 \rho(G)^m,
\label{eq:using_rho_FK}
\end{align}
and hence by \eqref{eq:FK_surgery} that
\[
\P(\lambda \leq |E(K_{X_0})|_J <\infty)^2 \leq 2\rho(G)^m + m q^{-m+3} + \frac{C_3 e^{C_4 m}}{\sqrt{\lambda}}
\]
for every $m\geq 1$ and $\lambda>0$. As before, the claim follows easily by taking $m=\lceil c \log \lambda \rceil$ for an appropriate choice of constant $c$.

Let us now briefly indicate how the assumption that $\sum_{e\in E^\rightarrow_o} J_e^\alpha < \infty$ for some $\alpha<1$ can be removed; if the reader is only interested in the locally finite case they may safely skip this paragraph. First, we easily verify from the definitions that the connected, transitive weighted graph $G'=(V,E,J^2)$ is nonamenable if and only if $G=(V,E,J)$ is. Moreover, any automorphism-invariant percolation process on $G$ may also be thought of as an automorphism-invariant percolation process on $G'$, and this change in perspective does not affect whether or not the process is a factor of i.i.d. Applying \cref{thm:Ising_factor}, it follows in particular that the random cluster model on $G$ has spectral radius at most $\rho(G')<1$ when considered as a percolation process on $G'$.
Moreover, if $\eta'$ is a random edge emanating from $o$ chosen with probability proportional to $J_e^2$ then we have that $\E J_{\eta'}^{-1} = (\sum_{e\in E^\rightarrow_o} J_e)/(\sum_{e\in E^\rightarrow_o} J_e^2)<\infty$ and hence that
\[
\E \exp\left( -\log \frac{e^{2\beta J_{\eta'}}-1}{e^{2\beta J_{\eta'}}} \right) \leq 
\E \exp\left( - \log \frac{e^{J_{\eta'}}-1}{e^{J_{\eta'}}} \right) \leq 2 \E \exp\left( - \log J_{\eta'} \right)<\infty
\]
for every $\beta \geq 1/2$.
 These observations allow us to straightforwardly extend the above analysis to arbitrary connected, nonamenable, transitive weighted graphs by considering the random walk on $G'$ instead of $G$; we omit the details.
\end{proof}

\begin{proof}[Sketch of proof of \cref{thm:free_random_cluster}]
Fix $q \geq 1$.
It is proven in \cite{1901.10363} that $\phi^\#_{q,\beta,0} |K_o| < \infty$ for every $\beta<\beta_c^\#(q)$ and $\# \in \{f,w\}$. 
As in \cite{Hutchcroft2016944}, the FKG inequality implies that the sequence 
\[
\kappa_{q,\beta}^\#(m):=\inf\left\{\phi^f_{q,\beta,0}(x \leftrightarrow y) : x,y \in V, d(x,y) \leq m \right\} 
\]
is supermultiplicative in the sense that $\kappa_{q,\beta}^\#(n+m)\geq \kappa_{q,\beta}^\#(n)\kappa_{q,\beta}^\#(m)$ for every $n,m\geq 1$, $\beta \geq 0$, and $\#\in \{f,w\}$, and it follows from Fekete's lemma \cite[Appendix II]{grimmett2010percolation} that
\[
\sup_{m\geq 1} \kappa_{q,\beta}^\#(m)^{1/m} = \lim_{m\to \infty} \kappa_{q,\beta}^\#(m)^{1/m} \leq \liminf_{m\to\infty} \left(\frac{\phi^\#_{q,\beta,0} |K_o|}{|B(o,m)|} \right)^{1/m} = \frac{1}{\operatorname{gr}(G)}
\]
for every $\# \in \{f,w\}$ and $0\leq \beta < \beta_c^\#(q)$, where $\operatorname{gr}(G)=\limsup_{n\to\infty} |B(o,n)|^{1/n}$ is the exponential growth rate of $G$.
Following a very similar argument to that of \cref{thm:finite_clusters} but using a geodesic between two points $x$ and $y$ with $d(x,y)=m$ minimizing $\phi^f_{q,\beta,0}(x \leftrightarrow y)$ instead of a random walk and using the FKG inequality instead of spectral considerations in \eqref{eq:using_rho_FK} yields that there exist constants $C$ and $\delta$ (depending on $\operatorname{gr}(G)$, $\deg(o)$, and $q$) such that
\[
\phi^\#_{q,\beta,0}(|K_o|\geq n) \leq C n^{-\delta}
\]
for every $\# \in \{f,w\}$, $n\geq 1$, and $0\leq \beta<\beta_c^f$. We conclude by taking $\beta \uparrow \beta_c^f(q)$ and using left-continuity of the free random cluster measure $\phi^f_{q,\beta,0}$.
\end{proof}

\subsection{A general continuity theorem}

To illustrate the flexibility of the method of proof developed here, and for possible future applications, let us also make note of the following very general theorem for percolation in random environment models with all edge probabilities positive. This theorem is not needed for the proofs of our main results. 

 \begin{thm}
\label{thm:general_AKN}
Let $G=(V,E,J)$ be a connected, transitive, nonamenable weighted graph and let $\Gamma \subseteq \Aut(G)$ be a closed transitive unimodular subgroup of automorphisms. Let $\cM$ be a tight family of $\Gamma$-invariant probability measures on $(0,1]^E$ with $\sup_{\mu \in \cM} \rho(\mu) < 1$. Then there exists a decreasing function $f:\N\to [0,1]$ such that $\lim_{n\to\infty} f(n)=0$ and
\[
\bP_{\mu}(n \leq |K_o| < \infty) \leq f(n)\]
for every $n\geq 1$ and $\mu \in \cM$. In particular, $\cM_\infty=\{\mu \in \cM : \bP_\mu$ is supported on configurations with no infinite clusters$\}$ is a weakly closed subset of $\cM$.
 \end{thm}

Note that the assumption that $\sup_{\mu \in \cM} \rho(\mu) < 1$ can be replaced by the assumption that $\cM=\overline{\cM}_\infty$ and that every measure in $\cM$ is positively associated.


\begin{proof}[Proof of \cref{thm:general_AKN}]
By scaling, we may assume without loss of generality that $\sup_e J_e \leq 1$.
The assumption that $\cM$ is tight on $(0,1]^E$ is equivalent to the assertion that there exists an increasing function $g:(0,1]\to [0,1]$ with $\lim_{\eps \downarrow 0}g(\eps)=0$ such that 
\[
\P_\mu(\bp_\eta \leq \eps) \leq g(\eps)
\]
for every $\mu \in \cM$ and $\eps>0$. 

 Let $\mu \in \cM$, let $(\bp,\omega)$ be drawn from $\bP_\mu$, and let $X=(X_m)_{m\geq 0}$ be an independent random walk on $G$ independent of $(\bp,\mu)$. Let the modified configurations $(\omega^i)_{i\geq 0}$ be defined as in the proof of \cref{thm:finite_percolation}. Similarly, for each $\lambda>0$, $m\geq 1$, and $1 \leq i \leq m$, let the events $\sA_{\lambda,m}$, $\sB_{\lambda,m,i}$, and $\sC_{\lambda,i}$ be defined as in the proof of \cref{thm:finite_percolation} but replacing each instance of the phrase `touches at least $n$ edges' with `touches a set of edges of total weight at least $\lambda$'. Thus, we have as before that
\[
\sA_{\lambda,m} \subseteq \bigcup_{j=1}^m\sB_{\lambda,m,j} \quad \text{ and } \quad \sC_{\lambda,i} \supseteq \sB_{\lambda,m,i} \cap \{\omega(X_j,X_{j+1})=1 \text{ for every $0 \leq j \leq i-2$}\}
\]
for every $\lambda>0$, $m\geq 1$, and $1\leq i \leq m$. For each $\eps>0$ and $i\geq 1$ let $\sD_{\eps,i}$ be the event that $\bp_{X_{j,j+1}} \geq \eps$ 
 for every $0\leq j \leq i-2$. The events $\sB_{\lambda,m,i}$ and $\{\omega(X_j,X_{j+1})=1$ for every $0 \leq j \leq i-2\}$ are conditionally independent given $X$ and $\bp$, so that
\begin{align*}
\P\bigl(\sC_{\lambda,i} \cap \sD_{\eps,i}) &\geq \E \left[ \mathbbm{1}(\sD_{\eps,i})\P(\omega(X_{j,j+1})=1 \text{ for every $0\leq j \leq i-2$}\mid \bp,X) \P\bigl( \sB_{\lambda,m,i}\mid \bp,X)\right]\\
&
\geq \eps^{i-1} \P(\sB_{\lambda,m,i} \cap \sD_{\eps,i}) \geq \eps^{i-1} \P(\sB_{\lambda,m,i})-\eps^{i-1}\P(\sD_{\eps,i}^c) 
\geq \eps^{i-1} \P(\sB_{\lambda,m,i})-\eps^{i-1}(i-1)g(\eps)
\end{align*}
for every $\lambda,\eps>0$ and $m \geq i \geq 1$, where the final two inequalities follow by union bounds. Meanwhile, \cref{cor:two_ghost_S} and the assumption that $\sup_e J_e \leq 1$ imply that
\[
\P(\sC_{\lambda,i} \cap \sD_{\eps,i}) \leq 42 \sqrt{\frac{1-\eps}{\eps}} \frac{1}{\sqrt{\lambda}}
\]
for every $\lambda>0$ and $i\geq 1$.
Rearranging, we deduce that
\begin{multline*}
\P(\sA_{\lambda,m}) \leq \sum_{i=1}^m \left[\eps^{-i+1}\P(\sC_{\lambda,i} \cap \sD_{\eps,i}) + (i-1)g(\eps)\right]
\leq
\sum_{i=1}^m \left[42\eps^{-i+1} \sqrt{\frac{1-\eps}{\eps}}\frac{1}{\sqrt{\lambda}} + (i-1)g(\eps)\right]
\\
\leq 42 \eps^{-m} \sqrt{\frac{\eps}{1-\eps}}\frac{1}{\sqrt{\lambda}} + \binom{m}{2} g(\eps)
\end{multline*}
for every $\lambda,\eps>0$ and $m\geq 1$. On the other hand, using \cref{prop:Schramm} and the definition of the spectral radius as in \eqref{eq:using_rho} yields that
\[
\P(\sA_{\lambda,m}) \geq \P(\lambda \leq |E(K_o)|_J<\infty)^2 - \rho(\mu)^m-\rho(G)^m
\]
and hence that
\begin{equation}
\P_\mu(\lambda \leq |E(K_o)|_J<\infty)^2  \leq \rho(\mu)^m+\rho(G)^m+42 \eps^{-m} \sqrt{\frac{\eps}{1-\eps}}\frac{1}{\sqrt{\lambda}} + \binom{m}{2} g(\eps)
\end{equation}
for every $\lambda,\eps>0$ and $m\geq 1$. 

The claim now follows by appropriate choice of $\eps>0$ and $m\geq 1$: For each $\eps>0$ let $m(\eps)$ be maximal such that $\binom{m}{2} \leq g(\eps)^{-1/2}$ and for each $\lambda>0$ let $\eps(\lambda)>0$ be minimal such that $42 \eps^{-m(\eps)} \sqrt{\frac{\eps}{1-\eps}} \leq \lambda^{1/4}$. Then we have that  $\lim_{\eps \downarrow 0} m(\eps) = \infty$ and $\lim_{\lambda \uparrow \infty} \eps(\lambda)=0$. Thus, if we define $f:(0,\infty)\to (0,\infty)$ by
\[
f(\lambda)^2 = \sup_{\nu \in \cM} \rho(\nu)^{m(\eps(\lambda))} + \rho(G)^{m(\eps(\lambda))} + \lambda^{-1/4} + \sqrt{g(\eps(\lambda))}
\]
for every $\lambda>0$ then $f$ is decreasing, $\lim_{\lambda\uparrow \infty}f(\lambda)=0$, and 
\[
\P_\mu(\lambda \leq |E(K_o)|_J<\infty)  \leq f(\lambda)
\]
for every $\lambda>0$. The first claim follows since $\mu\in \cM$ was arbitrary.  
It follows in particular that
\[
\P_\mu( |E(K_o)|_J\geq \lambda)  \leq f(\lambda)
\]
for every $\mu\in \cM_\infty$ and $\lambda>0$. The portmanteau theorem implies that the same estimate holds for every $\mu \in \overline{\cM_{\infty}}$ and $\lambda>0$, completing the proof.
\end{proof}

\section{Analysis of the Ising model}

In this section we apply the technology developed in \cref{sec:free_energy} to prove our main theorems, \cref{thm:main,thm:main_simple,thm:main_continuity,thm:main_continuity_FK}. 


It will be notationally convenient throughout this section for us to consider both the Ising measures $\bI_{\beta,h}^\#$ and the \textbf{gradient Ising measures} $\bG_{\beta,h}^\#$, defined as follows. 
Let $G=(V,E,J)$ be an infinite, connected weighted graph.
For each $\sigma \in \{-1,1\}^V$ and $h\geq 0$, we define the \textbf{gradient} $\nabla_h \sigma \in \R^{E \cup V}$ by
$\nabla_h \sigma (e) = J_e \sigma_e = J_e \sigma_x \sigma_y$
for each $e\in E$ with endpoints $x$ and $y$ and 
$\nabla_h \sigma (v) = h \sigma_v$ 
for each $v\in V$. For each $\beta,h \geq 0$ we define $\bG_{\beta,h}^f$ and $\bG_{\beta,h}^w$ to be the push-forwards of $\bI_{\beta,h}^f$ and $\bI_{\beta,h}^+$ through the gradient $\nabla_h$. That is, if $\sigma$ is a random variable with law $\bI_{\beta,h}^+$ then the random variable $\nabla_h \sigma$ has law $\bG_{\beta,h}^w$, with a similar statement holding in the free case. The gradient Ising measure $\bG_{\beta,h}$ on a finite weighted graph is defined similarly.
When $h>0$ we can trivially recover $\sigma$ from $\nabla_h \sigma$, so that the two measures are just different ways of thinking about the same object. On the other hand, when $h=0$, $\nabla_0 \sigma$ only retains the even information about the configuration $\sigma$, so that the difference is more genuine, and the two measures can have rather different properties. For example, it is possible for $\bG^f_{\beta,0}$ to be a factor of i.i.d.\ in situations when $\bI^f_{\beta,0}$ is not even ergodic \cite{ray2019finitary}.  

\subsection{Double random currents and the loop $O(1)$ model}
\label{subsec:random_currents}

We now introduce the (double) random current and loop $O(1)$ models, referring the reader to \cite{MR3890455} for further background. The random current model was introduced by Griffiths, Hurst, and Sherman \cite{griffiths1970concavity} and developed extensively by Aizenman \cite{MR678000}. It has been of central importance to most modern work on the Ising model, with notable recent applications including \cite{MR4026609,MR4089494,aizenman2019marginal,MR4072233,duminil2015new}.


Let $G=(V,E,J)$ be a weighted graph with $V$ finite. A \textbf{current} $\mathbf{n}=(\mathbf{n}_e)_{e\in E}=(\mathbf{n}(e))_{e\in E}$ on $G$ is an assignment of non-negative integers to the edges of $G$. We write $\Omega_G$ for the set of currents on $G$. A vertex $v$ of $G$ is said to be a \textbf{source} of the current $\mathbf{n}$ if $\sum_{e\in E^\rightarrow_v} \mathbf{n}_e$ is odd, and the set of sources of $\mathbf{n}$ is denoted $\partial \mathbf{n}$. For each current $\mathbf{n}$ and $\beta >0$, we define
\[
w_\beta(\mathbf{n}) = \prod_{e\in E} \frac{(\beta J_e)^{\mathbf{n}_e}}{n_e!}.
\]
Many quantities of interest for the Ising model can be expressed in terms of sums over currents. For example, if $G=(V,E,J)$ is a weighted graph with $V$ finite and $x$ and $y$ are vertices of $G$ then
\begin{equation}
\label{eq:firstrandomcurrent}
\langle \sigma_x \sigma_y \rangle_{G,\beta,0} = \frac{\sum_{\mathbf{n}\in \Omega_G:\partial \mathbf{n}=\{x,y\}} w_\beta(\mathbf{n})}{\sum_{\mathbf{n}\in \Omega_G:\partial \mathbf{n}=\emptyset} w_\beta(\mathbf{n})} 
\end{equation}
for every $\beta\geq 0$. This formula becomes much more useful when combined with the following fundamental lemma of Griffiths, Hurst, and Sherman \cite{griffiths1970concavity}, known as the \emph{switching lemma}. We write `$x \xleftrightarrow{\bn_1+\bn_2} y$ in $H$' to mean that there exists a path connecting $x$ to $y$ in $H$ all of whose edges $e$ have $\bn_1(e)+\bn_2(e)\geq 1$.  The following statement of the switching lemma is adapted from \cite[Lemma 2.2]{MR3306602}\footnote{NB: The names of $G$ and $H$ are switched in this reference.}.

\begin{lemma}[Switching Lemma]
\label{lem:switching}
 Let $G$ be a weighted finite graph, let $H$ be a subgraph of $G$, let $x$ and $y$ be vertices of $H$, and 
let $A$ be a set of vertices of $G$. Then
\begin{multline*}
\sum_{\substack{\bn_1 \in \Omega_H : \partial \bn_1 =\{x,y\}\\\bn_2 \in \Omega_G : \partial\bn_2 =A}} F(\bn_1+\bn_2) w_\beta(\bn_1)w_\beta(\bn_2)
\\=
\sum_{\substack{\bn_1 \in \Omega_H : \partial \bn_1 =\emptyset\\\bn_2 \in \Omega_G : \partial\bn_2 =A\Delta\{x,y\}}} F(\bn_1+\bn_2) w_\beta(\bn_1)w_\beta(\bn_2) \mathbbm{1}\left(x \xleftrightarrow{\bn_1+\bn_2}y \text{ \emph{in} $H$}\right)
\end{multline*}
for every $F:\Omega_G\to [0,\infty]$ and $\beta\geq 0$, where $ A \Delta B = A \cup B \setminus A \cap B$ denotes the symmetric difference of two sets.
\end{lemma}

It follows in particular that if $G$ is a finite graph and $H$ is a subgraph of $G$ then
\begin{multline}
\langle \sigma_x \sigma_y \rangle_{H,\beta,0}\langle \sigma_x \sigma_y \rangle_{G,\beta,0}=\frac{\sum_{\mathbf{n_1}\in \Omega_H:\partial \mathbf{n_1}=\{x,y\}} w_\beta(\mathbf{n_1})}{\sum_{\mathbf{n_1}\in \Omega_H:\partial \mathbf{n_1}=\emptyset} w_\beta(\mathbf{n_1})}
\frac{\sum_{\mathbf{n_2}\in \Omega_G:\partial \mathbf{n_2}=\{x,y\}} w_\beta(\mathbf{n_2})}{\sum_{\mathbf{n_2}\in \Omega_G:\partial \mathbf{n_2}=\emptyset} w_\beta(\mathbf{n_2})}
\\
=\frac{\sum_{\mathbf{n_1}\in \Omega_H:\partial \mathbf{n_1}=\emptyset}\sum_{\mathbf{n_2}\in \Omega_G:\partial \mathbf{n_2}=\emptyset} w_\beta(\mathbf{n_1})w_\beta(\mathbf{n_2})\mathbbm{1}(x \xleftrightarrow{\mathbf{n_1}+\bn_2}y \text{ in $H$})}
{\sum_{\mathbf{n}\in \Omega_H:\partial \mathbf{n_1}=\emptyset} w_\beta(\mathbf{n_1})\sum_{\mathbf{n_2}\in \Omega_G:\partial \mathbf{n_2}=\emptyset} w_\beta(\mathbf{n_2})}
\label{eq:correlations_to_DRC0}
\end{multline}
for every two vertices $x$ and $y$ of $H$ and every $\beta\geq 0$.
This formula motivates the definitions of the \emph{random current} and \emph{double random current} models. Given a finite graph $G$ and $\beta \geq 0$, we define the \textbf{random current measure} $\mathbf{C}_{G,\beta}$ on $\Omega_G$ by setting
\[
\mathbf{C}_{G,\beta}(\{\mathbf{n}\}) = \frac{w_\beta(\mathbf{n})\mathbbm{1}(\partial \bn = \emptyset)}{\sum_{\bm \in \Omega_G}w_\beta(\mathbf{m})\mathbbm{1}(\partial \bm = \emptyset)}
\]
for each current $\bn \in \Omega_G$. The equality \eqref{eq:correlations_to_DRC0} can be rewritten succinctly in terms of this measure as follows: If $G$ is a finite graph and $H$ is a subgraph of $G$ then
\begin{align}
\langle \sigma_x \sigma_y \rangle_{H,\beta,0}\langle \sigma_x \sigma_y \rangle_{G,\beta,0}&=\mathbf{C}_{H,\beta}\otimes\mathbf{C}_{G,\beta}
\bigl(\bigl\{
(\bn_1,\bn_2) : x \xleftrightarrow{\bn_1+\bn_2} y
\text{ in $H$}
\bigr\}\bigr)
\label{eq:correlations_to_DRC}
\end{align}
for every two vertices $x$ and $y$ of $H$ and every $\beta\geq 0$. This equality leads us naturally to consider the double random current model (i.e., the measure $\mathbf{C}_{H,\beta}\otimes\mathbf{C}_{G,\beta}$) as a percolation model on $H$, where an edge is open if it takes a positive value in at least one of the two currents.

\medskip




Correlations of the Ising model in non-zero external field may also be expressed in terms of the random current model as follows. Let $G=(V,E,J)$ be a weighted graph with $V$ finite, let $h>0$, and let $\overline{G}_h$ be the graph obtained from $G$ by the addition of a special vertex $\partial$ that is connected to every vertex of $G$ by a single edge of weight $h$, so that the edge set of $\overline{G}_h$ is naturally identified with $V \cup E$. We observe that the gradient Ising measure with zero external field on $\overline{G}_h$ coincides with the gradient Ising measure on $G$ with external field $h$, and define the random current measure 
$\bC_{G,\beta,h}=\bC_{\overline{G}_h,\beta,0}$ for each $\beta\geq 0$, which we consider as a probability measure on $\N_0^E \times \N_0^V$.
Since every subgraph $H$ of $G$ is also a subgraph of $\overline{G}_h$, we deduce from \eqref{eq:correlations_to_DRC} that
\begin{align}
\langle \sigma_x \sigma_y \rangle_{H,\beta,0}\langle \sigma_x \sigma_y \rangle_{G,\beta,h}&=\mathbf{C}_{H,\beta,0}\otimes\mathbf{C}_{G,\beta,h}
\bigl(\bigl\{
(\bn_1,\bn_2) : x \xleftrightarrow{\bn_1+\bn_2} y
\text{ in $H$}
\bigr\}\bigr).
\label{eq:correlations_to_DRC2}
\end{align}
for every $\beta,h \geq 0$ and every two vertices $x$ and $y$ of $H$. For consistency, we will from now on consider $\bC_{G,\beta,0}$ as a probability measure on $\N_0^E \times \N_0^V$ that is supported on configurations in which $\bn_v=0$ for every $v\in V$.

\medskip

\textbf{Relation to the loop $O(1)$ model.}
Let $G=(V,E,J)$ be a weighted graph with $V$ finite and let $\beta, h\geq 0$. To ease notation, we will write $J_v=h$ for each vertex $v$ of $G$.
It follows from the definitions that if $\bn=(\bn_x)_{x\in V \cup E}$ is a random variable with law $\mathbf{C}_{G,\beta,h}$ then the values of $\bn$ are conditionally independent given the the field of parities $(\mathbbm{1}(\bn_x \text{ odd}))_{x\in V \cup E}$. Indeed, if we let $(\mathsf{Odd}_x)_{x\in V \cup E}$ and $(\mathsf{Even}_x)_{x\in V \cup E}$ be non-negative integer-valued random variables, independent of $\bn$, such that all the random variables $(\mathsf{Odd}_x)_{x\in V \cup E}$ and $(\mathsf{Even}_x)_{x\in V \cup E}$ are mutually independent with laws
\begin{equation}
\label{eq:odd_even_distribution}
\P(\mathsf{Odd}_x = n) = \frac{(\beta J_x)^n}{n! \sinh(\beta J_x)}\mathbbm{1}(n \text{ odd}) \quad \text{ and } \quad \P(\mathsf{Even}_x = n) = \frac{(\beta J_x)^n}{n! \cosh(\beta J_x)}\mathbbm{1}(n \text{ even})
\end{equation}
then
\[
(\bn_x)_{x\in V \cup E} \quad \text{ has the same distribution as } \quad \Bigl(\mathsf{Odd}_x\mathbbm{1}(\bn_x \text{ odd})+\mathsf{Even}_x\mathbbm{1}(\bn_x \text{ even})\Bigr)_{x\in V \cup E}.
\]
This leads to a special role for the sign field $(\mathbbm{1}(\bn_x \text{ odd}))_{x\in V\cup E}$. The law of this random variable under $\mathbf{C}_{G,\beta,h}$ is denoted by $\mathbf{L}_{G,\beta,h}$ and is known as the \textbf{loop $O(1)$ measure} on the graph $G$. 
The loop $O(1)$ measure on the weighted graph $G=(V,E,J)$ with $V$ finite can also be defined explicitly as 
the unique purely atomic probability measure on $\{0,1\}^{E \cup V}$ satisfying
\[
\mathbf{L}_{G,\beta,h}(\{\omega\}) \propto \mathbbm{1}(\partial \omega = \emptyset)\prod_{v\in V} \tanh(\beta h)^{\mathbbm{1}(\omega(v)=1)} \prod_{e\in E} \tanh(\beta J_e)^{\mathbbm{1}(\omega(e)=1)} 
\] 
for each $\omega \in \{0,1\}^{V \cup E}$,
  where $\partial \omega = \{v\in V: \omega(v)+\sum_{e\in E^\rightarrow_v} \omega(e)$ is odd$\}$.


A further very useful expression for the distribution of the loop $O(1)$ model in terms of the (gradient) Ising model, proven in \cite[Equations 2.10-2.12]{MR3306602}, states that if $G=(V,E,J)$ is a weighted graph with $V$ finite and $\beta,h \geq 0$ then
\begin{align}
\nonumber
\mathbf{L}_{G,\beta,h} (\{ \omega : \omega(x) = 0 \text{ for every } x\in A\}) &=  \Bigl\langle \exp\Bigl(-\beta \sum_{x\in A} J_x \sigma_{x}\Bigr) \Bigr\rangle_{G,\beta,h} \prod_{x\in A} \cosh(\beta J_x) \\
 &=C_{\beta,h}(A) \mathbf{G}_{G,\beta,h}\left[e^{\beta H_A}\right]
\label{eq:LoopO(1)Ising}
\end{align}
for every finite set $A \subseteq E \cup V$, where we set $H_A=H_A(\sigma)=-\sum_{x\in A} J_x \sigma_{x}$ and $C_{\beta,h}(A) = \prod_{x\in A} \cosh(\beta J_x)$ and  recall that we write $\sigma_e=\sigma_x \sigma_y$ for an edge $e$ with endpoints $x$ and $y$. 
Note that the equation \eqref{eq:LoopO(1)Ising} completely characterizes the measure $\mathbf{L}_{G,\beta,h}$ since, by inclusion-exclusion, the family of indicator functions $\{\mathbbm{1}(\omega(x) = 0 \text{ for every } x\in A) : A \subseteq E \cup V\}$ has linear span equal to the space of all functions from $\{0,1\}^{E \cup V}$ to $\R$ depending on at most finitely many edges and vertices.

\medskip

As observed in \cite{MR3306602}, the equation \eqref{eq:LoopO(1)Ising} allows us to deduce various statements about infinite volume limits of the loop $O(1)$ and random current models from the corresponding statements concerning the (gradient) Ising model. Indeed, let $G=(V,E,J)$ be an infinite connected weighted graph, let $(V_n)_{n\geq 0}$ be an exhaustion of $G$, and let $(G_n)_{n\geq 0}$ and $(G_n^*)_{n\geq 0}$ be defined as in \cref{subsec:intro_definitions}. It follows from \eqref{eq:LoopO(1)Ising} together with the corresponding statements for the Ising model that the weak limits
\begin{align*}
\mathbf{L}_{G,\beta,h}^f &= \wlim_{n\to\infty} \mathbf{L}_{G_n,\beta,h},  &\mathbf{L}_{G,\beta,h}^w &= \wlim_{n\to\infty} \mathbf{L}_{G_n^*,\beta,h},\\
\mathbf{C}_{G,\beta,h}^f &= \wlim_{n\to\infty} \mathbf{C}_{G_n,\beta,h}, \text{ and } &\mathbf{C}_{G,\beta,h}^w &= \wlim_{n\to\infty} \mathbf{C}_{G_n^*,\beta,h}
\end{align*}
all exist and do not depend on the choice of exhaustion. 
 Moreover, these infinite volume loop $O(1)$ measures are related to the infinite volume Ising measures by the relations
\begin{align}
\mathbf{L}_{G,\beta,h}^\# (\{ \omega : \omega(x) = 0 \text{ for every } x\in A\}) =C_{\beta,h}(A)\mathbf{G}_{\beta,h}^\#\left[e^{\beta H_A}\right]
\label{eq:LoopO(1)Ising_infinite},
\end{align}
which holds for every finite set $A \subseteq V \cup E$, $\beta,h\geq 0$, and $\# \in \{f,w\}$. As in the finite case, the equation \eqref{eq:LoopO(1)Ising_infinite} completely determines the measures $\mathbf{L}_{G,\beta}^\#$ since, by inclusion-exclusion, the family of indicator functions $\{\mathbbm{1}(\omega(x) = 0 \text{ for every } x\in A) : A \subseteq V \cup E \text{ is finite}\}$ has linear span equal to the set of all functions depending on at most finitely many edges. 

\medskip

It follows by taking limits over exhaustions that, as in the finite case, we may obtain a random variable with law $\bC_{\beta,h}^\#$ as follows: Fix $\beta >0, h\geq 0$, and $\#\in \{f,w\}$. Let $L=(L_x)_{x\in V \cup E}$ be a random variable with law $\bL_{\beta,h}^\#$. Independently of $L$, let $\mathsf{Odd}=(\mathsf{Odd}_x)_{x\in V \cup E}$ and $\mathsf{Even}=(\mathsf{Even}_x)_{x\in V \cup E}$ be independent random variables with distributions given as in \eqref{eq:odd_even_distribution}. Then the random variable
\[
\bn=(\bn_x)_{x\in V \cup E}= \Bigl(\mathsf{Odd}_x\mathbbm{1}(L_x=1)+\mathsf{Even}_x\mathbbm{1}(L_x=0)\Bigr)_{x\in V \cup E}
\] 
has law $\bC_{\beta,h}^\#$. It follows in particular that if $G$ is transitive then $\bC_{\beta,h}^\#$ may be expressed as an $\Aut(G)$-factor of $\bL_{\beta,h}^\# \otimes \mu$ where $\mu$ is an appropriately chosen Bernoulli measure on $(\{0,1,\ldots\}^2)^{V \cup E}$. Moreover, this representation allows us to consider the random subgraph of $G$ spanned by those edges with a non-zero current as a percolation in random environment model, where the environment $\bp$ is defined in terms of the loop $O(1)$ configuration $\omega$ by 
\[\bp_e = \mathbbm{1}(\omega(e)=1) + \frac{\cosh(\beta J_e)-1}{\cosh(\beta J_e)} \mathbbm{1}(\omega(e)=0). \]
This representation will allow us to apply the machinery of \cref{sec:free_energy} to the random current and double random current models once we have bounded their spectral radius in \cref{subsec:randomcurrentspectrum}. 

\medskip

\textbf{The Lupu--Werner coupling}. In addition to the indirect connection between the random current model and the FK-Ising model via the Ising model and the Edwards--Sokal coupling, there is also a direct probabilistic connection between these two models due to Lupu and Werner \cite{MR3485382}. They proved that for each $\beta,h\geq 0$ and $\#\in \{f,w\}$ it is possible to obtain a sample of $\phi_{2,\beta,h}^\#$ as the union of a sample of the random current model $\bC_{\beta,h}^\#$ with an independent Bernoulli  process in which $x \in E \cup V$ is included independently at random with inclusion probability $1-e^{-\beta J_x}$, where we set $J_x=h$ for every $x\in V$. 
Combining this relationship with the relationship between the random current and loop $O(1)$ models discussed above, it follows that for each $\beta,h\geq 0$ and $\#\in \{f,w\}$ it is possible to obtain a sample of $\phi_{2,\beta,h}^\#$ as the union of a sample of the loop $O(1)$ model $\bL_{\beta,h}^\#$ with an independent Bernoulli process 
with  inclusion probabilities
\[
 1 - \frac{1}{\cosh(\beta J_x)}+\frac{1-e^{-\beta J_x}}{\cosh(\beta J_x)}=\tanh(\beta J_x).
\]
Together with the formula \eqref{eq:LoopO(1)Ising_infinite}, this implies that
\begin{align}
\phi^\#_{2,\beta,h}(\omega(x)=0 \text{ for all $x\in A$}) &= \bL^\#_{\beta,h}\bigl(\omega(x)=0 \text{ for all $x\in A$}\bigr)\prod_{x \in A} \left( 1-\tanh(\beta J_x)\right)
\nonumber
\\
&=\bG^\#_{\beta,h} \left[e^{\beta H_A} \right]\prod_{x \in A} e^{-\beta J_x} = 
\bG^\#_{\beta,h} \left[e^{-\beta \sum_{x \in A} J_x(\sigma_x + 1) } \right]
\label{eq:FK_Gradient_Formula}
\end{align}
for every finite set $A \subseteq E \cup V$, where we write $J_v=h$ for each $v\in V$ and recall that $H_A =H_A(\sigma)= -\sum_{x\in A} J_x \sigma_x$ for each finite set $A \subseteq E \cup V$. This formula, which can also be proven using the percolation-in-random-environment representation of the random cluster model used in \cref{subsec:finite_FK},  establishes a relationship between the gradient Ising and FK-Ising models that has better continuity properties than the Edwards--Sokal coupling, and will be very useful throughout our analysis.

Since $\phi_{2,\beta,h}^\#$ is stochastically dominated by the product measure $\phi_{1,\beta,h}^\#$ \cite[Theorem 3.21]{GrimFKbook}, it follows from the Lupu--Werner coupling that the measures $\bL_{\beta,h}^\#$ and $\bC_{\beta,h}^\#$ are stochastically dominated by this product measure also.
We deduce in particular that
\begin{equation}
\label{eq:bounding_current_by_FK1}
\bL_{\beta,h}^\#(e \text{ open})\leq \bC_{\beta,h}^\#(\bn(e)>0) \leq \phi_{2,\beta,h}^\#(e \text{ open}) \leq \frac{e^{2\beta J_e}-1}{e^{2\beta J_e}} \leq 2\beta J_e
\end{equation}
for every $e\in E$, $\beta,h\geq 0$, and $\#\in \{f,w\}$, and similarly that
\begin{equation}
\label{eq:bounding_current_by_FK2}
\bC_{\beta,h}^\#(\bn(v)>0) \leq \phi_{2,\beta,h}^\#(\omega(v)=1) \leq \frac{e^{2\beta h}-1}{e^{2\beta h}} \leq 2\beta h
\end{equation}
$v\in V$, $\beta,h\geq 0$, and $\#\in \{f,w\}$.

\subsection{The spectral radius of the random current model}
\label{subsec:randomcurrentspectrum}

We now apply \eqref{eq:LoopO(1)Ising_infinite} to bound the spectral radii of the loop $O(1)$ and random current models on a transitive weighted graph. Although very little is known about whether or not the loop $O(1)$ and random current models are factors of i.i.d.\ (see \cite{harel2018finitary} for some discussion), 
%
 the connections between the (gradient) Ising model and these models are strong enough to carry through bounds on the spectral radius without needing to express anything as a factor.

\begin{thm}
\label{thm:spectralradius}
Let $G=(V,E,J)$ be a connected transitive weighted graph and let $\Gamma$ be a closed unimodular transitive group of automorphisms. Then 
\[
\rho(\mathbf{L}^\#_{\beta,h}) \leq \max\left\{\rho(\mathbf{G}^\#_{\beta,h}),\rho(G)\right\}
\]
for every $\beta,h \geq 0$ and $\# \in \{w,f\}$.
\end{thm}

\begin{proof}[Proof of \cref{thm:spectralradius}]
Fix $\beta,h\geq 0$, and $\# \in \{f,w\}$. To ease notation, we write $\bL=\bL_{\beta,h}^\#$ and $\bG=\bG_{\beta,h}^\#$. Let $X=(X_n)_{n\geq 0}$ be the random walk on $G$ and let $\hat X = (\hat X_n)_{n\geq 0}$ be the associated random walk on $\Gamma$ as defined in \cref{subsec:spectral_background}. We write $\bE$ for expectations taken with respect to the law of $\hat X$ and write $\E$ for expectations taken with respect to the product measure $\bE \otimes \bG$.  Given $A \subseteq V \cup E$ and $\omega \in \{0,1\}^{V \cup E}$, we write $A \perp \omega$ to mean that $\omega(x)=0$ for every $x\in A$.
Inclusion-exclusion implies that events of the form $\{\omega : A \perp \omega \}$ with $A$ finite have dense linear span in $L^2(\{0,1\}^{V \cup E},\bL)$. Thus, by \eqref{eq:rho_limit_covariance}, it suffices to prove that
\begin{equation}
\limsup_{k\to\infty} \left|\bE\left[\bL(\{\omega : A \perp \omega \text{ and } \hat X_{2k}^{-1} A \perp \omega\}) \right] -\bL(\{\omega : A  \perp \omega\})^2 \right|^{1/2k} \\\leq \max\{\rho(\bG),\rho(G)\}
\label{eq:spectralradiusclaim}
\end{equation}
for every finite set $A \subseteq V \cup E$.
Fix one such finite set $A \subseteq E$ and write $A_k=\hat X_{k}^{-1} A$ for every $k\geq 0$. 
Then we have by \eqref{eq:LoopO(1)Ising} that
\begin{align}
\bE\left[\bL(\{\omega : A \perp \omega\})\right]
 &=  \E\left[C_{\beta,h}(A ) \exp\left(\beta H_{A}\right)\right] &&\text{ and}
 \label{eq:spectral-2}
\\
\bE\left[\bL(\{\omega : A \perp \omega \text{ and } \hat X_{2k}^{-1} A \perp \omega\})\right]
 &=  \E\left[C_{\beta,h}(A \cup A_{2k}) \exp\left(\beta H_{A\cup A_{2k}}\right)\right]&&
 \label{eq:spectral-1}
 \end{align}
 for every $k\geq 0$ 
 and hence that
 \begin{align*}
 \bE\left[\bL(\{\omega : A \perp \omega \text{ and } \hat X_{2k}^{-1} A \perp \omega\})\right]
=\E\left[ C_{\beta,h}(A) \exp\left(\beta H_{A}\right) C_{\beta,h}(A_{2k})  \exp\left(\beta H_{A_{2k}}\right)\right]\phantom{.}
\nonumber
\\
+ \E\left[ C_{\beta,h}(A \cup A_{2k}) \exp\left(\beta H_{A\cup A_{2k}}\right)  \mathbbm{1}(A \cap A_{2k} \neq \emptyset)\right]\phantom{.}
\nonumber
\\
-\E\left[ C_{\beta,h}(A) \exp\left(\beta H_{A}\right) C_{\beta,h}(A_{2k})  \exp\left(\beta H_{A_{2k}}\right)\mathbbm{1}(A \cap A_{2k} \neq \emptyset)\right].
\end{align*}
(Note that $C_{\beta,h}(A)=C_{\beta,h}(A_{2k})$ is a constant.)
Since the random variables $ e^{\beta H_{A\cup A_{2k}}} C_\beta(A \cup A_{2k})$ and $ e^{\beta H_{A}}  C_\beta(A) e^{\beta H_{A_{2k}}} C_\beta(A_{2k})$ are both bounded between two positive constants (depending on $A$, $\beta$, and $h$ but not $k$), the second and third terms satisfy
\begin{multline}
\limsup_{k\to\infty} \E\left[ C_{\beta,h}(A \cup A_{2k}) \exp\left(\beta H_{A\cup A_{2k}}\right)  \mathbbm{1}(A \cap A_{2k} \neq \emptyset)\right]^{1/2k} \\ = \limsup_{k\to\infty} \P\left(\hat X_{2k}^{-1} A \cap A \neq \emptyset\right)^{1/2k}\leq \|\hat P\|= \rho(G)
\label{eq:spectral1}
\end{multline}
and
\begin{multline}
\label{eq:spectral2}
\limsup_{k\to\infty}\E\left[ C_{\beta,h}(A) \exp\left(\beta H_{A}\right) C_{\beta,h}(A_{2k})  \exp\left(\beta H_{A_{2k}}\right)\mathbbm{1}(A \cap A_{2k} \neq \emptyset)\right]^{1/2k} \\ = \limsup_{k\to\infty} \P\left(\hat X_{2k}^{-1} A \cap A \neq \emptyset\right)^{1/2k}\leq \|\hat P\|=\rho(G).
\end{multline}
Meanwhile, we also have by definition of the spectral radius that
\begin{multline*}
\Bigl|\E\left[ C_{\beta,h}(A) \exp\left(\beta H_{A}\right) C_{\beta,h}(A_{2k})  \exp\left(\beta H_{A_{2k}}\right)\right]
\\-
\E\left[ C_{\beta,h}(A) \exp\left(\beta H_{A}\right)\right]^2\Bigr| \leq \rho(\bG)^{2k} \Var\left(C_{\beta,h}(A)e^{\beta H_A}\right)
\end{multline*}
for each $k\geq 0$ and hence that
\begin{multline*}
\limsup_{k\to\infty}\Bigl|\E\left[ C_{\beta,h}(A) \exp\left(\beta H_{A}\right) C_{\beta,h}(A_{2k})  \exp\left(\beta H_{A_{2k}}\right)\right]
-
\E\left[ C_{\beta,h}(A) \exp\left(\beta H_{A}\right)\right]^2\Bigr|^{1/2k} \leq \rho(\bG).
\end{multline*}
Applying this estimate together with those of \eqref{eq:spectral1} and \eqref{eq:spectral2}
 in light of \eqref{eq:spectral-2} and \eqref{eq:spectral-1}  yields the claimed inequality \eqref{eq:spectralradiusclaim}.
\end{proof}

Since $\bC_{\beta,h}^\#$ may be expressed as an $\Aut(G)$-factor of $\bL_{\beta,h}^\# \otimes \mu$ where $\mu$ is an appropriately chosen Bernoulli measure on $(\{0,1,\ldots\}^2)^{V \cup E}$, the following corollary follows immediately from \cref{thm:Ising_factor,thm:spectralradius,thm:Bernoulli_radius}.

\begin{corollary}
\label{cor:current_radius}
Let $G=(V,E,J)$ be a connected transitive weighted graph and let $\Gamma$ be a closed unimodular transitive group of automorphisms. Then 
\[
\rho(\mathbf{C}^w_{\beta,h})\leq \rho(G) \qquad \text{ and } \qquad \rho(\mathbf{L}^w_{\beta,h})\leq \rho(G)
\]
for every $\beta,h\geq 0$.
\end{corollary}


\begin{remark}
In \cref{subsec:free_spectral_radius}, we use \eqref{eq:FK_Gradient_Formula} to show that the conclusions of \cref{cor:current_radius} can in fact be extended to the case of free boundary conditions.
With slightly more work one can show that the inequalities in \cref{thm:spectralradius,cor:current_radius,thm:free_spectral_radius} are equalities when $\beta>0$.
\end{remark}

\subsection{Double random currents with mismatched temperatures}
\label{subsec:mismatched}

In this section we discuss how the switching lemma can be used to study pairs of random currents with \emph{different} values of the inverse temperature $\beta$ and external field $h$. The resulting tools, which appear to be new, lead to quantitative versions of the arguments of \cite{MR3306602} that will be used to control the effect of changing $\beta$ and $h$ on the Ising model in the proofs of our main theorems.

Let $G=(V,E,J)$ be a connected weighted graph with $V$ finite. For each $\theta \in (0,1)$, let $\tilde G_\theta$ be the weighted graph obtained from $G$ by replacing each edge $e$ of $G$ by two edges $e_1$ and $e_2$ in parallel, where $e_1$ has coupling constant $J_{e_1}:=(1-\theta)J_e$ and $e_2$ has coupling constant $J_{e_2}:=\theta J_e$.
Observe from the definitions that 
\begin{equation}
\label{eq:theta_Ising}
\bI_{\tilde G_\theta,\beta,h} = \bI_{G,\beta,h}
\end{equation}
for every $\theta\in (0,1)$ and $\beta,h\geq 0$. There is also a simple probabilistic relationship between the random current models on $G$ and $\tilde G_\theta$: Let $\mathbf{m}=(\mathbf{m}_x)_{x\in V \cup E}$ be a random variable with law $\bC_{G,\beta,h}$. Conditional on $\mathbf{m}$, for each $e \in E$ let $\mathbf{n}_{e_1}$ be a binomial random variable with distribution $\operatorname{Binom}(1-\theta,\mathbf{m}_e)$ and let $\mathbf{n}_{e_2}=\mathbf{m}_e - \mathbf{n}_{e_1}$, where we take the random variables $(\mathbf{n}_{e_1})_{e\in E}$ to be conditionally independent given $\mathbf{m}$. Finally, set $\mathbf{n}_v=\mathbf{m}_v$ for each $v\in V$. It follows from the identity
\begin{equation}
\label{eq:binomial_identity}
\sum_{k=0}^n \frac{((1-\theta)\beta)^k}{k!}\frac{(\theta \beta)^{n-k}}{(n-k)!} =
\frac{\beta^n}{n!}\sum_{k=0}^n \binom{n}{k}(1-\theta)^k\theta^{n-k}
= \frac{\beta^n}{n!}
\end{equation}
that $\mathbf{n}$ has law $\bC_{\tilde G_\theta, \beta, h}$. 
Now consider the subgraph $H_\theta$ of $\tilde G_\theta$ spanned by the edges $\{e_1:e\in E\}$, so that $H_\theta$ is isomorphic to the weighted graph obtained from $G$ by multiplying all coupling constants by $(1-\theta)$. Observe that if we identify the edge set of $H_\theta$ with that of $G$ then we have
 the equalities
\[
\bI_{H_\theta,\beta,h} = \bI_{G,(1-\theta)\beta,(1-\theta)^{-1}h}
\qquad \text{ and } \qquad
\bC_{H_\theta,\beta,h} = \bC_{G,(1-\theta)\beta,(1-\theta)^{-1}h}
\]
for every $\beta,h \geq 0$. Since $H_\theta$ is a subgraph of $\tilde G_\theta$, this construction will therefore allow us to apply the switching lemma \cref{lem:switching} to study the Ising model at two different values of the inverse temperature.

We now introduce an infinite-volume version of this construction in the wired case, in which we will also allow ourselves to change the strength of the external field. Let $G=(V,E,J)$ be an infinite, connected, weighted graph and let $\beta_2 \geq \beta_1 \geq 0$ and $h_2 \geq h_1 \geq 0$. Let $\bn_1=(\bn_1(x))_{x\in E \cup V}$ and $\bm_2=(\bm_2(x))_{x\in E\cup V}$ be independent random variables with laws $\bC^w_{\beta_1,h_1}$ and $\bC^w_{\beta_2,h_2}$ respectively. Let $\theta=1-\beta_1/\beta_2$ and let $\phi=1-(\beta_1 h_1)/(\beta_2h_2)$, so that $\theta,\phi\in [0,1]$. Conditional on $\bn_1$ and $\bm_2$, let $\bn_2 =(\bn_2(x_i))_{x\in E \cup V, i \in \{0,1\}}$ be defined as follows:
\begin{enumerate}
\item For each $e\in E$, let $\bn_2(e_1)$ be a $\operatorname{Binom}(1-\theta,\bm_2(e))$ random variable and let $\bn_2(e_2)=\bm_2(e)-\bn_2(e_1)$.
\item For each $v\in V$, let $\bn_2(v_1)$ be a $\operatorname{Binom}(1-\phi,\bn_2(v))$ random variable and let $\bn_2(v_2)=\bm_2(v)-\bn_2(v_1)$.
\end{enumerate}
We take the random variables $(\bn_2(x_1):x \in E \cup V)$ to be conditionally independent of each other and of $\bn_1$ given $\bm_2$. Let $\bQ_{\beta_1,h_1,\beta_2,h_2}$ denote the law of the resulting pair of random variables $(\bn_1,\bn_2)$, which we may think of as a measure on the product space $
\N_0^E \times \N_0^V \times (\N_0^2)^E \times (\N_0^2)^V \cong (\N_0^3)^E \times (\N_0^3)^V$.

We now relate the dependence on $\beta$ and $h$ of correlations in the Ising model to the percolative properties of $\bQ_{\beta_1,h_1,\beta_2,h_2}$. Let $(\bn_1,\bn_2)$ be a pair of random variables with law $\bQ_{\beta_1,h_1,\beta_2,h_2}$, and let $u$ and $v$ be vertices of $G$. We say that an edge or vertex $x$ of $G$ is $1$\textbf{-open} if $\bn_1(x)+\bn_2(x_1)>0$ and that $x$ is $2$\textbf{-open} if $\bn_1(x)+\bn_2(x_1)+\bn_2(x_2)>0$. In particular, every $1$-open $x$ is also $2$-open. For each vertex $v$ of $G$ and $i\in \{1,2\}$, we define $K^i_v$ to be the set of vertices that are connected to $v$ by an $i$-open path, so that $K^1_v \subseteq K^2_v$ for every $v\in V$. We say that a set $A \subseteq V$ is \textbf{$i$-infinite} if it is infinite or contains an $i$-open vertex.

For each $e\in E$, $v \in V$, and $i\in \{1,2\}$, let $K^1_{v,e}$ be the set of vertices that are connected to $v$ by an $i$-open path that does not include the edge $e$. Let $x$ and $y$ be the endpoints of $e$ and define
 $\sA_e$ to be the event that the following hold:
\begin{enumerate}
  \item
 $\bn_1(e)=0$ and $\bn_2(e_1)=1$,
 \item  $x$ and $y$ are not connected by any $1$-open path that does not include $e$ and the clusters $K^1_{x,e}$ and $K^1_{y,e}$ are not both $1$-infinite.
 \item At least one of the following hold:
 \begin{enumerate}
\item $\bn_2(e_2)>0$,
\item $x$ and $y$ are connected by a $2$-open path that does not include the edge $e$, or
\item $K^2_{x,e}$ and $K^2_{y,e}$ are both $2$-infinite.
 \end{enumerate}
\end{enumerate}
Intuitively, $\sA_e$ is the event that $x$ and $y$ are connected `through infinity' off of $e_1$ by $2$-open edges and vertices but are not connected through infinity off of $e$ by $1$-open edges and vertices. Similarly, for each $v\in V$ we define $\sA_v$ to be the event that the following hold:
\begin{enumerate}
  \item
 $\bn_1(v)=0$ and $\bn_2(v_1)=1$,
 \item  $K_v^1 \setminus \{v\}$ is not $1$-infinite.
 \item Either $\bn_2(v_2)>1$ or $K^1_v \setminus \{v\}$ is $2$-infinite.
\end{enumerate}
Note that if $G$ is locally finite then the set $\sA_x$ is closed for every $x\in E \cup V$.

\begin{prop}
\label{cor:gradient}
Let $G=(V,E,J)$ be an infinite, connected, transitive weighted graph. Then
\[
\beta_1 J_x \left(\langle \sigma_x \rangle_{\beta_2,h_2}^+ - \langle \sigma_x \rangle_{\beta_1,h_1}^+\right) \leq \bQ_{\beta_1,h_1,\beta_2,h_2}(\sA_x)
\]
for each $\beta_2 \geq \beta_1 \geq 0$, $h_2 \geq h_1 \geq 0$, and $x\in E \cup V$, where we set $J_x=h_1$ if $x\in V$.
\end{prop}

\begin{remark}
With a little more work this inequality can be shown to be an equality.
\end{remark}

We will deduce this proposition from the following general lemma, which is similar to \cite[Lemma 4.5]{1707.00520} and \cite[Eq.\ 3.10]{MR3306602}.

\begin{lemma}
\label{lem:general_gradient}
Let $G=(V,E,J)$ be a weighted graph with $V$ finite, let $H$ be a subgraph of $G$, and let $e$ be an edge of $H$ with endpoints $x$ and $y$. Then
\[\beta J_e \left(\langle \sigma_e \rangle_{G,\beta,0}- \langle \sigma_e \rangle_{H,\beta,0}\right) =  \mathbf{C}_{H,\beta}\otimes\mathbf{C}_{G,\beta}
\bigl(\sB_e\bigr)\]
for every $\beta \geq 0$, 
where $\sB_e$ is the set of pairs $(\bn_1,\bn_2)$ such that $\bn_1(e)=0$, $\bn_2(e)=1$,  and $x$ and $y$ are connected to each other by an $(\bn_1+\bn_2)$-open path in $G \setminus \{e\}$ but not in $H\setminus \{e\}$.
\end{lemma}

\begin{proof}
We may assume that $\beta>0$, the claim being trivial otherwise.
We have by \eqref{eq:firstrandomcurrent} that
\begin{align*}
\langle \sigma_x \sigma_y \rangle_{G,\beta,0}- \langle \sigma_x \sigma_y \rangle_{H,\beta,0} &= 
\frac{\sum_{\mathbf{n_2}\in \Omega_G:\partial \mathbf{n_2}=\{x,y\}} w_\beta(\mathbf{n_2})}{\sum_{\mathbf{n_2}\in \Omega_G:\partial \mathbf{n_2}=\emptyset} w_\beta(\mathbf{n_2})}
-
\frac{\sum_{\mathbf{n_1}\in \Omega_H:\partial \mathbf{n_1}=\{x,y\}} w_\beta(\mathbf{n_1})}{\sum_{\mathbf{n_1}\in \Omega_H:\partial \mathbf{n_1}=\emptyset} w_\beta(\mathbf{n_1})}
\\
&=
\frac{
\sum_{
\substack{
\mathbf{n_1}\in \Omega_H:\partial \mathbf{n_1}=\emptyset 
\\
\mathbf{n_2}\in \Omega_G:\partial \mathbf{n_2}=\{x,y\}
}}
 w_\beta(\bn_1)w_\beta(\bn_2)
-
\sum_{
\substack{
\mathbf{n_1}\in \Omega_H:\partial \mathbf{n_1}=\{x,y\} 
\\
\mathbf{n_2}\in \Omega_G:\partial \mathbf{n_2}=\emptyset
}}
 w_\beta(\mathbf{n_1})w_\beta(\bn_2)
}
{\sum_{\mathbf{n_1}\in \Omega_H:\partial \mathbf{n_1}=\emptyset} \sum_{\mathbf{n_2}\in \Omega_G:\partial \mathbf{n_2}=\emptyset} w_\beta(\mathbf{n_2}) w_\beta(\mathbf{n_1})},
\end{align*}
and applying the switching lemma to the second term yields that
\begin{align*}
\langle \sigma_x \sigma_y \rangle_{G,\beta,0}- \langle \sigma_x \sigma_y \rangle_{H,\beta,0}&=
\genfrac{}{}{}{}{\raisebox{1.1em}{$\sum_{
\substack{
\mathbf{n_1}\in \Omega_H:\partial \mathbf{n_1}=\emptyset 
\\
\mathbf{n_2}\in \Omega_G:\partial \mathbf{n_2}=\{x,y\}
}}
 w_\beta(\bn_1)w_\beta(\bn_2)
\mathbbm{1}\Bigl(x \nxleftrightarrow{\bn_1+\bn_2} y \text{ in $H$}\Bigr)
$}}{\raisebox{-0em}{$\sum_{\mathbf{n_1}\in \Omega_H:\partial \mathbf{n_1}=\emptyset} \sum_{\mathbf{n_2}\in \Omega_G:\partial \mathbf{n_2}=\emptyset} w_\beta(\mathbf{n_2}) w_\beta(\mathbf{n_1})$}}.
\end{align*}
 (This equality holds for all vertices $x$ and $y$.) The constraint ``$\partial \bn_2 = \{x,y\}$'' forces $x$ and $y$ to be connected in $\bn_2$, while the constraint ``$x \nxleftrightarrow{\bn_1+\bn_2} y$ in $H$" forces $\bn_1(e)=\bn_2(e)=0$. Thus, if we consider the two sets
\[
A = \{(\bn_1,\bn_2) \in \Omega_H \times \Omega_G : \partial \bn_1 = \emptyset,\, \partial \bn_2 = \{x,y\},\, \text{ and } x \nxleftrightarrow{\bn_1+\bn_2} y \text{ in $H$}\}
\]
and
\begin{multline*}
B = \{(\bn_1,\bn_2) \in \Omega_H \times \Omega_G : \partial \bn_1 = \emptyset,\, \partial \bn_2 = \emptyset,\, x \nxleftrightarrow{\bn_1+\bn_2} y \text{ in $H \setminus \{e\}$},\\ x \xleftrightarrow{\bn_1+\bn_2} y \text{ in $G \setminus \{e\}$},\, \bn_1(e)=0, \text{ and } \bn_2(e)=1\}
\end{multline*}
then we have a bijection $A \to B$ given by incrementing the value of $\bn_2(e)$ from $0$ to $1$. This increment changes the weight $w_\beta(\bn_2)$ by a factor of $\beta J_e$, so that
\begin{align*}
\langle \sigma_x \sigma_y \rangle_{G,\beta,0}- \langle \sigma_x \sigma_y \rangle_{H,\beta,0}&=
\frac{\sum_{\bn_1,\bn_2 \in A}
 w_\beta(\bn_1)w_\beta(\bn_2)}
 {\sum_{\mathbf{n_1}\in \Omega_H:\partial \mathbf{n_1}=\emptyset} \sum_{\mathbf{n_2}\in \Omega_G:\partial \mathbf{n_2}=\emptyset} w_\beta(\mathbf{n_2}) w_\beta(\mathbf{n_1})}\\
 &=
 \frac{\sum_{\bn_1,\bn_2 \in B}
 w_\beta(\bn_1)w_\beta(\bn_2)}
 {\beta J_e\sum_{\mathbf{n_1}\in \Omega_H:\partial \mathbf{n_1}=\emptyset} \sum_{\mathbf{n_2}\in \Omega_G:\partial \mathbf{n_2}=\emptyset} w_\beta(\mathbf{n_2}) w_\beta(\mathbf{n_1})}.
\end{align*}This is equivalent to the claim.
\end{proof}

\begin{proof}[Proof of \cref{cor:gradient}]
We may assume that $\beta>0$, the claim being trivial otherwise.
We prove the formula in the case that $x$ is an edge of $G$, the case that $x$ is a vertex being similar. 
We will also assume for convenience that $h_2>h_1$ and $\beta_2>\beta_1$; the case of equality is similar but requires one to define the graphs $\tilde G_\theta$ and $H_\theta$ differently to avoid having edges of weight zero, which were not allowed by the definition of a weighted graph. (This does not cause any actual problems in the proof.)

Fix $e\in E$, $\beta_2 > \beta_1 > 0$, and $h_2 > h_1 \geq 0$. Let $\theta=1-\beta_1/\beta_2$, let $\phi=1-\beta_1h_1/\beta_2h_2$, and let $\bQ=\bQ_{\beta_1,h_1,\beta_2,h_2}$. Let $\tilde G_\theta$ be the weighted graph obtained from $G$ by replacing each edge $e$ of $G$ by parallel edges $e_1$ and $e_2$, where $e_1$ has coupling constant $J_{e_1}:=(1-\theta)J_e$ and $e_2$ has coupling constant $J_{e_2}:=\theta J_e$, and let $H_\theta$ be the subgraph of $G_\theta$ spanned by the edges $\{e_1 :e \in E\}$. As above, it follows from the definitions and  \eqref{eq:binomial_identity} that a random variable $(\bn_1,\bn_2)$ with law $\bQ_{\beta_1,h_1,\beta_2,h_2}$ can be obtained from a random variable $(\bn_1,\bm_2)$ with law $\bC^w_{H_\theta,\beta_2,(1-\phi) h_2} \otimes \bC^w_{G_\theta,\beta_2,h_2}$ by setting $\bn_2(e_i)=\bm_2(e_i)$ for each $e\in E$ and $i\in \{1,2\}$, taking  $(\bn_2(v_1))_{v\in V}$ to be independent Binomial random variables conditioned on $(\bn_1,\bm_2)$, and setting $\bn_2(v_2)=\bm_2(v)-\bn_2(v_1)$ for each $v\in V$.

Let $(V_n)_{n\geq 1}$ be an exhaustion of $V$ such that both endpoints of $x$ belong to $V_n$ for every $n\geq 1$, and let $(G_n^*)_{n\geq 1}$ be formed by contracting each vertex in $V\setminus V_n$ into a single vertex $\delta_n$ as in \cref{subsec:intro_definitions}. For each $n\geq 1$, let the weighted graph  $\tilde G_n$ be obtained from $G_n^*$ as follows:
\begin{enumerate}
  \item  Replace each edge $e$ of $G_n^*$ with two parallel edges $e_1$ and $e_2$ with coupling constants $(1-\theta)J_e$ and $\theta J_e$.
  \item Add an additional vertex $\delta$ distinct from $\delta_n$. 
\begin{enumerate}
  \item If $h_1=0$, attach each vertex $v$ of $G_n^*$ (including $\delta_n$) to $\delta$ by a single edge of weight $h_2$. Call this edge $g_2(v)$.
  \item If $h_1>0$, attach each vertex $v$ of $G_n^*$ (including $\delta_n$) to $\delta$ by two edges in parallel. The first edge is called $g_1(v)$ and is given weight $(1-\phi)h_2=\beta_1 h_1/\beta_2$, while the second edge is called $g_2(v)$ and is given weight $\phi h_2$. 
\end{enumerate}
\end{enumerate}
If $h_1>0$ we define $H_n$ be the subgraph of $\tilde G_n$ spanned by the edges $\{e_1: e$ an edge of $G_n\} \cup \{g_2(v) : v\in V_n \cup \{\delta_n\} \}$. Otherwise, $h_1=0$ and we define $H_n$ to be the subgraph of $\tilde G_n$ spanned by the edges $\{e_1: e$ an edge of $G_n\}$. 
These weighted graphs are defined so that the gradient Ising models at inverse temperature $\beta$ and zero external field on $\tilde G_n$ and $H_n$ are equivalent to the gradient Ising models on $G_n^*$ with inverse temperatures $\beta_2$ and $\beta_1$ and external fields $h_2$ and $h_1$ respectively.
Let $(\bm_{n,1},\bm_{n,2})$ have law $\bC_{H_n,\beta_2,0}\otimes \bC_{\tilde G_n,\beta_2,0}$ and let $\bQ_n$ be the law of the random variable $(\bn_{n,1},\bn_{n,2}) \in (\N_0^3)^{E \cup V}$ defined by
\[
\bn_{n,1}(x) = \begin{cases} \bm_{n,1}(x_0)\mathbbm{1}(\text{$x$ has an endpoint in $V_n$}) & x \text{ is an edge}\\
\bm_{n,1}(x_0)\mathbbm{1}(x \in V_n) & x \text{ is a vertex}
\end{cases}
\]
for each $x\in E \cup V$ and
\[
\bn_{n,2}(x_i) = \begin{cases} \bm_{n,2}(x_i)\mathbbm{1}(\text{$x$ has an endpoint in $V_n$}) & x \text{ is an edge}\\
\bm_{n,2}(g_i(x))\mathbbm{1}(x \in V_n) & x \text{ is a vertex, $h>0$}\\
\bm_{n,2}(g_i(x))\mathbbm{1}(x \in V_n, i=2) & x \text{ is a vertex, $h=0$}
\end{cases}
\]
for each $x\in E \cup V$ and $i\in \{1,2\}$.
It follows from the definitions and from \cref{eq:LoopO(1)Ising,eq:LoopO(1)Ising_infinite} that $\bQ=\wlim_{n\to\infty} \bQ_n$.

Since $H_n$ is a subgraph of $\tilde G_n$ for each $n\geq 1$, we may apply \cref{lem:general_gradient} to deduce that
\begin{align}
\nonumber
\beta_1 J_x \langle \sigma_x \rangle_{G_n^*,\beta_2,h_2} - \langle \sigma_x \rangle_{G_n^*,\beta_1,h_1}
&=\beta_2 (1-\theta) J_x \left[\langle \sigma_x \rangle_{\tilde G_n,\beta_2,0} - \langle \sigma_x \rangle_{H_n,\beta_2,0}\right]
 \\&=  \bC_{H_n,\beta_2,0}\otimes \bC_{\tilde G_n,\beta_2,0}
\bigl(\sB_{x,n}\bigr)
=  \bC_{H_n,\beta_2,0}\otimes \bC_{\tilde G_n,\beta_2,0}
\bigl(\sB_{x,n}\bigr) 
\label{eq:Gradient_to_Bxn}
\end{align}
where $\sB_{x,n}$ is the set of pairs $(\bm_{n,1},\bm_{n,2})$ such that $\bm_{n,1}(x_1)=0$, $\bm_{n,2}(x_1)=1$, and the endpoints $u$ and $v$ of $x$ are connected to each other by an $(\bm_{n,1}+\bm_{n,2})$-open path in $\tilde G_n \setminus \{x_1\}$ but not in $H_n\setminus \{x_1\}$. 
Let $u$ and $v$ be the endpoints of $x$. For each $n\geq 1$, let $\sA_{x,n}$ be the set of pairs $(\bn_{1},\bn_{2}) \in \N_0^{E \cup V} \times (\N_0^2)^{E \cup V}$ such that the following hold:
\begin{enumerate}
  \item
 $\bn_1(x)=0$ and $\bn_2(x_1)=1$.
 \item  the endpoints $u$ and $v$ of $x$ are not connected by any $1$-open path that does not include $x$, $K^1_{u,x}(\bn_1,\bn_2)$ and $K^1_{v,x}(\bn_1,\bn_2)$ do not both intersect the set  $V \setminus V_n$, and do not both intersect the set $\{ w \in V: \bn_1(w)+\bn_2(w_1)>0 \}$.
 \item At least one of the following conditions hold:
 \begin{enumerate}
\item $\bn_2(x_1)>0$,
\item $u$ and $v$ are connected by a $2$-open path that does not include the edge $x$, or
\item $K^2_{u,x}(\bn_1,\bn_2)$ and $K^2_{v,x}(\bn_1,\bn_2)$ either both intersect the set  $V \setminus V_n$ or both intersect the set $\{ w \in V: \bn_1(w)+\bn_2(w_1)+\bn_2(w_2)>0 \}$.
 \end{enumerate}
 \end{enumerate}
 Here, as before $K^i_{w,x}(\bn_1,\bn_2)$ denotes the set of vertices that are connected to $w \in V$ by an $i$-open path in $G \setminus \{x\}$, where 
 an edge $e$ is said to be $1$-open if $\bn_1(e)+\bn_2(e_1)>0$ and $2$-open if $\bn_1(e)+\bn_2(e_1)+\bn_2(e_2)>0$. 
 It follows from the definitions that $\bC_{H_n,\beta_2,0}\otimes \bC_{\tilde G_n,\beta_2,0}
\bigl(\sB_{x,n}\bigr) = \bQ_n(\sA_{x,n})$
 and that
\[
\sA_x = \limsup_{n\to\infty} \sA_{x,n} = \liminf_{n\to\infty} \sA_{x,n}.
\]
If $G$ is locally finite then $\sA_x$ is closed in $\N_0^{E \cup V} \times (\N_0^2)^{E \cup V}$ and it follows by \eqref{eq:Gradient_to_Bxn} and the portmanteau theorem that
\[
\bQ(\sA_x) \geq \lim_{n\to\infty} \bQ_n (\sA_{x,n}) = \beta_1 J_x \lim_{n\to\infty} \left(\langle \sigma_x \rangle_{G_n^*,\beta_2,h_2} - \langle \sigma_x \rangle_{G_n^*,\beta_1,h_1} \right)
=
 \beta_1 J_x  \left(\langle \sigma_x \rangle_{\beta_2,h_2}^+ - \langle \sigma_x \rangle_{\beta_1,h_1}^+ \right),
 \]
 completing the proof in this case.

Let us now briefly discuss how the proof can be extended to the case that $G$ is not locally finite. The problem in this case is that $\sA_x$ need not be closed 
in the product topology on $\N_0^{E \cup V} \times (\N_0^2)^{E \cup V}$. (Indeed, the set of pairs $(\bn_1,\bn_2)$ for which $K_{u,x}^2$ contains more than one vertex is not closed.) 
This can be remedied as follows:  Let $X$ be the subset of $\N_0^{E \cup V} \times (\N_0^2)^{E \cup V}$ 
such that $\sum_{e\in E^\rightarrow_w} \bn_1(e)+\bn_2(e_1)+\bn_2(e_2)<\infty$ for every $w\in V$, and endow $X$ with the weakest topology 
that makes the families of functions
\begin{align*}
(\bn_1,\bn_2) &\mapsto (\bn_1(x),\bn_2(x_1),\bn_2(x_2)) &&: x\in E \cup V \qquad \text{ and }
\\ (\bn_1,\bn_2) &\mapsto \sum_{e\in E^\rightarrow_w} \bn_1(e)+\bn_2(e_1)+\bn_2(e_2) &&: w \in V
\end{align*}
continuous. Note that this topology is stronger than the product topology on $X$ and that $\sA_x \cap X$ is a closed subset of $X$ with respect to this topology. It follows easily from the Lupu--Werner coupling and the domination of the FK-Ising model by Bernoulli percolation that the measures $(\bQ_n)_{n\geq 1}$ are all supported on $X$ and are tight with respect to this topology. Since the probability measures $\bQ_n$ weakly converge to the probability measure $\bQ$ with respect to the product topology on $X$, it follows by tightness that they also converge to $\bQ$ with respect to the stronger topology introduced above. The claim now follows from the portmanteau theorem as before. \qedhere

\end{proof}

\subsection{H\"older continuity of the plus Ising measure}
\label{subsec:mainproof}

In this section we complete the proof of \cref{thm:main,thm:main_simple,thm:main_continuity,thm:main_continuity_FK}. We begin by applying the methods of \cref{sec:free_energy} to the percolation model considered in \cref{subsec:mismatched}.

\begin{prop}
\label{prop:finite_clusters_mismatched}
Let $G=(V,E,J)$ be a connected, unimodular, transitive, nonamenable weighted graph and let $o$ be a vertex of $G$. There exist positive constants $C$ and $\delta$ such that
\[
\bQ_{\beta_1,h_1,\beta_2,h_2} \left( n \leq |K_o^1| < \infty \right) \leq C n^{-\delta}
\]
for every $n\geq 1$, $\beta_2 \geq \beta_1 \geq 0$ and $h_2 \geq h_1 \geq 0$.
\end{prop}

\begin{proof}
By scaling, we may assume without loss of generality that $\sum_{e\in E^\rightarrow_o} J_e=1$. 
Let $(\bn_1,\bn_2)$ be distributed according to $\bQ_{\beta_1,\beta_2,h}$, let $\omega\in \{0,1\}^E$ be defined by $\omega(e)=\mathbbm{1}(\bn_1(e)+\bn_2(e_1) >0)$, and let $\mu_{\beta_1,\beta_2,h}$ be the law of $\omega$. 
It follows from the Lupu-Werner coupling that the sets $\{e:\bn_1(e)>0\}$ and $\{e:\bn_2(e_1)>0\}$ are each stochastically dominated by the Bernoulli percolation process with edge inclusion probabilities $1-e^{-2\beta_1 J_e} \leq 2\beta_1 J_e$ and hence that $\omega$ is stochastically dominated by the Bernoulli percolation process with edge inclusion probabilities $1-e^{-4\beta_1 J_e} \leq 4\beta_1 J_e$. It follows by a counting argument that $\bQ |K_o^1| \leq 2$ when $\beta_1 \leq 1/8$, so that it suffices to consider the case $\beta \geq 1/8$.

The connection between the loop $O(1)$ model and the random current model allows us to consider $\omega$ as a percolation in random environment model where the environment $\bp$ satisfies $\bp_e \geq (\cosh(\beta_1 J_e)-1)/\cosh(\beta_1 J_e)$ almost surely for every $e\in E$, and it follows by calculus that there exists a positive constant $c$ such that $\bp_e \geq (\cosh(\beta_1 J_e)-1)/\cosh(\beta_1 J_e) \geq c J_e^2$ for every $e\in E$ and $\beta_1 \geq 1/8$. Moreover, \cref{thm:spectralradius} implies that $\rho(\mu_{\beta_1,\beta_2,h})\leq \rho(\bQ_{\beta_1,\beta_2,h}) \leq \rho(G)$. With these ingredients in place, the claim follows by a similar (and slightly simpler) proof to that of \cref{thm:finite_clusters}, and we omit the details.
\end{proof}

We next apply \cref{prop:finite_clusters_mismatched,cor:gradient} to control the change in the expected degree in the FK-Ising model as we increase $\beta$ or $h$.

\begin{lemma}
\label{cor:change_in_correlations}
Let $G=(V,E,J)$ be a connected, unimodular, transitive, nonamenable weighted graph, and let $o$ be a vertex of $G$. There exists $\delta>0$ such that
\[
\beta h \langle \sigma_o \rangle_{\beta,h}^+, \quad \sum_{e\in E^\rightarrow_o} \beta J_e \langle \sigma_e \rangle_{\beta,h}^+, \quad  \phi_{2,\beta,h}^w\bigl(\omega(o)=1\bigr), \; \text{ and } \; \sum_{e\in E^\rightarrow_o} \phi_{2,\beta,h}^w\bigl(\omega(e)=1\bigr)
\]
are locally $\delta$-H\"older continuous functions of $(\beta,h)\in [0,\infty)^2$.
%
\end{lemma}



\medskip

Throughout the remainder of this section, we will use  $\asymp$, $\preceq,$ and $\succeq$ to denote equalities and inequalities that hold to within positive multiplicative constants depending on the weighted graph $G$ but not any further parameters.

\begin{proof}[Proof of \cref{cor:change_in_correlations}]
By scaling we may assume without loss of generality that $\sum_{e\in E^\rightarrow_o} J_e=1$.
We first prove local H\"older continuity of $\beta h\langle \sigma_o \rangle_{
\beta,h}^+$ and $\sum_{e\in E^\rightarrow_o} J_e \langle \sigma_e \rangle_{\beta,h}^+$. Since $\langle \sigma_x \rangle_{\beta,h}^+$ is increasing in $\beta$ and $h$ for each $x\in E \cup V$, it suffices to prove that there exists $\delta>0$ such that
\begin{align}
\beta_1 h_1 \left[\langle \sigma_o \rangle_{\beta_2,h_2}^+-\langle \sigma_o \rangle_{\beta_1,h_1}^+\right] &\preceq \left(\beta_2h_2 -\beta_1 h_1 + \beta_2-\beta_1\right)^{\delta} \qquad \text{ and }
\label{eq:HolderClaim11}
\\
\sum_{e\in E^\rightarrow_o} \beta_1 J_e \left[ \langle \sigma_e \rangle_{\beta_2,h_2}^+ - \langle \sigma_e\rangle_{\beta_1,h_1}^+\right] &\preceq (\beta_1^{1/2} \vee \beta_1)  \left(\beta_2h_2 -\beta_1 h_1 + \beta_2-\beta_1\right)^{\delta}
\label{eq:HolderClaim12}
\end{align}
for every $\beta_2 \geq \beta_1 \geq 0$ and $h_2 \geq h_1 \geq 0$, the claim then following by  a little elementary analysis. Fix one choice of these parameters $\beta_1,\beta_2,h_1,h_2$, and let $\bQ=\bQ_{\beta_1,h_1,\beta_2,h_2}$. All the implicit constants appearing in this proof will depend on $G$ but not on the choice of these parameters. 
We have by \cref{cor:gradient} that
\[
\beta_1 J_x \left[\langle \sigma_x \rangle_{\beta_2,h_2}^+ - \langle \sigma_x \rangle_{\beta_1,h_1}^+\right]
\leq
\bQ (\sA_x)
\]
for each $x\in E \cup V$,
where $\sA_x$ is the event defined just before the statement of that proposition and where we set $J_v=h_1$ for each $v\in V$.
 Let $\cB$ be the set of vertices $v$ of $G$ such that either $\bn_2(v_2)>0$ or there exists an edge $e$ of $G$ touching $v$ such that $\bn_2(e_2)>0$, and let $\sB$ be the event that $K_o^1$ is finite and that $K_o^1\cap \cB \neq \emptyset$. Observe that 
\[
\sA_x \subseteq \sB \cap \{\bn_2(x_1)>0\}
\]
for every $x \in E^\rightarrow_o \cup \{o\}$, so that
\[
\beta_1 h_1 \left[\langle \sigma_o \rangle_{\beta_2,h_2}^+-\langle \sigma_o \rangle_{\beta_1,h_1}^+\right] \leq \bQ(\sB)
\]
and similarly that
\begin{align}
\sum_{e\in E^\rightarrow_o} \beta_1 J_e \left[\langle \sigma_e \rangle_{\beta_2,h_2}^+ - \langle \sigma_e \rangle_{\beta_1,h_1}^+\right] 
&\leq \sum_{e\in E^\rightarrow_o} 
\bQ(\sA_e) \leq \bQ \left[ \mathbbm{1}(\sB) \cdot \#\{e\in E^\rightarrow_o : \bn_2(e_1)>0 \} \right]
\nonumber
\\
&
\leq  \sqrt{\bQ\left(\sB\right) \bQ\left[\#\{e\in E^\rightarrow_o : \bn_2(e_1)>0 \}^2 \right]},
\label{eq:degree_2nd_moment}
\end{align}
where we used Cauchy-Schwarz in the second line. (One can improve the exponent obtained here by using H\"older instead of Cauchy-Schwarz.) As discussed in \cref{subsec:random_currents}, it is a consequence of the Lupu-Werner coupling that the set $\{e:\bn_2(e_1)>0\}$ is stochastically dominated by the random cluster model on $G$ and hence also by a Bernoulli bond percolation process on $G$ in which each edge $e$ of $G$ is included independently at random with probability $(e^{\beta_1 J_e}-1)/e^{\beta_1 J_e} \leq \beta_1 J_e$. It follows easily that  
\[\bQ\left[\#\{e\in E^\rightarrow_o : \bn_2(e_1)>0 \}^2 \right]
\leq \sum_{e\in E^\rightarrow_o} \beta_1 J_e + 2\sum_{e \neq e' \in E^\rightarrow_o} \beta_1^2 J_e J_{e'}
\preceq \beta_1 \vee \beta_1^2 \] and hence that
\begin{align}
\sum_{e\in E^\rightarrow_o} \beta_1 J_e \left[\langle \sigma_e \rangle_{\beta_2,h_2}^+ - \langle \sigma_e \rangle_{\beta_1,h_1}^+\right] 
&
\preceq  (\beta_1^{1/2} \vee \beta_1) \bQ\left(\sB\right)^{1/2}.
\label{eq:degree_2nd_moment2}
\end{align}
Thus, to prove the claimed inequalities \eqref{eq:HolderClaim11} and \eqref{eq:HolderClaim12}, it suffices to prove that there exists a constant $\delta>0$ such that
\begin{equation}
\label{eq:HolderClaim2}
\bQ(\sB) \preceq \left(\beta_2h_2 -\beta_1 h_1 + \beta_2-\beta_1\right)^{\delta}.
\end{equation}

To this end, we consider the union bound 
\begin{equation}
\label{eq:Qunion}
\bQ(\sB)
\\\leq \bQ(n<|K_o^1|<\infty) + \bQ(\sB \cap \{|K_o^1| \leq n\}),
\end{equation}
which holds for every $n\geq 1$. It follows from \cref{prop:finite_clusters_mismatched} that there exists a positive constant $\delta_1$ such that the first term satisfies 
\begin{equation}
\label{eq:Qlargen}
\bQ(n<|K_o^1|<\infty) \preceq n^{-\delta_1}
\end{equation}
for every $n\geq 1$. We now bound the second term.
For each $n\geq 1$ define a mass-transport function $F_n:V^2 \to [0,\infty]$ by
\[
F_n(u,v) = \bQ\left(|K_u^1| \leq n, v \in K_u \cap \cB \right),
\]
so that
\begin{multline}
\bQ(\sB \cap \{|K_o^1| \leq n\}) \leq \sum_{v\in V} F_n(o,v) = \sum_{v\in V} F_n(v,o)\\ = \bQ\left[|K_o^1| \mathbbm{1}\left(|K_o^1|\leq n, o \in \cB\right) \right]\leq n \bQ(o \in \cB).
\end{multline}
(The final inequality here is presumably rather wasteful.)
Let $\theta=1-\beta_1/\beta_2$ and $\phi=1-(\beta_1h_1)/(\beta_2h_2)$. 
It follows straightforwardly from \eqref{eq:bounding_current_by_FK1} and \eqref{eq:bounding_current_by_FK2} and the definition of $\bQ$ that 
\begin{align}
\label{eq:Qsmalln}
\bQ(o \in \cB) &\leq \bQ(\bn_2(o_2)>0)+\sum_{e\in E^\rightarrow_o}\bQ(\bn_2(e_2)>0)
\\&\leq \phi \beta_2 h_2 + \sum_{e\in E^\rightarrow_o}  \theta \beta_2 J_e = \beta_2h_2-\beta_1h_1 + \beta_2-\beta_1.
\end{align}
Putting together \eqref{eq:Qunion}, \eqref{eq:Qlargen}, and \eqref{eq:Qsmalln} yields that
\begin{equation*}
\bQ(\sB)
\\\preceq  n^{-\delta_1}+ n \left(\beta_2h_2-\beta_1h_1 + \beta_2-\beta_1\right)
\end{equation*}
for every $n\geq 1$, and taking $n=\lceil (\beta_2h_2-\beta_1h_1 + \beta_2-\beta_1)^{-1/(1+\delta_1)} \rceil$ yields the claimed inequality \eqref{eq:HolderClaim2} with the exponent $\delta=\delta_1/(1+\delta_1)$. This completes the proof of local H\"older continuity of $\beta h\langle \sigma_o \rangle_{
\beta,h}^+$ and $\sum_{e\in E^\rightarrow_o} J_e \langle \sigma_e \rangle_{\beta,h}^+$.


We now deduce local H\"older continuity of $\sum_{e\in E^\rightarrow_o} \phi_{2,\beta,h}^w\bigl(\omega(e)=1\bigr)$ and $\phi_{2,\beta,h}^w\bigl(\omega(o)=1\bigr)$ from local H\"older continuity of $\sum_{e\in E^\rightarrow_o} \beta J_e \langle \sigma_e \rangle_{\beta,h}^+$ and $\beta h\langle \sigma_o \rangle_{\beta,h}^+$.
The equality \eqref{eq:LoopO(1)Ising_infinite} implies that
\[
\phi_{2,\beta,h}^w\bigl(\omega(x)=1\bigr)
= 1 - e^{-\beta J_x}\langle e^{-\beta J_x \sigma_x}\rangle_{\beta,h}^+.
\]
Using the identity 
$e^{-\beta J_x \sigma_x} = \cosh(\beta J_x) - \sigma_x \sinh(\beta J_x)$
(which holds since $\sigma_x\in \{-1,+1\}$) 
yields that
\begin{align}
\phi_{2,\beta,h}^w\bigl(\omega(x)=1\bigr) &= 1- e^{-\beta J_x} \cosh(\beta J_x) + e^{-\beta J_x} \sinh(\beta J_x)\langle \sigma_x \rangle_{\beta,h}^+ 
\label{eq:phisigma_formula}
\end{align}
for every $x\in E \cup V$ and $\beta,h \geq 0$. 
 Since the functions $1-e^{-\beta J_x} \cosh(\beta J_x)$ and $e^{-\beta J_x}\sinh(\beta J_x)/\beta J_x$ are both locally Lipschitz and since local H\"older continuity is preserved under sums and products, we deduce immediately that $\phi_{2,\beta,h}^w(\omega(o)=1)$ is locally $\delta$-H\"older continuous as claimed. The proof that $\sum_{e\in E^\rightarrow_o} \phi_{2,\beta,h}^w\bigl(\omega(e)=1\bigr)$ is locally $\delta$-H\"older continuous for some $\delta>0$ is similar and we omit the details.
\end{proof}

\begin{remark}
The $\beta$-derivatives of $\sum_{e\in E^\rightarrow_o} J_e \langle \sigma_e \rangle_{\beta,0}^+$ and $\sum_{e\in E^\rightarrow_o} \phi_{2,\beta,0}^w\bigl(\omega(e)=1\bigr)$ are closely related to the \emph{specific heat} of the Ising model. It is conjectured, and known in some cases, that these  derivatives are bounded (but not necessarily continuous) in high-dimensional models and unbounded in low-dimensional models; see e.g.\ \cite{MR3162481,MR678000,MR857063,MR588470} for detailed discussions.
\end{remark}



We are now ready to complete the proof of \cref{thm:main_continuity}.
 (Note that local H\"older continuity of the magnetization $\langle \sigma_o \rangle_{\beta,h}^+$ is \emph{not} implied by \cref{cor:change_in_correlations}, which does not give any control of $\langle \sigma_o \rangle_{\beta,0}^+$.)
The proof will use the fact that there exists an automorphism-invariant monotone
coupling, sometimes known as the \textbf{Grimmett coupling} \cite{MR1379156}, between the two random cluster measures $\phi^w_{q,\beta_1,h_1}$ and $\phi^w_{q,\beta_2,h_2}$ whenever $q\geq 1$, $\beta_2 \geq \beta_1 >0$ and $h_2 \geq h_1 \geq 0$. That is, for each such $q,\beta_1,\beta_2,h_1,h_2$, there exists a pair of random variables $(\omega_1,\omega_2)$ whose law is invariant under the automorphisms of $G$ such that the marginal law of $\omega_1$ is $\phi^w_{q,\beta_1,h_1}$, the marginal law of $\omega_2$ is $\phi^w_{q,\beta_2,h_2}$, and $\omega_1(x) \leq \omega_2(x)$ for every $x\in V \cup E$ almost surely.
The existence of such a coupling is essentially due to H\"aggstr\"om, Jonasson, and Lyons \cite{MR1913108}, who proved that Grimmett's monotone coupling of the random cluster model at different temperatures \cite{MR1379156} can be extended to infinite graphs in an automorphism-invariant way; while that paper considers only locally finite models in zero external field, the proof generalizes straightforwardly to possibly long-range models in non-negative external field.

\begin{proof}[Proof of \cref{thm:main_continuity}]

 By inclusion-exclusion, it suffices to prove that 
there exists $\delta>0$ such that $\bI^+_{\beta,h}(\sigma(v)=1 \text{ for every $v\in A$})$ is a locally $\delta$-H\"older continuous function of $(\beta,h)\in [0,\infty)^2$ for each finite $A \subseteq V$. 
 Given the configuration $\omega$, we say that a set $A \subseteq V$ is $w$\textbf{-finite} if it is finite and does not intersect the set $\{w\in V: \omega(w)=1\}$.
The Edwards--Sokal coupling implies that
\begin{equation}
\label{eq:ES_cylinder}
\bI^+_{\beta,h}(\sigma(v)=1 \text{ for every $v\in A$}) = \phi^w_{2,\beta,h}\left[2^{-\#\{w\text{-finite clusters intersecting }A\}} \right]
\end{equation}
for every $\beta,h \geq 0$ and every finite set $A \subseteq V$. Since the right hand side is increasing in $\beta$ and $h$, it suffices to prove that for each $M<\infty$ there exists a constant $C=C(M)$ such that
\begin{multline}
\bI^+_{\beta_2,h_2}(\sigma(v)=1 \text{ for every $v\in A$})
-
\bI^+_{\beta_1,h_1}(\sigma(v)=1 \text{ for every $v\in A$})
\\
\hspace{3.5cm}\leq C |A| (\beta_2-\beta_1 + h_2-h_1)^\delta
\end{multline}
for every $A \subseteq V$, $0\leq \beta_1 \leq \beta_2 \leq M$ and $0\leq h_1 \leq h_2 \leq M$.
%
%

 Fix $M$ and one such choice of $M \geq \beta_2 \geq \beta_1 \geq 0$, $M \geq h_2 \geq h_1 \geq 0$.
We write $\preceq_M$ for an inequality that holds to within a positive multiplicative constant depending on $M$ but not any further parameters. 
 Let $(\omega_1,\omega_2)$ be an automorphism-invariant monotone coupling of $\phi^w_{2,\beta_1,h_1}$ and $\phi^w_{2,\beta_2,h_2}$ as above. We write $\P$ for probabilities taken with respect to the joint law of $(\omega_1,\omega_2)$, and write $K^1_v$ and $K^2_v$ for the clusters of $v$ in $\omega_1$ and $\omega_2$ respectively for each $v\in V$. 
We say that a set $W \subseteq V$ is $w_i$\textbf{-finite} if it is finite and does not intersect the set $\{w\in V: \omega_i(w)=1\}$.
 Let $\sA$ be the event that there are fewer $w_2$-finite clusters intersecting $A$ in $\omega_2$ than there are $w_1$-finite clusters intersecting $A$ in $\omega_1$. The equality \eqref{eq:ES_cylinder} implies that
\begin{align*}
\bI^+_{\beta_2,h_2}(\sigma(v)=1 \text{ for every $v\in A$})
-
\bI^+_{\beta_1,h_1}(\sigma(v)=1 \text{ for every $v\in A$})
\leq \P(\sA).
\end{align*}
%
%
Let $\cP$ be the set of vertices $v\in V$ such that either $\omega_2(v)=1$ and $\omega_1(v)=0$ or there exists $e\in E^\rightarrow_v$ such that $\omega_2(e)=1$ and $\omega_1(e)=0$. We have from the definitions that
\[
\sA \subseteq \bigcup_{v\in A} \{ K^1_v \text{ is $w_1$-finite and } K^1_v \cap \cP \neq \emptyset\} 
\subseteq \bigcup_{v\in A} \{ |K^1_v|<\infty \text{ and } K^1_v \cap \cP \neq \emptyset\} 
\]
so that transitivity and a union bound give that
\begin{multline*}
\bI^+_{\beta_2,h_2}(\sigma(v)=1 \text{ for every $v\in A$})
-
\bI^+_{\beta_1,h_1}(\sigma(v)=1 \text{ for every $v\in A$})\\
\leq \P(\sA) \leq |A| \P\bigl(|K_o^1|<\infty \text{ and } K_o^1 \cap \cP \neq \emptyset\bigr).
\end{multline*}
Thus, to conclude the proof it suffices to prove that there exists a positive constant $\delta$ such that
\begin{equation}
\label{eq:HolderClaim4}
\P\bigl(|K_o^1|<\infty \text{ and } K_o^1 \cap \cP \neq \emptyset\bigr) \preceq_M (\beta_2-\beta_1 + h_2-h_1)^\delta.
\end{equation}

We prove \eqref{eq:HolderClaim4} by following a similar strategy to the proof of \eqref{eq:HolderClaim2} but using \cref{thm:finite_clusters} and \cref{cor:change_in_correlations} instead of \cref{prop:finite_clusters_mismatched} and \cref{eq:bounding_current_by_FK1,eq:bounding_current_by_FK2}. 
We begin by writing down for each $n\geq 1$ the union bound
\[
\P(|K_o^1|<\infty \text{ and } K_o^1 \cap \cP \neq \emptyset) \leq \P(n \leq |K_o^1| < \infty) + \P(|K_o^1| \leq n \text{ and } K_o^1 \cap \cP \neq \emptyset).
\]
\cref{thm:free_random_cluster} implies that there exists a constant $\delta_1>0$ such that $\P(n \leq |K_o^1| <\infty) \preceq n^{-\delta_1}$. Meanwhile, as in the proof of \cref{cor:change_in_correlations}, we may apply the mass-transport principle to bound
\[
\P(|K_o^1| \leq n \text{ and } K_o^1 \cap \cP \neq \emptyset) \leq n \, \P(o \in \cP)
\]
for every $n\geq 1$. Another union bound then implies that
\begin{align*}
\P(o \in \cP) &\leq \P(\omega_2(o) = 1)-\P(\omega_1(o) = 1) + 
\sum_{e\in E^\rightarrow_o} \left[\phi_{2,\beta_2,h_2}^w\bigl(\omega(e)=1\bigr)-\phi_{2,\beta_1,h_1}^w\bigl(\omega(e)=1\bigr) \right],
\end{align*}
and applying \cref{cor:change_in_correlations} yields that there exists a positive constant $\delta_2$ such that
\[\P(o \in \cP) \preceq_M (\beta_2-\beta_1 + h_2-h_1)^{\delta_2}.\]
It follows that
\[
\P\bigl(|K_o^1|<\infty \text{ and } K_o^1 \cap \cP \neq \emptyset\bigr) \preceq_M  n^{-\delta_1} + n \left(\beta_2-\beta_1+h_2-h_1\right)^{\delta_2}
\]
for each $n\geq 1$, and taking 
$n=\left\lceil \left(\beta_2-\beta_1+h_2-h_1\right)^{-\delta_2/(1+\delta_1)} \right\rceil$ implies the claimed inequality \eqref{eq:HolderClaim4}. This completes the proof.
\qedhere

\end{proof}


\begin{proof}[Proof of \cref{thm:main,thm:main_simple}]
It follows immediately from \cref{thm:main_continuity} that 
\[m^*(\beta_c)=m^+(\beta_c,0)=\lim_{\beta \uparrow \beta_c} m^+(\beta,0)=0,\]
establishing \cref{thm:main_simple}. Similarly, \cref{thm:main} follows from \cref{thm:main_continuity} together with the fact that
\[
m^\#(\beta,h) \leq m^+(\beta \vee \beta_c,|h|) = m^+(\beta \vee \beta_c,|h|) - m^+(\beta_c,0)
\]
for every $\beta \geq 0$, $h\in \R$, and $\#\in \{f,+,-\}$.
\end{proof}

It remains to deduce \cref{thm:main_continuity_FK} from \cref{thm:main_continuity}.

\begin{proof}[Proof of \cref{thm:main_continuity_FK}]
 By \cref{thm:main_continuity}, there exists $\delta>0$ such that if $F:\{0,1\}^V \to \R$ depends on at most finitely many vertices then $\bI^+_{\beta,h} [F(\sigma)]$ is a locally $\delta$-H\"older continuous function of $(\beta,h)\in [0,\infty)^2$. It follows easily that $\bG^w_{\beta,h} \left[e^{\beta H_A} \right]$ is locally $\delta$-H\"older continuous for each finite $A \subseteq E \cup V$, and hence by \eqref{eq:FK_Gradient_Formula} that $\phi^w_{2,\beta,h}(\omega(x)=0 \text{ for all $x\in A$})$ is locally $\delta$-H\"older continuous for each finite $A \subseteq E \cup V$. The claim now follows by inclusion-exclusion.
\end{proof}

We note that the local H\"older continuity provided by \cref{thm:main_continuity} is presumably very far from optimal when $\beta \neq \beta_c$. We conjecture that the following much stronger statement holds. This conjecture is most interesting in the case that $h=0$ and $\beta>\beta_c$.

\begin{conjecture}
Let $G$ be an infinite Cayley graph. Then the magnetization $m^+(\beta,h)$ is an infinitely differentiable function of $(\beta,h) \in [0,\infty)^2 \setminus \{(\beta_c,0)\}$.
\end{conjecture}

See \cite{HermonHutchcroftSupercritical,georgakopoulos2018analyticity,georgakopoulos2020analyticity} for related results for Bernoulli percolation.

\section{Closing remarks and open problems}

\subsection{Equality of critical parameters for the FK-Ising model}
\label{subsec:betafbetaw}

Let $G=(V,E,J)$ be a connected, transitive, weighted graph. Recall that we define $\beta_c^\#(q)= \sup\{ \beta \geq 0: \phi^\#_{q,\beta,0}$ is supported on configurations with no infinite clusters$\}$ for each $q \geq 1$ and $\#\in \{f,w\}$. When $G$ is amenable and $q\geq 1$, it is a classical theorem \cite[Theorem 4.63]{GrimFKbook} that $\phi^f_{q,\beta,0} \neq \phi^w_{q,\beta,0}$ for at most countably many values of $\beta$  and hence that $\beta_c^f(q)=\beta_c^w(q)$ for every $q\geq 1$. On the other hand, when $G$ is nonamenable, Jonasson \cite{MR1671859} proved that there is a strict inequality $\beta_c^w(q)<\beta_c^f(q)$ between these two critical parameters   for all sufficiently large values of $q$ \cite{MR1671859}. For a regular tree, strict inequality holds if and only if $q>2$ \cite{MR1373377}. We now show that equality always holds when $q=2$; we believe that this result is new in the nonamenable case.

\begin{prop}
\label{prop:betafbetaw}
Let $G=(V,E,J)$ be a connected, transitive, weighted graph. Then the critical inverse temperatures for the Ising model, the wired FK-Ising model, and the free FK-Ising model coincide. That is, $\beta_c=\beta_c^f(2)=\beta_c^w(2)$.
\end{prop}

\begin{proof}
It follows by sharpness of the Ising phase transition \cite{MR894398,duminil2015new} that
 \[\sum_{x\in V} \langle \sigma_o \sigma_x \rangle_{\beta,0}^f < \infty \iff \sum_{x\in V} \langle \sigma_o \sigma_x \rangle_{\beta,0}^+ < \infty  \iff 
\text{$\beta < \beta_c$}.
 \] 
See in particular \cite[Theorem 2.1 and Section 2.2, Remark 4]{duminil2015new}. On the other hand, we have by the Edwards--Sokal coupling that
\[
\langle \sigma_o \sigma_x \rangle_{\beta,0}^f = \phi^f_{2,\beta,0}(o \leftrightarrow x) \qquad \text{ and } \qquad \langle \sigma_o \sigma_x \rangle_{\beta,0}^+ = \phi^w_{2,\beta,0}(o \leftrightarrow x \text{ or } |K_o|=|K_x|=\infty),
\]
so that $\phi_{2,\beta,0}^\#[|K_o|]<\infty$ if and only if $\beta<\beta_c$ for each $\# \in \{f,w\}$. The claimed equality $\beta_c^f(2)=\beta_c^w(q)=\beta_c$ follows from the sharpness of the phase transition for the random cluster model \cite{MR3898174,1901.10363}, and in particular from \cite[Theorem 1.5]{1901.10363} which states that $\phi_{q,\beta,0}^\#[|K_o|] < \infty$ for every $\beta < \beta_c^\#(q)$ and hence that $\beta_c^\#(q)=\sup\{\beta \geq 0 : \phi_{q,\beta,0}^\#[|K_o|] < \infty\}$ for every $q\geq 1$ and $\# \in \{f,w\}$.
\end{proof}

\subsection{Consequences for planar graphs}
\label{subsec:planar}

We now explain how our results interact with those of H\"aggstr\"om, Jonasson, and Lyons \cite{MR1894115} to deduce further consequences in the planar case. 
Let $G=(V,E)$ be a transitive nonamenable graph. For each $q \geq 1$ and $\#\in \{f,w\}$ we define
\[
\beta_u^\#(q) = \inf \bigl\{ \beta \geq 0 : \phi^\#_{q,\beta,0} \text{ is supported on configurations with a unique infinite cluster}\bigr\}.
\]
It follows from a theorem of Lyons and Schramm \cite[Theorem 4.1]{LS99} that if $G$ is unimodular then $\phi^\#_{q,\beta,0}$ is supported on configurations with a unique infinite cluster if and only if $\inf_{u,v\in V} \phi^\#_{q,\beta,0}(u \leftrightarrow v) > 0$, so that $\phi^\#_{q,\beta,0}$ is supported on configurations with a unique infinite cluster for every $\beta>\beta_u^\#(q)$. The relationships between $\beta_c^f,\beta_c^w,\beta_u^f,$ and $\beta_u^w$ are discussed in detail in \cite[Section 3]{MR1894115}.


 Let $G$ be a unimodular, quasi-transitive, nonamenable proper plane graph with locally finite quasi-transitive dual $G^\dagger$.
Let $\omega$ be a random variable with law $\phi^\#_{G,q,\beta,0}$ for some $q\geq 1$, $\beta > 0$, and $\# \in \{f,w\}$. It is well-known (see e.g.\ \cite[Proposition 3.4]{MR1894115}) that the dual configuration $\omega^\dagger = \{e^\dagger : e\notin \omega\}$ has law 
$\phi^{\#^\dagger}_{G^\dagger,q,\beta^\dagger,0}$, where $w^\dagger=f$, $f^\dagger = w$, and $\beta^\dagger=\beta^\dagger(q,\beta)>0$ is the unique solution to
\[
(e^{2\beta}-1)(e^{2\beta^\dagger}-1)=q.
\]
Let $k$ and $k^\dagger$ be the number of infinite clusters of $\omega$ and $\omega^\dagger$ respectively. Proposition 3.5 of \cite{MR1894115} generalizes an argument of Benjamini and Schramm \cite{BS00} to show that 
\begin{equation}
\label{eq:duality_clusters}
(k,k^\dagger) \in \bigl\{(0,1),(1,0),(\infty,\infty)\bigr\} \qquad \text{almost surely.}
\end{equation} It follows in particular that
$\beta_c^\#(G^\dagger,q) = \beta_u^{\#^\dagger}(G,q)^\dagger$ for every $q \geq 1$ \cite[Corollary 3.6]{MR1894115}. 
When $q=2$ and $G^\dagger$ is transitive we have that $\beta_c^w(G^\dagger)=\beta_c^f(G^\dagger)$ by \cref{prop:betafbetaw} and hence that $\beta_u^w(G)=\beta_u^f(G)$ also.
Since there are a.s.\ no infinite clusters in the critical free random-cluster model on $G^\dagger$, it follows from \eqref{eq:duality_clusters} that there is a unique infinite cluster in the wired random cluster model on $G$ at $\beta_u^w$ \cite[Corollary 3.7]{MR1894115}. 
Combining these facts with \cref{thm:main,thm:main_continuity} yields the following corollary.  

\begin{corollary}[The non-uniqueness phase on planar graphs]
\label{cor:planar}
Let $G=(V,E)$ be a unimodular, transitive, nonamenable proper plane graph with locally finite transitive dual $G^\dagger$, and consider the FK-Ising model on $G$. Then the following hold:
\begin{enumerate}
\item The parameters $\beta_u=\beta_u^f=\beta_u^w$ coincide and satisfy $\beta_u>\beta_c$.
\item The free and wired FK-Ising measures on $G$ coincide at $\beta_u$ and are both supported on configurations with a unique infinite cluster.
\item The free FK-Ising measure $\phi^f_{2,\beta,0}$ is weakly continuous in $\beta$.
\item For each $\beta \geq 0$, $\bG^w_{\beta,0} \neq \bG^f_{\beta,0}$ if and only if $\phi^w_{2,\beta,0} \neq \phi^f_{2,\beta,0}$ if and only if $\beta_c < \beta <\beta_u$.
\end{enumerate}
\end{corollary}

The analogous results for Bernoulli percolation are due to Benjamini and Schramm \cite{BS00}. The condition that the dual of $G$ is transitive should not really be necessary, since \cref{thm:main,thm:main_continuity} should both extend to quasi-transitive nonamenable graphs. 

\begin{proof}
The identity $\beta_u=\beta_u^f=\beta_u^w$ follows from \cref{prop:betafbetaw} and duality as explained above. \cref{thm:main_continuity_FK} implies that the free and wired FK-Ising models on $G^\dagger$ coincide at $\beta_c$, so that item $2$ follows by planar duality together with the corresponding fact for the free measure at $\beta_c$ \cite[Theorem 3.1]{MR1894115}. Since the free and wired FK-Ising measures both have no infinite clusters at $\beta_c$ and a unique infinite cluster at $\beta_u$, it follows that $\beta_u>\beta_c$, completing the proof of item $1$. Item $3$ follows immediately from \cref{thm:main_continuity_FK} applied to $G^\dagger$ together with planar duality. Item $4$ follows from \cite[Proposition 3.8]{MR1894115} and the formula \eqref{eq:FK_Gradient_Formula}.
\end{proof}

\subsection{Discontinuity of the free Ising measure}
\label{subsec:discontinuity}

It is natural to wonder whether an analogue of \cref{thm:main_continuity} also holds for the \emph{free} Ising measure with zero external field; we now argue that this is not the case in general. 
In the previous subsection we saw examples of Cayley graphs in which there is a unique infinite cluster in the free FK-Ising model at $\beta_u^f$.
It is also possible for there to be infinitely many infinite clusters at $\beta_u^f$. Indeed, it follows from the proof of \cite[Corollary 6.6]{LS99} that if $G$ is a Cayley graph of an infinite Kazhdan group then $\beta_f^u<\infty$ and the number of infinite clusters in the free FK-Ising model at $\beta_c^f$ is either $0$ or $\infty$. The perturbative criteria of \cite{MR1833805} imply that every nonamenable group has a Cayley graph for which $\beta_u^f > \beta_c$ (see \cite{MR1756965} for a similar result for Bernoulli percolation), and it follows that there exists a nonamenable Cayley graph such that the free FK-Ising model has infinitely many infinite clusters at the uniqueness threshold $\beta_u^f$. 

We claim that if $\beta_u^f>\beta_c$ and there is non-uniqueness at $\beta_u^f$ then there exist vertices $u$ and $v$ such that
$\langle \sigma_u \sigma_v \rangle^f_{\beta,0} =\phi^f_{2,\beta,0}(u \leftrightarrow v)$ is discontinuous at $\beta_u^f$. Indeed, under this assumption, we have by the aforementioned theorem of Lyons and Schramm \cite[Theorem 4.1]{LS99} that there exist $u,v \in V$ such that
\[
\phi^f_{2,\beta_u^f,0}(u \leftrightarrow v) \leq \frac{1}{2}\phi^f_{2,\beta_u^f,0}(o \to \infty)^2.
\]
On the other hand, for each $\beta>\beta_u^f$ we have by FKG that
\[
\phi^f_{2,\beta,0}(u \leftrightarrow v) \geq \phi^f_{2,\beta,0}(u \to \infty \text{ and } v\to \infty) \geq \phi^f_{2,\beta_u^f,0}(o \to \infty)^2,
\]
which yields the desired discontinuity. It follows that the measures $\bI^f_{\beta,0}$ and $\bG_{\beta,0}^f$ are both weakly discontinuous in $\beta$ at $\beta_u^f$, so that \cref{thm:main_continuity} cannot be extended to the case of free boundary conditions. Similar phenomena for Bernoulli percolation are discussed in \cite[Section 5.1]{HermonHutchcroftSupercritical}.

In fact, the free FK-Ising measure $\phi^f_{2,\beta,0}$ is also weakly discontinuous in $\beta$ at $\beta_u^f$ in the same class of examples. This follows from the following general proposition, which shows that the gradient Ising model and FK-Ising model always have the same continuity properties.

\begin{prop}
\label{prop:continuity_equivalence}
Let $G=(V,E)$ be an infinite connected weighted graph, let $\# \in \{f,w\}$, and let $((\beta_n,h_n))_{n\geq1}$ be a sequence in $[0,\infty)^2$ converging to some $(\beta,h)$. Then 
\[
\bG^\#_{\beta,h} = \wlim_{n\to\infty} \bG^\#_{\beta_n,h_n}  \iff \phi^\#_{\beta,h} = \wlim_{n\to\infty} \phi^\#_{\beta_n,h_n}.
\] 
\end{prop}

A similar proof extends this equivalence to the random current and loop $O(1)$ models. It is also possible to prove a similar statement for the random cluster model with $q\in \{3,4,\ldots\}$ and the Potts model by using the fact that the random cluster model can be represented as Bernoulli percolation on the colour clusters to deduce a formula analogous to \cref{eq:FK_Gradient_Formula}.

\begin{proof}
The claim is trivial when $\beta=0$ so we may suppose that $\beta>0$.
The implication $\Rightarrow$ follows immediately from \eqref{eq:FK_Gradient_Formula}. To deduce the implication $\Leftarrow$ from $\eqref{eq:FK_Gradient_Formula}$, it suffices to prove that 
\begin{itemize}\item the functions $\{e^{-\beta \sum_{e\in A} J_e \sigma_e} : A \subseteq E \text{ finite}\}$ have dense linear span in $C(\{-1,1\}^E,\| \cdot \|_\infty)$ for each $\beta>0$, and that
\item the functions $\{e^{-\beta \sum_{e\in A \cap E} J_e \sigma_e - \beta \sum_{v\in A \cap V} h\sigma_v } : A \subseteq E \cup V \text{ finite}\}$ have dense linear span in $C(\{-1,1\}^{E\cup V},\| \cdot \|_\infty)$ for each $\beta,h>0$. 
\end{itemize}
In both cases, the set $A$ is allowed to be empty. 
We prove the first claim, the proof of the second being similar. Fix $\beta>0$. Using the identity $e^{-\beta \sum_{e\in A} J_e \sigma_e} =\prod_{e\in A} (\cosh(\beta J_e)-\sigma_x \sinh(\beta J_e))$ one can prove by induction on $|B|$ that $\prod_{e\in B}  \sigma_e$ belongs to the linear span of $\{e^{-\beta \sum_{e\in A} J_e \sigma_e} : A \subseteq E \text{ is finite}\}$ for every finite (possibly empty) set $B \subseteq E$. The linear span of $\{\prod_{e\in B}  \sigma_e : B \subseteq E$ finite$\}$ is an algebra that separates points and is therefore dense in $C(\{-1,1\}^E,\| \cdot \|_\infty)$ by Stone-Weierstrass, concluding the proof.
\end{proof}

\begin{remark}
At this point, the reader may be wondering where the \emph{proof} of \cref{thm:main_continuity,thm:main_continuity_FK} breaks down in the case of free boundary conditions. Here is a short answer: The free version of \cref{cor:gradient} in zero external field does not allow connections through infinity. This means that one must also consider infinite clusters when applying the free version of this proposition to control the change in the edge marginals of the FK-Ising model as in \cref{cor:change_in_correlations}, and the free version of \cref{prop:finite_clusters_mismatched} does not suffice to do this. A similar problem also arises when attempting to apply the free version of \eqref{eq:ES_cylinder}.
\end{remark}

\subsection{The spectral radius of the free gradient Ising and random current models}
\label{subsec:free_spectral_radius}

We now note that, although the free gradient Ising measure $\bG^f_{\beta,0}$ is not known to be a factor of i.i.d.\ when $\beta > \beta_c$, it always has spectral radius at most that of the graph. Note that, in contrast, the free Ising measure $\bI^f_{\beta,0}$ on a $k$-regular tree has spectral radius strictly greater than that of the tree when $\beta$ is sufficiently large \cite{MR3603969}. 

\begin{thm}
\label{thm:free_spectral_radius}
Let $G=(V,E,J)$ be a connected transitive weighted graph and let $\Gamma$ be a closed unimodular transitive group of automorphisms. Then 
$\rho(\mathbf{G}^f_{\beta,h})$, $\rho(\mathbf{C}^f_{\beta,h})$, and $\rho(\mathbf{L}^f_{\beta,h})$ are all at most $\rho(G)$ 
for every $\beta,h\geq 0$.
\end{thm}

Note that when $G$ is a tree, the gradient Ising measure $\mathbf{G}^f_{\beta,0}$ is equivalent to Bernoulli bond percolation on $G$, so that this inequality is trivial.

\begin{proof}[Sketch of proof]
The claim is trivial when $\beta=0$, so suppose $\beta>0$. By \cref{thm:spectralradius}, it suffices to prove the claim for the gradient Ising measure $\mathbf{G}^f_{\beta,h}$. 
Let $X=(X_n)_{n\geq 0}$ be the random walk on $G$ and let $\hat X = (\hat X_n)_{n\geq 0}$ be the associated random walk on $\Gamma$ as defined in \cref{subsec:spectral_background}, and let $\sigma$ be a random variable with law $\bG^f_{\beta,h}$ that is independent of $\hat X$.
Recall from \eqref{eq:FK_Gradient_Formula} that
\begin{align}
\phi^f_{2,\beta,h}(\omega(x)=0 \text{ for all $x\in A$}) 
&=\bG^f_{\beta,h} \left[e^{\beta H_A} \right]\prod_{x \in A} \left( 1-\tanh(\beta J_x)\right)
\end{align}
for every finite set $A \subseteq E \cup V$, where we write $J_v=h$ for each $v\in V$ and recall that $K_A = \sum_{x\in A} J_x \sigma_x$. 
Thus, it follows by a similar analysis to the proof of \cref{thm:spectralradius} that if we set $F_A(\sigma)=e^{\beta H_A(\sigma)}$  for each finite $A \subseteq E \cup V$ (so that $F_\emptyset \equiv 1$) then
\begin{equation}
\limsup_{k\to\infty} \Cov\left(F_A(\sigma),F_A(\hat X_{2k}^{-1} \sigma) \right)^{1/2k} \\\leq \max\{\rho(\phi^f_{\beta,h}),\rho(G)\} = \rho(G),
\end{equation}
for every finite set $A \subseteq E \cup V$, 
where the final inequality follows from \cref{thm:Ising_factor}. The proof of \cref{prop:continuity_equivalence} implies that the set $\{F_A : A \subseteq E \cup V$ is finite$\}$ has dense linear span in $L^2(\bG^\#_{\beta,h})$, so that the claim follows from \eqref{eq:rho_limit_covariance}.
\end{proof}

This theorem raises the following natural question. See \cite{ray2019finitary,MR3603969,harel2018finitary} for related results. 

\begin{question}
Let $G=(V,E,J)$ be a connected transitive weighted graph. For what values of $\beta$, $h$, and $\#$ can $\mathbf{G}^\#_{\beta,h}$, $\mathbf{C}^\#_{\beta,h}$, and $\mathbf{L}^\#_{\beta,h}$ be expressed as factors of i.i.d.?
\end{question}

\subsection{Other graphs}

We remark that the methods used to prove \cref{thm:finite_clusters,thm:free_random_cluster} can easily be combined with the methods of \cite{HermonHutchcroftIntermediate} to prove the following theorem, which applies in particular to certain Cayley graphs of intermediate volume growth.

\begin{theorem}
Let $G=(V,E)$ be a connected, locally finite, unimodular transitive graph, and suppose that there exist constants $c>0$ and $\gamma > 1/2$ such that the return probabilities for simple random walk on $G$ satisfy $p_n(o,o) \leq e^{-cn^\gamma}$ for every $n\geq 1$. Then for each $q\geq 1$ the critical free random cluster measure $\phi^f_{q,\beta_c^f(q),0}$ is supported on configurations with no infinite clusters.
\end{theorem}

Applying the results of \cite{MR3306602}, we immediately deduce the following corollary.

\begin{corollary}
Let $G=(V,E)$ be an \emph{amenable}, connected, locally finite, unimodular transitive graph, and suppose that there exist constants $c>0$ and $\gamma > 1/2$ such that the return probabilities for simple random walk on $G$ satisfy $p_n(o,o) \leq e^{-cn^\gamma}$ for every $n\geq 1$. Then the Ising model on $G$ has a continuous phase transition in the sense that $m^*(\beta_c)=0$.
\end{corollary}

Our results naturally raise the following interesting problem.

\begin{problem}
Extend \cref{thm:main_continuity,thm:main_continuity_FK} to nonunimodular transitive graphs.
\end{problem}

 One approach to this problem, which may be very challenging, would be to attempt to extend the analysis of \cite{Hutchcroftnonunimodularperc} from Bernoulli percolation to the Ising model. This would have the added benefit of giving a very complete description of the Ising model at and near criticality on such graphs, or more generally on graphs with a nonunimodular transitive subgroup of automorphisms such as $T \times \Z$, going far beyond the conclusions of \cref{thm:main_continuity,thm:main_continuity_FK}.

\section*{Acknowledgments}

We thank Jonathan Hermon for making us aware of Freedman's work on maximal inequalities for martingales \cite{MR0380971}, which inspired \cref{lem:martingale_stuff,lem:martingale_stuff2}. We also thank Hugo Duminil-Copin, Geoffrey Grimmett, and Russ Lyons for helpful comments on an earlier version of the manuscript.

 \setstretch{1}
 \footnotesize{
  \bibliographystyle{abbrv}
  \bibliography{unimodularthesis.bib}
  }

\end{document}